\documentclass[a4paper, 11pt]{article}

\usepackage[utf8]{inputenc}
\usepackage{amsmath}
\usepackage{amsthm}
\usepackage{amsfonts}
\usepackage{amssymb}
\usepackage{amstext}
\usepackage{dsfont}
\usepackage{graphicx}
\usepackage{comment}
\usepackage{color}
\usepackage[colorlinks]{hyperref}
\usepackage{epigraph}
\usepackage[left=3cm,top=2cm,bottom=3cm,right=2.5cm,nohead,nofoot]{geometry}
\usepackage{tikz}
\usepackage{tkz-graph}
\usepackage{tkz-berge}
\usetikzlibrary{arrows,shapes,matrix}

\graphicspath{{Figures/}{.}}%

\title{ Whittaker processes and Landau-Ginzburg potentials for flag manifolds }
\author{Reda \textsc{Chhaibi}        \footnote{\texttt{reda.chhaibi@math.uzh.ch}} } 
\date{}

\allowdisplaybreaks[4]

\DeclareMathOperator{\hw}{hw}

\DeclareMathOperator{\wt}{wt}

\DeclareMathOperator{\ch}{ \textrm{ch} }
\DeclareMathOperator{\eqlaw}{\stackrel{\Lc}{=}}

\DeclareMathOperator{\Id}{Id}

\def\half{\frac{1}{2}}

\def\N{{\mathbb N}}
\def\Z{{\mathbb Z}}
\def\Q{{\mathbb Q}}
\def\R{{\mathbb R}}
\def\C{{\mathbb C}}

\def\P{{\mathbb P}}
\def\E{{\mathbb E}}

\def\Ac{{\mathcal A}}
\def\Bc{{\mathcal B}}
\def\Cc{{\mathcal C}}
\def\Dc{{\mathcal D}}

\def\Fc{{\mathcal F}}
\def\Gc{{\mathcal G}}

\def\Kc{{\mathbb K}}
\def\Lc{{\mathcal L}}
\def\Oc{{\mathcal O}}
\def\Pc{{\mathcal P}}

\def\Rc{{\mathcal R}}

\def\Tc{{\mathcal T}}

\def\afrak{{\mathfrak a}}
\def\Bfrak{{\mathfrak B}}
\def\bfrak{{\mathfrak b}}
\def\gfrak{{\mathfrak g}}
\def\hfrak{{\mathfrak h}}

\newtheorem{thm}{Theorem}[section]
\newtheorem{proposition}[thm]{Proposition}
\newtheorem{corollary}[thm]{Corollary}

\newtheorem{definition}[thm]{Definition}
\newtheorem{example}[thm]{Example}
\newtheorem{lemma}[thm]{Lemma}
\newtheorem{properties}[thm]{Properties}

\newtheorem{rmk}[thm]{Remark}
\newtheorem{notation}{Notation}[section]

\numberwithin{equation}{section}

\begin{document}
\maketitle

\begin{abstract}
The Robinson-Schensted (RS) correspondence and its variants naturally give rise to integrable dynamics of non-intersecting particle systems.

In previous work, the author exhibited a RS correspondence for geometric crystals by constructing a Littelmann path model, in general Lie type. Since this representation-theoretic map takes as input a continuous path on a Euclidian space, the natural starting measure is the Wiener measure. In this paper, we characterize the measures induced by Brownian motion through the RS map. 

On the one hand, the highest weight in the output is a remarkable Markov process (the Whittaker process), and can be interpreted as a weakly non-intersecting particle system which deforms Brownian motion in a Weyl chamber. One the other hand, the measure induced on geometric crystals is given by the Landau-Ginzburg potential for complete flag manifolds which appear in mirror symmetry. The measure deforms the uniform measure on string polytopes.

Whittaker functions, which appear as volumes of geometric crystals, play the role of characters in the theory.
\end{abstract}

\newpage
\tableofcontents

\newpage
\section{Introduction}
Let $\afrak \approx \R^r$ be a Euclidean space endowed with the action of a finite crystallographic reflection group $W$. The fundamental domain of action for $W$ is the Weyl chamber which we denote by $C$:
$$C := \left\{ x \in \afrak \ | \ \forall \alpha \in \Delta, \alpha(x) > 0 \right\}$$
where $\Delta$ are the simple roots. The example to have in mind is the type $A_{r-1}$ where the symmetric group on $r$ elements $S_r$ acts on $\afrak$ by permuting coordinates. In this case, the Weyl chamber is $C_{A_{r-1}} = \left\{ x \in \R^{r} | x_1 > x_2 > \dots > x_{r} \right\}$ and a Markov process in $C_{A_{r-1}}$ is interpreted as a non-intersecting particle system. Other Lie types give other cones which are interpreted as other non-intersection conditions.

Now, consider a sub-Markovian Brownian motion $X$ on $\afrak$ with infinitesimal generator:
\begin{align}
\label{eq:schrodinger_operator_V}
H = & \half \Delta - V(x) 
\end{align}
where $V: \afrak \rightarrow \R_{+} \bigsqcup \{ \infty \}$ is the killing potential. If $V(x) = \infty \mathds{1}_{ \left\{x \notin C \right\} }$, the underlying Brownian motion is killed immediately upon touching $\partial C$. By conditioning the process to survive, one obtains Brownian motion conditioned to remain in the Weyl chamber. We will call this case the strictly non-intersecting setting or the tropical setting for the following reason. From the works of Biane, Bougerol and O'Connell \cite{bib:BBO}, \cite{bib:BBO2}, Brownian motion in Weyl chambers appears through an algebraic construction as the Pitman transform of a standard Brownian motion. The Pitman transform is of representation-theoretic significance and involves the tropical semi-ring $\left( +, - , \min \right)$. A continuous RS correspondence is implicit in their work.

In this paper, we are concerned with algebraic constructions of non-intersecting particle systems coming from the geometric RS correspondence exhibited in \cite{bib:chh14a} by constructing a Littelmann path model for geometric crystals. The sub-Markovian process that appears has infinitesimal generator as in \eqref{eq:schrodinger_operator_V} with $V$ being the Toda potential:
$$ V(x) = \sum_{\alpha \in \Delta} \frac{\langle \alpha, \alpha \rangle}{2} e^{-\alpha(x)}.$$
Notice that the killing rate increases dramatically as the particles leave the Weyl chamber, while it is very low inside. Therefore, the process conditioned to survive is interpreted as a weakly non-intersecting particle system. We will see that the killing rate can be tuned in order to recover the non-intersecting setting.

\paragraph{On the strictly non-intersecting/tropical setting:}
As a general background, O'Connell considers in \cite{bib:OC03} a random walk as input for the usual RS correspondence, and proves that the shape has a Markovian evolution - a discrete version of Dyson's Brownian motion. In his construction, the Pieri rule can be interpreted as giving the transition probabilities for a random walk in the cone $C_{A_{r-1}}\left( \N \right) := \left\{ x \in \N^{r} | x_1 > x_2 > \dots > x_{r} \right\}$, killed upon hitting the boundary of $\partial C_{A_{r-1}}\left( \N \right)$.

Moving from a discrete to a continuous setting, recall that Dyson's Brownian motion is the process obtained from the spectrum $\left( \Lambda_t ; t \geq 0 \right)$ of a Hermitian Brownian motion, hence its importance in random matrix theory (\cite{bib:AGZ10} and references therein). It has the same distribution as Brownian motion conditioned in the sense of Doob to remain in the cone $C_{A_{r-1}}$, a very natural non-intersecting particle system. Moreover a classical fact is that the eigenvalues of principal minors interlace, forming a Gelfand-Tsetlin pattern. Not only that, the conditional distribution of these interlaced eigenvalues given $\{ \Lambda_t = \Lambda \}$ is the uniform measure on the Gelfand-Tsetlin polytope with top row $\Lambda$ \cite{bib:Baryshnikov}.

For a generalization to arbitrary Lie type, the starting point is the representation-theoretic reinterpretation of the RS correspondence in the language of crystals. We refer to the section 10 of \cite{bib:chh14a} for a zoology of RS correspondences. Let $\gfrak$ be a complex semi-simple Lie algebra and $\afrak \approx \R^r$ be the real part of the Cartan subalgebra. With probabilistic applications in mind, an extremely convenient way of encoding the combinatorics of representation theory of (the Langlands dual of) $\gfrak$ is the Littelmann path model \cite{bib:Littelmann95} where crystal elements are piece-wise linear paths in $\afrak$. Because concatenation of paths is a model for tensor products of crystals, the tensor product dynamic becomes a simple random walk. To such a simple walk up to a finite time, one can associate bijectively the pair made of the highest weight path and an abstract crystal element. This is the true nature of the RS correspondence. In the large sense, non-intersecting particle 
systems appear when observing the highest weight in a tensor product dynamic.

Let $T>0$ be a fixed time horizon. Continuous counterparts of the Littelmann path model were constructed in \cite{bib:BBO}, \cite{bib:BBO2} giving a crystal structure on $\Cc_0\left( [0, T], \afrak \right)$, the set of continuous paths starting at $0$. A major role is played by the Pitman transform:
$$ \Pc_{w_0}: \Cc_0\left( [0, T], \afrak \right) \rightarrow \Cc_0\left( [0, T], C \right). $$
It was understood that the highest weight of a Littelmann crystal generated by a path $\left(\pi_s\right)_{0 \leq s \leq t}$ is the endpoint of $\left( \Pc_{w_0} \pi \right)_{0 \leq s \leq t}$. One of their main results is that the Pitman transform $\Pc_{w_0}$ of a Brownian motion is a Brownian motion conditioned to remain in the Weyl chamber. We are interested in such algebraic constructions of non-intersecting particle systems, and we argue that such a miraculous Markov property is proof that Brownian motion in a Weyl chamber is canonical. Moreover (\cite{bib:BBO2}), the set of paths $\pi \in \Cc_0\left( [0, T], \afrak \right)$ with fixed Pitman transform $\eta$ and $\eta(T) = \lambda$ is exactly parametrized by a polytope $\Delta_{\bf i}(\lambda)$ called the string polytope. The polytope depends on a choice of reduced word ${\bf i}$ and the elements in such a polytope are called string coordinates. In particular for type $A_{r-1}$, the Pitman transform folds a standard Brownian motion in $\R^r$ into the 
cone $C_{A_{r-1}}$, giving Dyson's Brownian motion. Also, string polytopes can be picked to be a Gelfand-Tsetlin pattern \cite{bib:Littelmann98}, for specific choices of reduced words. We will recover these results as a degeneration of the Whittaker process and the Landau-Ginzburg potentials. 

\paragraph{Weakly non-intersecting/geometric setting:}
A geometric analogue of Dyson's Brownian motion, which we call the Whittaker process, has been exhibited by O'Connell in \cite{bib:OConnell}. From a particle system point of view, the Whittaker process in type $A_{r-1}$ can be thought of as a system of $r$ weakly non-intersecting particles - in the sense that repulsion is insufficient to force the process to remain in the cone $C_{A_{r-1}}$. Furthermore, random matrix behavior does remain. For example it is proved in theorem 1.6.1 \cite{bib:BorCor} that, as $n\rightarrow \infty$ the first particle properly rescaled converges in law to the famous Tracy-Widom distribution. O'Connell's construction relies crucially on Givental's integral formula for Whittaker functions \cite{bib:Givental}. Such a formula is relevant to mirror symmetry as the integrand is the Landau-Ginzburg potential for complete flag manifolds, which Givental related to the quantum cohomology ring of flag manifolds. From the point of view of probability theory, this potential serves as an 
analogue of the uniform distribution on Gelfand-Tsetlin patterns. 

In the present paper, we complete the work of \cite{bib:OConnell} by treating the problem for general Lie type and with a uniform approach. Moreover, our approach is designed to force the appearance of the Landau-Ginzburg potential rather than using it as an ingredient. The objects at hand will be properly defined in the preliminaries section. The first step was to realize that the underlying theory of crystals is Berenstein and Kazhdan's geometric crystals \cite{bib:BK00}, \cite{bib:BK04}. To that end, we developed a geometric Littelmann path model in the paper \cite{bib:chh14a} that lays the foundational material for this probabilistic work. It allows to realize Berenstein and Kazhdan's geometric crystals as paths in $\Cc_0\left( [0, T], \afrak \right)$. Such an algebraic structure on the set of Brownian paths is key in our analysis. It identifies the geometric Pitman transform:
$$\Tc_{w_0}: \Cc_0\left( [0, T], \afrak \right) \longrightarrow \Cc_0\left( (0, T], \afrak \right)$$
as the natural invariant under crystal actions: It gives the highest weight path in the \emph{geometric} crystal structure. Unlike its tropical counterpart, the geometric Pitman transform does not fold paths into the Weyl chamber $C$. Moreover, the set of paths $\pi \in \Cc_0\left( [0, T], \afrak \right)$ with fixed geometric Pitman transform $\eta$ and $\eta(T) = \lambda$ is naturally identified with a geometric object $\Bc\left( \lambda \right)$. The morphism of crystals allowing this identification is denoted by $p$. The union $\Bc := \bigsqcup_{ \lambda \in \afrak} \Bc(\lambda)$ is a subset of ``lower triangular'' matrices in a Lie group. In the end, we obtained as Theorem 10.4 in \cite{bib:chh14a} the geometric RS correspondence which is a map:
$$ \begin{array}{cccc}
  RS : & \Cc_0\left( [0, T], \afrak \right) & \longrightarrow & \left\{ (x, \eta) \in \Bc \times \Cc\left( (0,T], \afrak \right) \right\}\\
       &         \pi                        &   \mapsto       & \left( p\left( \pi \right), \left( \left( \Tc_{w_0} \pi \right) (t); 0 < t \leq T \right) \right)
   \end{array}.
 $$

The idea is to consider the measure induced on crystals by uniform paths in a path model. In the discrete Littelmann path model, taking the uniform probability measure on finite paths induces the uniform measure on the generated crystal and transition probabilities are expressed in terms of Littlewood-Richardson coefficients. This is particularly clear in \cite{bib:LLP}, where the authors revisit the discrete case using the Littelmann path model for finite type and generalize their construction in \cite{bib:LLP2} to the Kac-Moody setting. In the continuous setting of \cite{bib:BBO2}, the uniform measure on finite paths is replaced by the Wiener measure: a ``uniformly chosen path'' is nothing but Brownian motion. It is the essence of \cite{bib:BBO2}, and this idea proves to be even more fruitful in the geometric setting.

\paragraph{Results:} We consider a Brownian motion $W^{(\mu)}$ in $\afrak$ with drift $\mu \in \afrak$. Our main constributions are a description of the canonical measures induced by Brownian motion through the RS correspondence.
\begin{itemize}
 \item The process $\left( \Lambda_t^{(\mu)} = \Tc_{w_0}\left( W^{(\mu)} \right)_t ; \ 0 \leq s \leq t \right)$ is a Markov process, the Whittaker process for general Lie type (Theorem \ref{thm:highest_weight_is_markov}). Its infinitesimal generator is given by a Doob transform of the quantum Toda Hamiltonian using one of its eigenfunctions, a Whittaker function. In a sense, we prove that the quantum Toda Hamiltonian is the geometric infinitesimal version of the Littlewood-Richardson rule.
 
 \item We will examine the law of $p\left( \left(W_{s} \right)_{ 0 \leq s \leq t} \right)$ conditionally to its geometric Pitman transform being fixed and with endpoint $\lambda$. This probability measure on $\Bc(\lambda)$ will be the canonical measure on a geometric crystal appropriately normalized into a probability measure.
 
 \item The Archimedean Whittaker function will have two interpretations. From the point of view of probability theory, it gives the survival probability of a Brownian motion killed under the Toda potential. From the point of view of representation theory, it appears as the character of a geometric crystal. In a sense, it encodes its volume because it is the integral of the Landau-Ginzburg potential.
 
 \item The previous results tropicalize to the results of \cite{bib:BBO} and \cite{bib:BBO2}. More precisely, the Whittaker function degenerates to Kirillov's orbital integral and the Whittaker process degenerates to Brownian motion in the Weyl chamber. This is obtained via a simple time rescaling.
\end{itemize}

From a representation-theoretic perspective, we study the path crystal $\Bc_t = \langle \left(W_{s} \right)_{ 0 \leq s \leq t}\rangle$ as a random object. By analogy with the Young tableaux growth obtained in the classical RS correspondence \cite{bib:OC03}, one can think of $\Bc_t$ as a dynamical object growing with time. We start by preliminaries in section \ref{section:preliminaries} that define the mathematical objects we will need. We will allow ourselves to be less explicit in the definition of the reference measure in Landau-Ginzburg potentials, in order not to dwell on parametrizations of geometric crystals. Then we give precise statements of the results in section \ref{section:main_results}.

\section{Preliminaries}
\label{section:preliminaries}

\subsection{Classical Lie theory}

\paragraph{On the structure of Lie algebras (\cite{bib:Humphreys72}):}
Let $\left( \gfrak, \left[ \cdot, \cdot \right] \right)$ be a complex semi-simple Lie algebra of rank $r$. The Cartan subalgebra is a maximal abelian subalgebra denoted by $\hfrak \approx \C^r$. The set of roots can be written as a disjoint union of positive and negative roots $\Phi = \Phi^+ \bigsqcup \Phi^-$, with $\Phi^- := -\Phi^+$. This choice uniquely determines $\Delta = (\alpha_i)_{i \in I} \subset \Phi^+$ a simple system such that every positive root is a sum with positive integer coefficients of simple roots - and reciprocally, a simple system uniquely determines a positive system. Moreover, the simple system $\Delta$ forms a basis of $\hfrak^*$. When convenient, we will index simple roots by $I = \left\{ 1, 2, \dots, r \right\}$. The choice of a simple system $\Delta$ fixes an open Weyl chamber:
$$C := \left\{ x \in \afrak \ | \ \forall \alpha \in \Delta, \alpha(x) > 0 \right\}.$$
The Weyl co-vector is half the sum of positive co-roots:
\begin{align}
\label{eq:def_weyl_covector}
\rho^\vee := & \frac{1}{2} \sum_{ \beta \in \Phi^+ } \beta^\vee.
\end{align}

The Cartan subalgebra has a decomposition $\hfrak = \afrak + i \afrak$ with $\afrak$ chosen to be the real subspace of $\hfrak$ where roots are real valued. By Cartan's criterion, since $\gfrak$ is semi-simple, the Killing form is non-degenerate. Its restriction to $\hfrak$ is in fact a scalar product written $\langle \cdot, \cdot \rangle$. In the identification of $\hfrak$ and $\hfrak^*$ thanks to the Killing form, it is customary to write the coroot $\beta^\vee$ as $\beta^\vee = \frac{2 \beta}{\left( \beta, \beta \right)}$ for $\beta \in \Phi$.

For each positive root $\beta \in \Phi^+$, we can choose an $\mathfrak{sl}_2$-triplet $(e_\beta, f_\beta, h_\beta = \beta^\vee) \in \gfrak_\beta \times \gfrak_{-\beta} \times \hfrak$ such that $[e_\beta, f_\beta] = h_\beta$. These determine a pinning given by a family of Lie algebra homomorphisms $\phi_\beta: \mathfrak{sl}_2 \rightarrow \gfrak$. $(e_\alpha, f_\alpha, h_\alpha)_{\alpha \in \Delta}$ is the set of simple $\mathfrak{sl}_2$-triplets, also known as Chevalley generators. 

\paragraph{On the structure of Lie groups (\cite{bib:Springer09}):}
Let $G$ be a simply-connected complex semi-simple Lie group with Lie algebra $\gfrak$.  The Langlands dual $G^\vee$ is the adjoint semi-simple complex Lie group with Lie algebra $\gfrak^\vee$. The maximal torus with Lie algebra $\hfrak$ is denoted by $H$. Our choice of positive roots yields a choice of opposite Borel subgroups $B$ and $B^+$, whose unipotent radicals are respectively $N$ and $U$. $N$ and $U$ are the lower and upper unipotent subgroups, while $B$ and $B^+$ are the lower and upper Borel subgroups. The exponential map is denoted by $\exp: \gfrak \rightarrow G$. It lifts the homomorphisms $\left( \phi_\beta \right)_{\beta \in \Phi^+}$ from the Lie algebra $\gfrak$ to the group $G$: each $\phi_\beta$ gives rise at the group level to a Lie group homomorphism that embed $SL_2$ in $G$ and that will be denoted in the same way. The following notations are common for $t \in \C$:
\begin{align*}
 t^{h_\beta} & = e^{ \log(t) h_\beta} = \phi_\beta\left(  \left( \begin{array}{cc} t & 0 \\ 0 & t^{-1} \end{array}\right) \right), t \neq 0 \\
 x_\beta(t) & = e^{ t e_\beta} = \phi_\beta\left(  \left( \begin{array}{cc} 1 & t \\ 0 & 1 \end{array}\right) \right)\\
 y_\beta(t) & = e^{ t f_\beta} = \phi_\beta\left(  \left( \begin{array}{cc} 1 & 0 \\ t & 0 \end{array}\right) \right)\\
 x_{-\beta}(t) & = y_\beta(t) t^{-h_\beta} = \phi_\beta\left(  \left( \begin{array}{cc} t^{-1} & 0 \\ 1 & t \end{array}\right) \right)\\
\end{align*}

The Bruhat decomposition states that $G$ is the disjoint union of cells:
$$ G = \bigsqcup_{\omega \in W} B^+ \omega B^+  = \bigsqcup_{\tau \in W} B \tau B^+ $$
In the largest opposite Bruhat cell $B B^+ = N H U$, every element $g$ admits a unique Gauss decomposition in the form $g = n a u$ with $ n \in N$, $a \in H$, $u \in U$. In the sequel, we will write $g = [g]_- [g]_0 [g]_+$, $[g]_- \in N$, $[g]_0 \in H$ and $[g]_+ \in U$ for the Gauss decomposition. Also $[g]_{-0} := [g]_- [g]_0$ and $[g]_{0+} := [g]_0 [g]_+$.

\paragraph{Weyl group:} 
The Weyl group of $G$ is defined as $W := \textrm{Norm}_G( H )/ H $, $\textrm{Norm}_G( H )$ being the normalizer of $H$ in $G$. It acts on the torus $H$ by conjugation and hence on $\hfrak$. To every linear form $\beta \in \hfrak^*$, define the associated reflection $s_\beta$ on $\hfrak$ with:
$$\forall \lambda \in \hfrak,  s_\beta \lambda := \lambda - \beta\left( \lambda \right) \beta^\vee. $$
The reflections $(s_\alpha)_{\alpha \in \Delta}$ are called simple reflections and they generate a finite Coxeter group isomorphic to $W$. For $w \in W$, a reduced expression is given by writing $w$ as product of simple reflections with minimal length:
$$ w = s_{i_1} s_{i_2} \dots s_{i_{\ell}}.$$
A reduced word is such a tuple ${\bf i} = \left( i_1, \dots, i_{\ell} \right)$ and the set of reduced words for $w \in W$ is denoted by $R(w)$. Since all reduced expressions have necessarily the same length, it defines unambiguously the length function $\ell: W \rightarrow \N$. The unique longest element is denoted by $w_0$ and in all the following, we will use the integer $m=\ell(w_0)$.

A common set of representatives in $G$ for the generating reflections $(s_i)_{i \in I} \subset W$ is:
$$ \bar{s}_i := \phi_i\left(  \left( \begin{array}{cc} 0 & -1 \\ 1 & 0 \end{array}\right) \right) = e^{-e_i} e^{f_i} e^{-e_i} = e^{f_i} e^{-e_i} e^{f_i}.$$
Another common choice is:
$$ \bar{\bar{s}}_i := \bar{s}_i^{-1} = \phi_i\left(  \left( \begin{array}{cc} 0 & 1 \\ -1 & 0 \end{array}\right) \right) = e^{e_i} e^{-f_i} e^{e_i} = e^{-f_i} e^{e_i} e^{-f_i}.$$

Following lemma 2.3 in \cite{bib:KacPeterson}, these Weyl group representatives satisfy the braid relations, which allows to define unambiguously $\bar{w} = \bar{u} \bar{v}$ if $w = uv$ and $\ell(w) = \ell(u) + \ell(v)$. However they do not form a presentation of the Weyl group, since for example $ (\bar{s}_i)^2 = \phi_i(-id) \neq id $. The representative $\bar{w}_0$ of the longest element $w_0$ will be very often used.

\begin{rmk}
\label{rmk:sln_example}
The reader unfamiliar with Lie groups can have in mind the example of $SL_n(\C)$, of rank $r = n-1$. The following matrices can be chosen as Chevalley generators. If $E_{i,j} = \left( \delta_{i,r} \delta_{j,s} \right)_{1 \leq r, s \leq n}$ are the usual elementary matrices, then $h_i = E_{i,i} - E_{i+1,i+1}$, $e_i = E_{i,i+1}$, and $f_i = E_{i+1,i}$. $\hfrak$ is the set of complex diagonal matrices with zero trace, which we identify with $\left\{ x \in \C^n \ | \ \sum x_i = 0 \right\}$. Then $H$ is the set of diagonal matrices with determinant $1$. $N$ (resp. $U$) is the set of lower (resp. upper) triangular unipotent matrices. We have:
$$ \forall t \in \C, y_i(t) = \textrm{Id} + t E_{i+1,i}, \ x_i(t) = \textrm{Id} + t E_{i,i+1}.$$

The Weyl group $W$ is the group of permutation matrices and acts on $\hfrak$ by permuting coordinates. In the identification with the symmetric group acting on $n$ elements, the reflections $s_i$ are identified with transpositions $\left( i  \ i+1 \right)$. The longest word $w_0$ reorders the elements $1, 2, \dots, n$ in decreasing order.

In the case of $GL_n$, the second Bruhat decomposition is known in linear algebra as the LPU decomposition which states that every invertible matrix can be decomposed into the product of a lower triangular matrix $L$, a permutation matrix $P$ and an upper triangular matrix $U$. $P$ is unique. The largest opposite Bruhat cell corresponds to $P = id$. It is dense as it is the locus where all principal minors are non-zero.
\end{rmk}

\paragraph{Involutions}:
Since $w_0 \in W$ transforms all simple positive roots to simple negative roots, there is an involution on $\Delta$ (or equivalently the index set $I$) denoted by $*$ such that:
$$ \forall \alpha \in \Delta, \beta^* = -w_0 \alpha .$$

We define the following group antimorphisms of $G$ by their actions on a torus element $a \in H$ and the one-parameters subgroups generated by the Chevalley generators. For convenience, we also give their action at the level of the Lie algebra $\gfrak$.
\begin{itemize}
 \item The transpose:
       $$ \begin{array}{ccc}
           a^T = a          & x_i(t)^T     = y_i(t) & y_i(t)^T     = x_i(t)
          \end{array}
       $$
       $$ \forall \alpha \in \Delta, 
          \begin{array}{ccc}
           h_\alpha^T = h_\alpha &  e_\alpha^T = f_\alpha & f_\alpha^T = e_\alpha
          \end{array}
       $$
 \item The positive inverse or Kashiwara involution (\cite{bib:Kashiwara91} (1.3)):
       $$ \begin{array}{ccc}
           a^\iota = a^{-1} & x_i(t)^\iota = x_i(t) & y_i(t)^\iota = y_i(t)
          \end{array}
       $$
       $$ \forall \alpha \in \Delta, 
          \begin{array}{ccc}
           h_\alpha^\iota = -h_\alpha &  e_\alpha^\iota = e_\alpha & f_\alpha^\iota = f_\alpha
          \end{array}
       $$
 \item Sch\"utzenberger involution:
       $$S(x) = \bar{w}_0 \left( x^{-1} \right)^{\iota T} \bar{w}_0^{-1} = \bar{w}_0^{-1} \left( x^{-1} \right)^{\iota T} \bar{w}_0$$
       It acts as ( relation 6.4 in \cite{bib:BZ01}):
       $$ S\left( x_{i_1}(t_1) \dots x_{i_\ell}(t_\ell) \right) = x_{i_\ell^*}(t_\ell) \dots x_{i_1^*}(t_1)$$
       $$ \forall \alpha \in \Delta, 
          \begin{array}{ccc}
           S(h_\alpha) = h_{\alpha^*} &  S(e_\alpha) = e_{\alpha*} & S(f_\alpha) = f_{\alpha*} .
          \end{array}
       $$
       Notice that $S = \iota \circ S \circ \iota$.
\end{itemize}

\subsection{Probability theory and processes on Lie groups}
The usual framework is the probability triplet $( \Omega, \Ac, \P)$ with $\Omega$ the sample space, $\Ac$ the set of events and $\P$ our working probability measure. Equality in law between random variables or processes will be denoted by $\stackrel{\mathcal{L}}{=}$.

For every process $X$ (or even a deterministic path) taking its values in a Euclidian space, we will use the notation $X^{(\mu)}_t := X_t + \mu t$. Also, we use $X^{x_0} := X + x_0$. Unless otherwise stated, the absence of superscript will indicate that the process is starting at zero. The filtration generated by $X$ is denoted by:
$$ \Fc_t^X := \sigma\left( X_s ; s \leq t \right)$$

Since $\afrak$ is made into an Euclidean space thanks to the Killing form $\langle \cdot, \cdot \rangle$, there is a natural notion of Brownian motion on $\afrak$. Indeed, $\langle \cdot, \cdot \rangle$ is a scalar product once restricted to $\afrak$ and the induced norm is denoted by $||\cdot||$. Consider a Brownian motion with drift $\mu$ and starting position $x_0 \in \afrak$:
$$ X_t^{(x_0, \mu)} := x_0 + X_t + \mu \ t \ .$$

One obtains a left-invariant Brownian motion $B_t(X^{(\mu)})$ on $B$ by solving the following stochastic differential equation driven by any semimartingale $X$ starting at $0$. The symbol $\circ$ indicates that the SDE has to be understood in the Stratonovitch sense:
\begin{align}
\label{lbl:process_B_sde_stratonovich}
\left\{ \begin{array}{ll}
dB_t(X^{(\mu)}) = B_t(X^{(\mu)}) \circ \left( \sum_{\alpha \in \Delta} \half \langle \alpha, \alpha \rangle f_\alpha dt + dX^{(\mu)}_t\right) \\
 B_0(X^{(\mu)}) = \textrm{ Id }
\end{array} \right.
\end{align}

This SDE was first introduced in \cite{bib:BBO} and its flow is intimately related to the geometric RS correspondence, as we will see in the next subsection. It can be explicitly solved:
\begin{align}
\label{eq:process_B_explicit}
B_t(X^{(\mu)}) & = 
\left( \sum_{ \substack{k \geq 0 \\ i_1, \dots, i_k } } f_{i_1} \dots f_{i_k} \int_{ t \geq t_k \geq \dots \geq t_1 \geq 0}  \prod_{j=1}^k dt_j \frac{|| \alpha_{i_j} ||^2}{2} e^{ -\alpha_{i_j}(X^{(\mu)}_{t_j}) } \right) e^{X^{(\mu)}_t}
\end{align} 

Notice that the flow makes sense for any deterministic continuous path and we take eq. \eqref{eq:process_B_explicit} as a definition if $X$ is not starting at $0$. Clearly, $B_t(X^{(\mu)})$ has an $NA$ decomposition is given by:
\begin{align}
\label{eq:process_N_explicit}
N_t(X^{(\mu)}) & = 
\sum_{ \substack{k \geq 0 \\ i_1, \dots, i_k } } \int_{ t \geq t_k \geq \dots \geq t_1 \geq 0} e^{ -\alpha_{i_1}(X^{(\mu)}_{t_1}) \dots -\alpha_{i_k}(X^{(\mu)}_{t_k}) } \frac{|| \alpha_{i_1} ||^2}{2} \dots \\
               & \quad \quad
\dots \frac{|| \alpha_{i_k} ||^2}{2} f_{i_1} \dots f_{i_k} dt_1 \dots dt_k
\end{align} 
\begin{align}
\label{eq:process_A_explicit}
A_t\left( X^{(\mu)} \right) & = e^{X^{(\mu)}_t}
\end{align} 

\begin{rmk}
\label{rmk:ADE}
This is a slightly modified version of the flow defined in subsection 5.2 of \cite{bib:chh14a} and the results we will invoke carry verbatim. The presence of half squared norms $\frac{|| \alpha ||^2}{2}$ is here to account for the fact that time does not flow in the same fashion for all roots and probability measures have nicer expressions in this setting. In the simply-laced groups ($ADE$-types), one can choose all roots to be the same length, hence choosing $\frac{|| \alpha ||^2}{2} = 1$ for all $\alpha$.
\end{rmk}

\begin{rmk}
In the $A_1$ case only, we opt out of the normalization made in the previous remark. Indeed, the classical choice for the only root is $\alpha = 2$. Hence the factor $\frac{|| \alpha ||^2}{2} = 2$ in the following example of $SL_2$.
\end{rmk}

\begin{example}[$SL_2$ - $A_1$ type]
Let $X^{(\mu)}$ be a Brownian motion with drift $\mu$ on $\R$. The SDE is:
$$ dB_t(X^{(\mu)}) = B_t(X^{(\mu)}) \circ \begin{pmatrix} dX^{(\mu)}_t & 0\\ 2 dt & -dX^{(\mu)}_t \end{pmatrix}$$
and its solution is:
$$ B_t(X^{(\mu)}) = \begin{pmatrix} e^{X^{(\mu)}_t} & 0\\ e^{X^{(\mu)}_t} \int_0^t 2 e^{-2 X^{(\mu)}_s} \ ds & e^{-X^{(\mu)}_t} \end{pmatrix}$$
\end{example}

\begin{example}[$SL_n$ - $A_{n-1}$ type]
As in remark \ref{rmk:sln_example}, let $X$ be a Brownian motion with drift $\mu$ on $\left\{ x \in \R^n \ | \ \sum x_i = 0 \right\}$. For notational reasons, we drop the superscript $(\mu)$ and put indices as exponents. The SDE becomes:
$$ dB_t(X^{(\mu)}) = B_t(X^{(\mu)}) \circ \begin{pmatrix} dX^1_t & 0      & 0      & \cdots     & 0     \\
                                                          dt     & dX^2_t & 0      & \ddots     & \vdots\\
                                                          0      & dt     & \ddots & \ddots     & 0     \\
                                                          \vdots & \ddots & \ddots & dX^{n-1}_t & 0     \\
                                                          0      & \cdots & 0      & dt         & dX^n_t\\
                                           \end{pmatrix}$$
and its solution $B_t(X^{(\mu)})$ is given by:
$$ 
		    \begin{pmatrix} e^{X^1_t}                             & 0                                     & 0      & \cdots        \\
				    e^{X^1_t} \int_0^t e^{X^2_s-X^1_s} ds & e^{X^2_t}                             & 0      & \cdots        \\
				    e^{X^1_t} \int_0^t e^{X^2_{s_1}-X^1_{s_1}} ds_1 \int_0^{s_1} e^{X^2_{s_2}-X^1_{s_2}} ds_2 &
				    e^{X^2_t} \int_0^t e^{X^3_s-X^2_s} ds &
				    e^{X^3_t}                             & \ddots                 \\
				    \vdots                                & \vdots                                & \ddots & \ddots  
		      \end{pmatrix}$$
\end{example}

Let $B\left( \R \right)$ be the set of real points of the Borel subgroup $B$ and $\bfrak\left(\R\right) = T_e B\left(\R\right)$ its Lie algebra. If $\Cc^\infty\left( B\left( \R \right) \right)$ is the space of continuously differentiable functions on the real solvable group $B\left( \R \right)$, every $X \in \bfrak\left( \R \right)$ can be seen as a left-invariant vector field on $B\left( \R \right)$. It acts as a derivation on $\Cc^\infty\left( B\left( \R \right) \right)$ via the Lie derivative:
$$ \forall f \in \Cc^{\infty}(B\left( \R \right)), \forall b \in B\left( \R \right), \Lc_X f(b) = \frac{d}{dt}\left( f(g e^{t X}) \right)_{|t=0}$$
When convenient, it is customary to write $X$ instead of $\Lc_X$. For example, the Laplace operator on $A$ is defined using the orthonormal basis $\left(V_1, \dots, V_n\right)$ by:
$$ \Delta_\afrak :=  \sum_{i=1}^n V_i^2 $$

\begin{definition}
\label{def:hypoelliptic_operator}
\index{$\Dc^{(\mu)}$: Infinitesimal generator of a hypoelliptic Brownian motion on $B$}
For $\mu \in \afrak$, define $\Dc^{(\mu)}$ to be the left-invariant differential operator on the solvable group $B$ given by:
$$ \Dc^{(\mu)} := \half \Delta_\afrak + \langle \mu, \nabla \rangle + \half \sum_{\alpha \in \Delta} \langle \alpha, \alpha \rangle \Lc_{f_\alpha}$$
where $\langle \mu, \nabla \rangle$ stands for the first order differential $\Lc_\mu$.
\end{definition}

Such an operator is exactly the Casimir element in Kostant's Whittaker model, as explained in the appendix of \cite{bib:chh14a}. We adopted this notation so that $\half \Delta_\afrak + \langle \mu, \nabla \rangle$ is the infinitesimal generator of (multiplicative) Brownian motion with drift $\mu$ on $A$. The following proposition is obtained by applying theorem 1.2, chapter 5, in Ikeda and Watanabe \cite{bib:IkedaWatanabe89}, which is a standard reference for stochastic analysis on manifolds. For more details, we refer to \cite{bib:thesis} where we explain how to proceed using only linear algebra, differential calculus and Euclidian Brownian motion.

\begin{proposition}
\label{proposition:inf_generator}
The operator $\Dc^{(\mu)}$ is the infinitesimal generator of the hypoelliptic Brownian motion $\left( B_t(W^{(\mu)}); t\geq 0 \right)$ driven by $W^{(\mu)}$, a Euclidian Brownian motion on $\afrak$ with drift $\mu$.
\end{proposition}

\begin{rmk}
Let
$$ V_0 = \mu + \half \sum_{\alpha \in \Delta} \langle \alpha, \alpha \rangle f_\alpha$$
This allows us to write the infinitesimal generator in the usual ``sums of squares'' form:
$$ \Dc^{(\mu)} = \half \sum_{i=1}^n V_i^2 + V_0$$
We refer to $\left( B_t\left( W^{(\mu)} \right) ; t \geq 0 \right)$ as our hypoelliptic Brownian motion on $B$ driven by $W^{(\mu)}$, because its infinitesimal generator satisfies the (parabolic) H\"ormander condition \cite{bib:Hormander}: When taking the vector fields $\left( V_1, \dots, V_n \right)$ and their iterated Lie brackets with the family $\left\{ V_0, V_1, \dots, V_n \right\}$, one generates indeed all of $\bfrak\left( \R \right) = T_e B\left( \R \right)$.

Although we will not make use of this fact, it is reassuring to know that it has a smooth transition kernel.
\end{rmk}

\subsection{Geometric crystals}
\label{subsection:geom_crystals}

A geometric crystal in the sense of Berenstein and Kazhdan is a topological set $L$ endowed with structural maps and crystal actions $\left( e^c_\alpha; c \in \R \right)_{\alpha \in \Delta}$ (Definition 3.1 in \cite{bib:chh14a}). Among these structural maps, there is a weight map $\wt: L \rightarrow \afrak$. Thanks to the crystal actions, one defines an action of Weyl group $W$ for which the weight map is equivariant:
$$ \forall x \in L, \forall w \in W, \ \wt( w \cdot x ) = w \wt(x) $$
The notation $\langle \pi \rangle$ is for the connected component generated by $\pi \in L$. This is not to be confused with the topological notion of connectedness: Two elements are in the same component if they are related by successive applications of crystal actions. The following notations are extracted from \cite{bib:chh14a}, where these objects are presented and studied in more detail.

\paragraph{The group picture:} Define the geometric Lusztig variety and the geometric Kashiwara variety as:
$$ U_{>0}^{w_0} := U              \cap B w_0 B \cap G_{\geq 0}$$ 
$$ C_{>0}^{w_0} := U \bar{w_0} U  \cap B       \cap G_{\geq 0}$$ 
Each one of these varieties possesses parametrizations indexed by reduced words ${\bf i} \in R(w_0)$ which are denoted respectively by $x_{\mathbf{i}}: \R_{>0}^{\ell(w_0)} \rightarrow U_{>0}^{w_0}$ and $x_{-\mathbf{i}}: \R_{>0}^{\ell(w_0)} \rightarrow C_{>0}^{w_0}$.

\begin{definition}[Geometric crystals]
\label{def:geom_crystal}
Define the geometric crystal of highest weight $\lambda \in \afrak$ as the set:
$$\Bc\left(\lambda\right) := C_{>0}^{w_0} e^{\lambda}$$
The union of all highest weight crystals will be denoted by $\Bc$, which is nothing but the set of totally positive elements in $B$:
$$ \Bc := \bigsqcup_{ \lambda \in \afrak} \Bc(\lambda) = B_{\geq 0}$$
\end{definition}

The weight map $\wt: \Bc \rightarrow \afrak$ and the highest weight map $\hw: \Bc \rightarrow \afrak$ are defined for $b \in \Bc$ as:
\begin{align}
\label{eq:def_wt}
\wt(b) = & \log [b]_0
\end{align}

\begin{align}
\label{eq:def_hw}
\hw(b) = & \log [ \bar{w}_0^{-1} b]_0
\end{align}

\paragraph{Path model:} The Littelmann path model developed in \cite{bib:chh14a} details a geometric crystal structure on the set of paths in $\Cc_0\left( [0,T], \afrak \right)$. In this case, the weight of $\pi$ is its endpoint $\pi(T)$. The fundamental invariant under crystal actions is the highest weight path and is given by the geometric Pitman transform $\Tc_{w_0}$, which is part of the following larger family.
\begin{definition}
\label{def:pitman_transform}
For every $w \in W$, define an associated the path transform $\Tc_{w}: \Cc\left( [0,T], \afrak\right) \rightarrow \Cc\left( (0,T], \afrak\right)$ from either of the following. Both definition are equivalent:
\begin{itemize}
 \item For $t>0$:
       $$ \left(\Tc_{w} \pi\right)(t) := \log \left[ \bar{w}^{-1} B_t(\pi) \right]_0$$
 \item For every reduced word ${ \bf i } \in R(w)$ such that $\ell = \ell(w)$:
       $$ \Tc_{w} := \Tc_{\alpha_{i_1}} \circ \dots \Tc_{\alpha_{i_\ell}}$$
       where
       $$ \Tc_{\alpha}\left( \pi \right)(t) := \pi(t) + \log\left[ \int_0^t e^{-\alpha\left( \pi(s) \right)} \frac{\| \alpha \|^2}{2} ds \right] \alpha^\vee $$
\end{itemize}
\end{definition}
From \cite{bib:chh14a} (Section 9, Isomorphism results), the connected component $\langle \pi \rangle$ can be entirely recovered from the knowledge of $\left( \Tc_{w_0} \pi \right)_{0 < t \leq T}$ and its isomorphism class depends only on $\lambda = \hw(\pi) = \Tc_{w_0} \pi(T)$. Moreover, if $h>0$, then as a consequence of the Laplace method:
$$ h \Tc_{\alpha}\left( h^{-1}\pi \right)(t) \stackrel{h \rightarrow 0}{\longrightarrow} \Pc_{\alpha}\left( \pi \right)(t) := \pi(t) - \inf_{0 \leq s \leq t} \alpha\left( \pi(s) \right) \alpha^\vee $$
so that:
\begin{align}
\label{eq:limit_pitman}
\Pc_{w} & = \lim_{h \rightarrow \infty} h \Tc_{w} h^{-1}
\end{align}
are the Pitman transforms introduced in \cite{bib:BBO}. Also, there is a morphism of geometric crystals $p: \Cc_0\left( [0,T], \afrak \right) \rightarrow \Bc$ that is given by the flow in eq. \eqref{eq:process_B_explicit}. If $\pi \in \Cc_0\left( [0,T], \afrak \right)$, then $p(\pi) = B_T(\pi)$. It projects the geometric path model onto the group picture. Thus geometric RS correspondence takes the form:

\begin{thm}[Geometric Robinson-Schensted correspondence]
\label{thm:dynamic_rs_correspondence}
For each $T>0$, we have a bijection:
$$ \begin{array}{cccc}
  RS : & \Cc_0\left( [0, T], \afrak \right) & \longrightarrow & \left\{ (x, \eta) \in \Bc \times \Cc^{high}_{w_0}\left( (0,T], \afrak \right) \ | \ \hw(x) = \eta(T) \right\}\\
       &         \pi                      &   \mapsto         & \left( B_T\left( \pi \right), \left(\Tc_{w_0} \pi_t; 0 < t \leq T \right) \right)
   \end{array}
 $$
where $\Cc^{high}_{w_0}\left( (0,T], \afrak \right)$ are the paths in $\afrak$ with a certain asymptotic behavior at time zero.
\end{thm}

\subsection{Landau-Ginzburg potentials}
\label{subsection:LG_potentials}
A Landau-Ginzburg potential is a measure on $\Bc\left( \lambda \right)$ of the form:
\begin{align}
\label{def:LG_potentials} 
 & e^{-f_B(x) } \omega(dx)
\end{align}
where $\omega$ is a reference measure of the form $\prod_{j=1}^m \frac{dt_j}{t_j}$ in the appropriate choice of coordinates, and $f_B$ is the following superpotential map. It uses the standard character $\chi: U \rightarrow \C$ defined as:
\begin{align}
\label{eq:def_st_character}
\chi   := & \sum_{\alpha \in \Delta} \chi_\alpha
\end{align}
where $\chi_\alpha: U \rightarrow \C$ are the elementary additive unipotent characters given by:
$$ \forall t \in \C, \forall \alpha, \beta \in \Delta, \chi_\alpha( e^{t e_\beta} ) = t \delta_{\alpha, \beta}.$$

\begin{definition}
\label{def:f_B}
Define the superpotential $f_B$ on $\Bc$ as the map:
$$
\begin{array}{cccc}
f_B: &          \Bc           & \rightarrow & \R_{>0} \\
     &   x = z \bar{w_0} t u  & \mapsto     & \chi(z) + \chi(u)
\end{array}
$$
where in the decomposition $x = z \bar{w_0} t u$, we have $z \in U^{w_0}_{>0}$, $u \in U^{w_0}_{>0}$ and $t \in A$.
\end{definition}

\begin{properties}
\label{properties:superpotential_properties}
For $x \in \Bc$:
$$ f_B\left( x \right) = f_B\left( x^\iota \right)
                       = f_B \circ S \left( x \right)$$
$$ \forall w \in W, f_B\left( w \cdot x \right) = f_B\left( x \right)$$
\end{properties}
\begin{proof}
The first relations are immediate from the definition \ref{def:f_B} of $f_B$ and that of the involutions $\iota$ and $S$. The invariance under the Weyl group action on the crystal can be checked easily on the reflections $\left( s_\alpha, \alpha \in \Delta \right)$ and has been done in \cite{bib:BK00}.
\end{proof}

\begin{example}
\label{lbl:example_superpotential_rank1}
In rank one case, by writing for $x \in \Bc(\lambda)$:
$$ x = \begin{pmatrix} t & 0 \\ 1 & t^{-1} \end{pmatrix} e^{\lambda \alpha^\vee}
     = \begin{pmatrix} 1 & t  \\ 0 & 1 \end{pmatrix}
       \begin{pmatrix} 0 & -1 \\ 1 & 0 \end{pmatrix}
       \begin{pmatrix} e^{\lambda} & 0 \\ 0 & e^{-\lambda} \end{pmatrix}
       \begin{pmatrix} 1 & \frac{e^{-2\lambda}}{t} \\ 0 & 1 \end{pmatrix}
       $$
We find the superpotential for $G/B \approx P^1\left( \C \right)$ (see \cite{bib:Rietsch11}):
$$f_B\left(x\right) = \chi\left( \begin{pmatrix} 1 & t  \\ 0 & 1 \end{pmatrix} \right) + 
                      \chi\left( \begin{pmatrix} 1 & \frac{e^{-2\lambda}}{t} \\ 0 & 1 \end{pmatrix} \right)
                    = t + \frac{e^{-2 \lambda}}{t}$$
\end{example}

These seemingly simple and innocent ingredients have appeared in two quite different circumstances.
\begin{itemize}
 \item(Geometric crystals) Berenstein and Kazhdan use the map $f_B$ to ``cut'' a discrete free crystal $\Bfrak^{free}(\lambda)$ obtained by tropicalizing $\Bc(\lambda)$ and its structural maps. Then they obtain normal Kashiwara crystals by setting (\cite{bib:BK06}):
  \begin{align}
  \label{eq:BK_pruning}
  \Bfrak\left( \lambda \right) & = \left\{b \in \Bfrak^{free}(\lambda) \ | \ [f_B]_{trop}(b) \geq 0 \right\} 
  \end{align}

  The surprise lies in the fact that a simple function like $f_B$ encodes exactly the string cones. As stated in the introduction of \cite{bib:BK04}, exhibiting such a function, along with its properties, proves a corollary of the local Langlands conjectures. More will be said on the tropicalization procedure when needed in section \ref{section:degenerations}.

 \item(Mirror symmetry) For a smooth projective Fano manifold $X$, one can form the small quantum cohomology ring. Mirror symmetry predicts the existence of a complex manifold $\Rc$, the ``mirror'', along with a Landau-Ginzburg potential $e^{-W(x)} \omega(dx)$ such that the quantum cohomology ring of $X$ is obtained as the Jacobi ring of $W$ over the mirror $\Rc$. The superpotential $W$ must depend smoothly on quantum parameters in $H^2(X, \C)$. For $X = \P^1(\C) \approx G/B$, with $G=SL_2$, one finds that $\Rc = \C^*$ and:
$$ W(t) = t + \frac{q}{t}$$
where $q \in H^2(X, \C) \approx \C$. This is exactly $f_B(x)$ if $x \in \Bc(\lambda)$ has Lusztig parameter $t$ and $q=e^{-2\lambda}$ (see example \ref{lbl:example_superpotential_rank1} ). More generally for all flag manifolds $X = G/B$, $H^2(X, \C) \approx \hfrak$. In this case, Rietsch \cite{bib:Rietsch07} has provided a construction of the mirror $\Rc$ and has proved that the superpotential is indeed the one in definition \ref{def:f_B}.

We have no explanation as to why mirror symmetry appears while investigating geometric crystals with probabilistic considerations and we will not try to push into that direction. However, the strength of our approach is that the Landau-Ginzburg potential $e^{-f_B(x)} \omega(dx)$ appear as a canonical measure induced by Brownian motion. It does come out as the result of a computation.
\end{itemize}

\section{Main results}
\label{section:main_results}

Having in mind the geometric RS correspondence, it is natural to look for a description of the highest weight process $\left( \Tc_{w_0}\left( W \right)_s ; s \leq t \right)$ and the distribution of the random crystal element $B_t(W^{(\mu)})$ conditionnally to the highest weight being fixed. Both aspects in this description have known analogues in the classical case of Young tableaux (see O'Connell \cite{bib:OC03}): the dynamic of the $Q$ tableau, or equivalently the shape, is Markovian and the distribution of the $P$ tableau conditionnally to the shape being fixed is the uniform measure on semi-standard tableaux. Our two main theorems are the geometric analogues.

\subsection{Whittaker process as the highest weight process}

The Whittaker process is obtained via the following algebraic construction as the geometric Pitman transform of Brownian motion, which is also the second component of the RS correspondence \ref{thm:dynamic_rs_correspondence} with the Wiener measure as input.
\begin{thm}
\label{thm:highest_weight_is_markov}
Let $W^{(\mu)}$ be a Brownian motion with drift $\mu$ in the Cartan subalgebra $\afrak \approx \R^r$, then
$$\Lambda_t := \hw\left( B_t\left( W^{(\mu)} \right) \right)  = \Tc_{w_0} W^{(\mu)}_t$$
is a diffusion process with infinitesimal generator
$$ \Lc = \psi_\mu^{-1}\left( \frac{1}{2} \Delta - \sum_{\alpha \in \Delta} \frac{1}{2} \langle \alpha, \alpha \rangle e^{-\alpha\left( x \right)}  - \frac{ \langle \mu, \mu \rangle }{2} \right) \psi_\mu = \frac{1}{2} \Delta + \nabla \log\left(\psi_\mu\right) \cdot \nabla $$
where $\psi_\mu$ is the Whittaker function.
\end{thm}
\begin{proof}[Outline of proof]
In section \ref{section:markovian_points}, we prove a weaker version as theorem \ref{thm:whittaker_process_g} for $\mu \in C$ and a finite starting point $x_0$. We strenghten that result in section \ref{section:intertwined_markov_kernels} to all $\mu \in \afrak$. And finally we take $x_0$ to ``$-\infty$'' in subsection \eqref{subsection:entrance_point}.
\end{proof}

\begin{rmk}
From the point of view of particle systems, if $W^{(\mu)}$ is the process with generator $\frac{1}{2} \Delta + \langle \mu, \nabla \rangle - \sum_{\alpha \in \Delta} \frac{1}{2} \langle \alpha, \alpha \rangle e^{-\alpha\left( x \right)}$, i.e it is the Brownian motion with drift $\mu \in C$ and killing measure following the Toda potential then for $x \in \afrak$, $\psi_\mu(x)$ is proportional to $ e^{ \langle \mu, \lambda \rangle }\P_x\left( W^{(\mu)} \textrm{ survives} \right)$. Then, the Whittaker process is interpreted as the sub-Markovian Brownian motion $W^{(\mu)}$ conditioned to survive.

From the point of view of representation theory, this theorem says that the highest weight process is Markov. Because concatenation of paths models tensor products of crystals in the path model (Theorem 6.10 in \cite{bib:chh14a}), and Brownian motion has independent increments, we are observing a tensor product dynamic. Therefore, the evolution of the highest weight tells us how tensor products decompose and the quantum Toda Hamiltonian is an infinitesimal version of the Littlewood-Richardson rule.

Another pedantic restatement would be ``the isomorphism class of crystals generated by Brownian motion is Markovian''. 
\end{rmk}

This theorem reduces to the Matsumoto-Yor theorem (\cite{bib:MY00-1, bib:MY00-2}) in the case of $SL_2$ and the theorem by O'Connell (\cite{bib:OConnell}) in the case of $GL_n$ (type $A_n$). Note that since O'Connell's contruction has an application to a semi-discrete polymer model, one expects other Lie types to be related to different geometries. In section \ref{section:degenerations}, we will see that this theorem is a geometric lifting of theorem 5.6 in \cite{bib:BBO}, which represents the highest weight process in the continuous Littelmann path model as a Brownian motion in the Weyl chamber. Because of the underlying representation theory, we argue that the Whittaker process is the canonical weakly non-intersecting particle system, in the same way as Dyson's Brownian motion is the natural strictly non-intersecting particle system.

\subsection{Landau-Ginzburg potentials as a conditionnal distribution}

Landau-Ginzburg potentials are indeed what appear when considering a random crystal element in $\Bc$ conditioned to have its highest weight being $\lambda$. More precisely, we observe the first component of the RS correspondence (Theorem \ref{thm:dynamic_rs_correspondence}).
\begin{thm}
\label{thm:canonical_measure}
Let $W^{(\mu)}$ be a BM with drift $\mu$ in $\afrak$ and fix $t>0$. The distribution of $B_t(W^{(\mu)})$ conditionally to the $\sigma$-algebra $\Fc_t^\Lambda$  and $\Lambda_t = \Tc_{w_0} \left( W^{(\mu)} \right)(t) = \lambda$ depends only on $\lambda$ and is given by:
\begin{align}
\P\left( B_t\left( W^{(\mu)} \right) \in dx \ | \ \Fc^\Lambda_t, \Lambda_t = \lambda \right)  & = \frac{1}{\psi_\mu\left( \lambda \right)} e^{ \langle \mu, \wt(x) \rangle - f_B(x) } \omega(dx)
\end{align}
with $\psi_\mu\left( \lambda \right)$ the normalizing constant.
\end{thm}
\begin{proof}
 See section \ref{section:intertwined_markov_kernels}.
\end{proof}

Therefore the measure
\begin{align}
\label{eq:canonical_measure_def}
e^{ - f_B(x) } \omega(dx) 
\end{align}
on $\Bc(\lambda)$ is the natural one induced by Brownian motion of geometric crystals, and the Archimedean Whittaker function appears as the geometric crystal's volume. A similar statement holds for Kirillov's orbital integral which encodes the volume of coadjoint orbits. The previous theorem is the ``{\it raison d'\^etre}'' of the following:
\begin{definition}
\label{def:canonical_probability_measure}
For a spectral parameter $\mu \in \afrak$, define $C_\mu(\lambda)$ as a $\Bc(\lambda)$-valued random variable whose distribution satisfies for every bounded measurable function on $\Bc(\lambda)$:
\begin{align}
\label{eq:canonical_measure_density}
\E\left( \varphi(C_\mu(\lambda)) \right) & = \frac{1}{\psi_\mu(\lambda)}  \int_{\Bc(\lambda)} \varphi(x) e^{ \langle \mu, \wt(x) \rangle - f_B(x) } \omega(dx)
\end{align}
We will refer to its law as the canonical probability measure on $\Bc(\lambda)$ with spectral parameter $\mu$. And $C_\mu(\lambda)$ will be referred to as a canonical random variable on $\Bc(\lambda)$ with spectral parameter $\mu$.
\end{definition}

\begin{properties}
\label{properties:canonical_probability_measure}
 For $\lambda \in \afrak$ and $\mu \in \afrak$:
 \begin{itemize}
  \item[(i)]   $$ S\left( C_\mu(\lambda) \right)     \stackrel{\Lc}{=} C_{w_0 \mu}(\lambda) $$
  \item[(ii)]  $$ \iota\left( C_\mu(\lambda) \right) \stackrel{\Lc}{=} C_{-\mu}(-w_0 \lambda) $$
  \item[(iii)] $W$-invariance:
               $$\forall w \in W, C_{w\mu}(\lambda)  \stackrel{\Lc}{=} C_\mu(\lambda)$$
 \end{itemize}
\end{properties}
\begin{proof}
Now for fixed $T>0$, consider a Brownian motion $\left(W_{t}^{(\mu)} \right)_{ 0 \leq t \leq T}$ in $\afrak$ with drift $\mu$, on the time interval $[0,T]$ :
\begin{itemize}
 \item[(i)] The Sch\"utzenberger involution acts at the path level as (see subsection 11.2 in \cite{bib:chh14a}):
$$ S\left( W^{(\mu)} \right)_t = - w_0 \left( W^{(\mu)}_T - W^{(\mu)}_{T-t} \right) ; 0 \leq t \leq T$$
which is also a Brownian motion, with drift $w_0 \mu$. Because the Sch\"utzenberger involution leaves the highest weights fixed (properties 11.5 in \cite{bib:chh14a}), we have equality for endpoints:
$$ \lambda = \Tc_{w_0}\left( W^{(\mu)} \right)_T = \Tc_{w_0} \circ S \left( W^{(\mu)} \right)_T$$
Hence, even if the filtrations generated by the paths $\Tc_{w_0}\left( W^{(\mu)} \right)$ and $\Tc_{w_0} \circ S \left( W^{(\mu)} \right)$ are different, we still have:
$$                   \left( B_T\left( W^{(\mu)}    \right) | \Tc_{w_0}\left( W^{(\mu)}    \right)_T = \lambda \right) 
   \stackrel{\Lc}{=} \left( B_T\left( S(W^{(\mu)}) \right) | \Tc_{w_0}\left( S(W^{(\mu)}) \right)_T = \lambda \right) $$
Using theorems 11.7 in \cite{bib:chh14a} and \ref{thm:canonical_measure}, we conclude.

 \item[(ii)] The proof is similar to (i) as the involution $\iota$ changes the drift $\mu$ to $-\mu$ when applied to a Brownian motion and $\iota\left( \Bc(\lambda) \right) = \Bc(-w_0 \lambda )$.

 \item[(iii)] It is a consequence of the invariance of $f_B$ w.r.t the action of $W$ on the crystal (Properties \ref{properties:superpotential_properties}) and the invariance of $\omega$ w.r.t. crystal actions (proved later in theorem \ref{thm:omega_invariance_on_crystal}). The weight map is also equivariant. 
\end{itemize}
\end{proof}

For classical discrete crystals \cite{bib:HongKang}, the measure of interest is the counting measure. For instance, the number of elements gives the dimension of the associated module. More precisely, for $\lambda \in P^+$ is a dominant weight, consider $\Bfrak(\lambda)$ the Kashiwara $G^\vee$-crystal with highest weight $\lambda$. The canonical measure is simply:
\begin{align}
\label{eq:canonical_measure_discrete}
& \sum_{b \in \Bfrak(\lambda)} \delta_{b} 
\end{align}
where $\delta_b$ stands for the Dirac measure at the element $b$. The character formula for the highest weight module $V\left(\lambda\right)$ reads:
\begin{align}
\label{eq:character_discrete}
\forall \mu \in \afrak, \ch V\left( \lambda \right)(\mu) & = \sum_{b \in \Bfrak(\lambda)} e^{\langle \mu, \wt(b)\rangle}
\end{align}
This allows to normalize the canonical measure \eqref{eq:canonical_measure_discrete} into a probability measure on $\Bfrak(\lambda)$ with spectral parameter $\mu$:
\begin{align}
\label{eq:canonical_probability_measure_discrete}
& \frac{1}{\ch V\left( \lambda \right)(\mu)}\sum_{b \in \Bfrak(\lambda)} e^{\langle \mu, wt(b)\rangle} \delta_{b} 
\end{align}
Let us mention that the characters can be seen as the Laplace transform of the image measure of \eqref{eq:canonical_measure_discrete} through the weight map. Completely analogous facts will hold for the Whittaker functions and Kirillov's orbital integral.

\subsection{The Archimedean Whittaker function as a character for geometric crystals}

Originally, the Duistermaat-Heckman measure was used to refer to the asymptotic weight multipliticities for a very large finite dimensional representation of a semisimple group (\cite{bib:Heckman82}, \cite{bib:GS90} section 33). It is also the image measure of the uniform measure on a continuous crystal under the weight map (\cite{bib:BBO2} section 5.3). One can use the Littelmann path model for very long paths to recover easily the Duistermaat-Heckman measure as asymptotic weight multiplicities (\cite{bib:BBO} remark 5.8). Now that we have identified a natural measure on geometric crystals, we will take virtually the same definition.
\begin{definition}
For $\lambda \in \afrak$, define the geometric Duistermaat-Heckman measure $DH^\lambda$ on $\mathfrak{a}$ as the image of the canonical measure \eqref{eq:canonical_measure_def} under the weight map $\wt$.
\end{definition}
One needs to think of $DH^{\lambda}$ as the measure encoding \emph{geometric} weight multiplicities and hence its Fourier-Laplace transform plays the role of character, analogously to equation \eqref{eq:character_discrete} in the discrete setting. In this geometric setting, we obtain a representation-theoretic definition of Whittaker functions that goes beyond the mere fact that $\psi_\mu\left( \lambda \right)$ is the volume of the geometric crystal $\Bc(\lambda)$. Moreover, we will see in theorem \ref{thm:intertwining_torus} that this measure intertwines the Laplacian on $\afrak$ and the quantum Toda Hamiltonian, or equivalently Brownian motion and the Whittaker process. 

\begin{definition}[Whittaker functions]
Whittaker functions are defined as the Laplace transform of the geometric Duistermaat-Heckman measure. For $\lambda \in \afrak$ and $\mu \in \hfrak$, it is given by:
\label{def:whittaker_functions}
\begin{align*}
\psi_\mu(\lambda) & = \int_\afrak e^{\langle \mu, k\rangle} DH^\lambda(dk)\\
                  & = \int_{\Bc(\lambda)} e^{ \langle \mu, \wt(x) \rangle - f_B(x) } \omega(dx)
\end{align*}
\end{definition}

Now, define $b: \hfrak \rightarrow \C$ as the meromorphic function
$$ b(\mu) := \prod_{\beta \in \Phi^+} \Gamma\left( \langle \beta^\vee, \mu \rangle \right)$$
It allows to define a natural normalization in our setting.

\begin{thm}
 \label{thm:whittaker_functions_properties}
 The Whittaker function satisfies the following:
 \begin{itemize}
  \item[(i)]   $\psi_\mu(\lambda)$ is an entire function in $\mu \in \hfrak = \afrak \otimes \C \approx \C^n$.
  \item[(ii)]  $\psi_\mu$ is invariant in $\mu$ under the Weyl group's action.
  \item[(iii)] For $\mu \in C$, the Weyl chamber, we have a probabilistic representation of the Whittaker function using $W^{(\mu)}$ a Brownian motion on $\afrak$ with drift $\mu$:
               $$ \psi_\mu\left( \lambda \right) = b(\mu) e^{ \langle \mu, \lambda \rangle }\E_\lambda\left( \exp\left( - \sum_{\alpha \in \Delta} \half \langle \alpha, \alpha \rangle \int_0^\infty ds \ e^{-\alpha(W_s^{(\mu)})} \right) \right)$$
               In this case, $\psi_\mu$ is the unique solution to the quantum Toda eigenequation:
               $$ \half \Delta \psi_\mu(x) - \sum_{\alpha \in \Delta } \half \langle \alpha, \alpha \rangle e^{-\alpha(x) } \psi_\mu(x) = \half \langle \mu, \mu \rangle \psi_\mu(x)$$
               such that $\psi_\mu(x) e^{-\langle \mu, x \rangle }$ is bounded with growth condition $\psi_\mu(x) e^{-\langle \mu, x \rangle } \stackrel{ x \rightarrow \infty, x \in C }{\longrightarrow} b(\mu)$
 \end{itemize}
\end{thm}
\begin{proof}
\begin{itemize}
 \item[(i)] In coordinates, thanks to the estimate in theorem \ref{thm:superpotential_estimate} and the weight map expression in theorem 4.22 in \cite{bib:chh14a}, we see that
$$ \phi(\mu, x) := \exp\left( \langle \mu, \wt(x) \rangle - f_B(x) \right) \omega(dx)$$
is holomorphic in $\mu \in \hfrak$ and integrable in the $x$ parameter uniformly for $\mu$ in a compact set. The same holds for partial derivatives w.r.t $\mu$. Thus, integration in the $x$ parameter will give a holomorphic function whose domain is all of $\hfrak$. Hence, $\psi_\mu(\lambda)$ is entire in the $\mu$ parameter.
 \item[(ii)] Invariance under the Weyl group action is a consequence of the invariance for $e^{-f_B(x)} \omega(dx)$, and equivariance for the weight map.
 \item[(iii)] In proposition \ref{proposition:link_with_canonical}, we prove that the probabilistic representation coincides indeed with the previous definition. For the characterization as the unique solution of the above PDE, we reproduce the martingale argument of \cite{bib:BOC09} corollary 2.3, for completeness. Write $\varphi_\mu = \psi_\mu e^{-\langle \mu, .\rangle}$, which needs to be the unique bounded solution to:
 $$ \half \Delta \varphi_\mu(x) + \langle \mu, \nabla \varphi_\mu \rangle - \sum_{\alpha \in \Delta } \half \langle \alpha, \alpha \rangle e^{-\alpha(x) } \varphi_\mu(x) = 0$$
 with growth condition $\varphi_\mu(x) \stackrel{ x \rightarrow \infty, x \in C }{\longrightarrow} b(\mu)$.

 As a consequence of the Feynman-Kac formula:
 $$ \varphi_\mu(x) = b(\mu) \E\left( \exp\left( - \sum_{\alpha \in \Delta} \half \langle \alpha, \alpha \rangle \int_0^\infty ds \ e^{-\alpha(x + W_s^{(\mu)})} \right) \right)$$
 solves the partial differential equation. For uniqueness, we use a martingale argument. If $\phi$ is a bounded solution such that:
 $$\phi(x) \stackrel{ x \rightarrow \infty, x \in C }{\longrightarrow} 0$$
 Then:
 $$ \phi(x + W_t^{(\mu)})\exp\left( - \sum_{\alpha \in \Delta} \half \langle \alpha, \alpha \rangle \int_0^t ds \ e^{-\alpha(x + W_s^{(\mu)})} \right)$$
 is bounded martingale going to zero as $t \rightarrow \infty$. Therefore, it must vanish identically. Hence uniqueness, by linearity.
\end{itemize}
\end{proof}

We will see in the next section that the integral is finite and that the Whittaker functions are well behaved. This semi-explicit integral formula given for Whittaker functions hints directly to the work of $\cite{bib:GLO1, bib:GLO2, bib:Givental}$. It is not so easy to link their formulae to ours, because of the multiple choices of coordinates. Notice however that our approach makes the choice of totally positive matrices a natural integration cycle.

We defined Whittaker functions $\psi_\mu$ as the Laplace transform of measure induced on weights by the canonical measure. Our definition is very different from Jacquet's original definition \cite{bib:Jacquet67} as an integral on the unipotent group $N$. In those circumstances, Whittaker functions are of special interest in number theory, as they appear in the Fourier expansion of Maass forms (see Goldfeld \cite{bib:Goldfeld}, chapter 5). Our approach has the advantage to define well-behaved functions using integrals that converge rapidly for all $\mu$. Moreover, the integrands are positive. Finally, a lot of structure is exhibited thanks to the underlying geometric crystals: Whittaker functions play the role of characters in the theory.

When comparing to the non-Archimedean setting where Whittaker functions are exactly proportional to characters of highest weight modules, one realizes that the function $b: \hfrak \rightarrow \C$ is an Archimedean analogue of the Weyl denominator. The reader is invited to look at the main results of \cite{bib:chh14-adic} and the discussion in subsection 3.2 which details these similarities.

\subsection{Outline of paper and strategy of proof}
We will work our way towards the proofs in four steps: 

\begin{itemize}
 \item Section \ref{section:markovian_points} is of probabilistic nature and analyzes the hypoelliptic Brownian motion on the group at hand. At this level, we will have to restrict our framework to $\mu \in C$, the Weyl chamber. We exhibit the Whittaker process starting at $x_0$ as a random path transform $T_g$ of Brownian motion. In fact, our result in that section is much stronger and of independent interest. We give a complete a characterization of the Poisson boundary and interpret the random path transform $T_g$ as a conditioning to hit the Poisson boundary at a certain point. The Whittaker process appears naturally as the only Markovian evolution which can appear from such a conditioning. 
 
 \item Section \ref{section:properties_LG} describes more precisely the parametrizations of geometric crystals and gives estimates on Landau-Ginzburg potentials. We have refrained from introducing these parametrizations until necessary. This allows to prove that the exiting laws corresponding to the Whittaker process are well-defined for all $\mu \in C$ and intimately related to the Landau-Ginzburg potentials.
 
 \item We extend the definition of the Whittaker process to arbitrary $\mu$ in section \ref{section:intertwined_markov_kernels}. We will use the Rogers-Pitman intertwining criterion (Theorem \ref{thm:rogerspitman_criterion}) in the framework of Markov functions.
 
 \item In section \ref{section:asymptotics}, we obtain the highest weight process and thus theorem \ref{thm:highest_weight_is_markov} by taking a starting point that is ``$x_0 = -\infty$''. To that endeavor, we will have to analyze the asymptotic behavior of the canonical measure as the highest weight goes to infinity in the opposite Weyl chamber.
 
 \item Finally, in section \ref{section:degenerations}, we describe the tropicalization procedure at the geometric and the probabilistic levels.
\end{itemize}

\section{Markovian points on the Poisson boundary}
\label{section:markovian_points}

While we are mainly interested in the Archimedean Whittaker functions, it is better to look at the larger picture. As we will see, the natural invariant process related to Whittaker functions is the hypoelliptic Brownian motion from eq. \eqref{eq:process_B_explicit}. By interpreting $N_{\geq 0}$ as a topological Furstenberg boundary where the process exits, every bounded harmonic function can be represented as an integral on $N_{\geq 0}$. We give in subsection \ref{subsection:poisson_boundary} a complete description of bounded harmonic functions - the Poisson boundary. It is identified to the set of functions:
\begin{align}
\label{eq:identification_poisson_boundary}
\Pc := & \left\{ v: N_{\geq 0} \rightarrow \R \ | \ v \textrm{ bounded } \right\}
\end{align}

This is relevant in relationship to the following. There is a path transform that appeared naturally in the Littelmann path model for geometric crystals. For $g \in U_{\geq 0}$, the transform:
$$
\begin{array}{cccc}
  T_g : & \Cc\left( \R_+, \afrak \right) & \longrightarrow & \Cc\left( \R_+, \afrak \right)
\end{array}
$$
is implicitly defined on a continuous path $X$ by
\begin{align}
\label{eq:def_path_transform} 
T_g X(t) := & \log \left[ g B_t(X)\right]_0 .
\end{align}
Its properties are detailed in section 5.4 of \cite{bib:chh14a}. 

\begin{thm}[Conditional representation theorem 2.12 \cite{bib:chh14_exit}]
\label{thm:conditional_representation}
If $\mu \in C$ and $W$ is a Brownian motion on $\afrak$, then:
 $$X^{x_0,(\mu)} := x_0 + T_{g } \left( W^{(w_0 \mu)} \right)$$
is a Brownian motion with drift $\mu \in C$, starting at $x_0$ and conditioned to 
$$ N_\infty(X^{0,(\mu)}) = \Theta(g) := \left[ g \bar{w}_0 \right]_0$$ 
\end{thm}

As explained in \cite{bib:chh14_exit}, the diffeomorphism $\Theta: U_{>0}^{w_0} \rightarrow N_{>0}^{w_0}$ is a group-theoretic analogue of the inverse map on $\R_+^*$. Moreover if $X^{(\mu)}$ is a Brownian motion on $\afrak$ with drift $\mu$, then the law of $\Theta^{-1}\left[ N_\infty(X^{(\mu)}) \right]$ is expressed in coordinates in term of gamma random variables. Hence:

\begin{definition}[Group-theoretic gamma and inverse gamma distributions - definition 2.19 in \cite{bib:chh14_exit} ]
\label{def:group_gamma_law}
For $X^{(\mu)}$ a Brownian motion on $\afrak$ with drift $\mu \in C$, define an inverse gamma random variable $D_\mu$ on $N$ as the distribution of $N_\infty(X^{(\mu)})$. The gamma random variable $\Gamma_\mu$ on $U$ is such that:
$$ D_\mu \stackrel{\mathcal{L}}{=} \Theta\left( \Gamma_\mu \right) \ .$$
Both are positive random variables in the sense of total positivity: almost surely $\Gamma_\mu \in U_{>0}^{w_0}$ and $D_\mu \in N_{>0}^{w_0}$.
\end{definition}

Now consider a positive function $v \in \Pc$. Revisiting theorem \ref{thm:conditional_representation}, if $g$ is taken as random with a distribution such that for every bounded function $\varphi$ on $U$:
$$ \E\left( \varphi \circ \Theta\left( g \right) \right)
 = \frac{ \E\left( \varphi(D_\mu) v( D_\mu ) \right) }
        { \E\left( v( D_\mu ) \right) }$$
then $X^{x_0,(\mu)}$ is distributed as a Brownian motion with drift $\mu$ conditioned to $N_\infty(X^{(\mu)})$ to follow a certain distribution depending on $v$. In terms of the potential-theoretic point of view given in the next subsection \ref{subsection:poisson_boundary}, we are conditioning the hypoelliptic Brownian motion $\left( B_t(X^{(\mu)}); t \geq 0 \right)$ to hit the Poisson boundary at the point $v \in \Pc$. Taking $v \equiv 1$ identically yields no conditioning. 

We ask the question ``Which $v \in \Pc$ force the process $\left( X^{x_0,(\mu)}_t ; t \geq 0 \right)$ to be Markovian in its own filtration?'' The answer is the main result of this section.
\begin{thm}[Characterization of Markovian points on the Poisson boundary]
\label{thm:markovian_points}
Assume $v \in \Pc$ differentiable. The process $\left( X^{x_0,(\mu)}_t ; t \geq 0 \right)$ is a Markovian diffusion if and only if $v$ is given in terms of unipotent characters by:
$$ v = \exp\left( - \sum_{\alpha \in \Delta} C_\alpha \chi^-_\alpha \right)$$
for constants $C_\alpha \geq 0$. The superscript $-$ refers to the same characters as in subsection \ref{subsection:LG_potentials}, only for the opposite unipotent group $N$.
\end{thm}
\begin{proof}
Subsection \ref{subsection:remarkable_deformation}.
\end{proof}

Leaving the degenerate cases where some of the $C_\alpha$ vanish in theorem \ref{thm:markovian_points}, we assume that $C_\alpha > 0$ for all $\alpha \in \Delta$. Under such an assumption, these constants can be absorbed in the characters as there is an $a \in \afrak$ such that $C_\alpha = e^{-\alpha(a)}$, and:
$$ \sum_{\alpha \in \Delta} C_\alpha \chi_\alpha^{-}(n) = \chi^-\left( e^a n e^{-a} \right).$$
There is no loss of generality in using the standard character $\chi^-$ on $N$. More general characters simply add a shift via conjugation by torus elements. Theorem \ref{thm:markovian_points} urges to define a deformation of the laws of $\Gamma_\mu$ and $D_\mu$ using the standard character.

\begin{definition}[Generalized gamma and inverse gamma]
\label{def:generalized_gamma}
 For $\mu \in C$ and $\lambda \in \afrak$, define $D_\mu\left( \lambda \right)$ as the $N_{\geq 0}$-valued random variable defined by:
$$ \forall \varphi \geq 0, \E\left( \varphi( D_\mu(\lambda) ) \right)
                  = \frac{ \E\left( \varphi(D_\mu) \exp\left( -\chi^-(e^{\lambda} D_\mu e^{-\lambda} ) \right) \right) }
                         { \E\left( \exp\left( -\chi^-(e^{\lambda} D_\mu e^{-\lambda} ) \right) \right) }
$$
and the $U$-valued random variable $\Gamma_\mu(\lambda)$ as:
$$ D_\mu(\lambda) = \Theta\left( \Gamma_\mu(\lambda) \right)$$
\end{definition}

\begin{rmk}
Because:
$$ \frac{\langle \alpha, \alpha \rangle}{2} \int_0^\infty e^{-\alpha( X^{(\mu)}_s )} ds = \chi^-_\alpha \left( N_\infty(X^{(\mu)}) \right)$$
and $D_\mu \eqlaw N_\infty(X^{(\mu)})$, the Whittaker function's expression in (iii) of theorem \ref{thm:whittaker_functions_properties} takes the form:
\begin{align}
\label{eq:whittaker_function_D_mu}
\psi_\mu(x) & = b(\mu) e^{\langle \mu, x \rangle}\E\left( \exp\left( -\chi^-(e^{x} D_\mu e^{-x} ) \right) \right)
\end{align}
\end{rmk}

Even though the definition assumes $\mu \in C$, the random variables from definition \ref{def:generalized_gamma} are in fact well-defined for all $\mu \in \afrak$ as we will see in proposition \ref{proposition:link_with_canonical}. In the rank one setting for example, we obtain the natural group-theoretic generalization of generalized inverse Gaussian laws used by Matsumoto and Yor \cite{bib:MY00-2}. Moreover, theorem 3 in \cite{bib:Bau02} is the particular case of our theorem \ref{thm:markovian_points} in the $A_1$-type.
\begin{example}[$A_1$ type - Example 2.17 in \cite{bib:chh14_exit} continued]
 In the $A_1$ type, if $D_\mu(\lambda) = \begin{pmatrix} 1 & 0\\ GIG_{\mu}(\lambda) & 1 \end{pmatrix}$ then $GIG_{\mu}(\lambda)$ is the generalized inverse gaussian law used in \cite{bib:MY00-2}:
$$\P\left( GIG_\mu(\lambda) \in dt \right) = \frac{e^{\lambda \mu}}{K_\mu(e^{-2 \lambda})} t^{\mu} e^{-t - \frac{e^{-2\lambda}}{t}} \frac{dt}{t}$$
Here $K_\mu$ is a Bessel function of the second kind, also known as the Macdonald function. It coincides with the $SL_2$-Whittaker function.
\end{example}

The main ideas were in fact already in \cite{bib:BOC09} where the Whittaker process was built out of a Brownian motion conditionned with respect to its exponential functionals. What Baudoin and O'Connell express as conditioning the exponential functionals to follow a certain law, we rephrase as conditioning the hypoelliptic Brownian motion to hit a specific point on the boundary. This change of perspective stresses the role of $\left( B_t(W^{(\mu)}); t \geq 0 \right)$, which is a process containing not only exponential functionals but also iterated integrals. Thus, we obtain the generalization of proposition 3.3 in \cite{bib:OConnell} to all Lie groups.

\begin{thm}
\label{thm:whittaker_process_g}
Consider a Brownian motion $W$, a drift $\mu \in C$ and an independent random variable $g \stackrel{\Lc}{=} \Gamma_{\mu}(x_0)$. Then the process $\left( X^{x_0,(\mu)}_t = x_0 + T_{g }\left( W^{(w_0 \mu)} \right)_t; t\geq0 \right)$ is Markovian with infinitesimal generator
$$\frac{1}{2} \Delta + \langle \nabla \log \psi_\mu , \nabla \rangle $$
\end{thm}
\begin{proof} 
Subsection \ref{subsection:remarkable_deformation}.
\end{proof}

\subsection{The hypoelliptic Brownian motion and its Poisson boundary}
\label{subsection:poisson_boundary}
As announced, we start by a complete description of the Poisson boundary.

\paragraph{Reformulation in term of an invariant process:} 
An analytic approach to the Archimedean Whittaker functions (cf. \cite{bib:Hashizume82} for instance) is to look at them as eigenfunctions of the quantum Toda Hamiltonian, a Schr\"odinger operator on $\afrak \approx \R^n$:
\begin{align*}
 H & = \half \Delta - \half \sum_{\alpha} \langle \alpha, \alpha \rangle e^{-\alpha(x)}
\end{align*}
They satisfy:
$$ H \psi_\mu = \half \langle \mu, \mu \rangle \psi_\mu$$

However, the analysis of this Schr\"odinger operator is not completely obvious. Morally, the only operators on a group that are easy to analyze are those that invariant under the group at hand. A way to turn the problem into an invariant problem is to look at the Euclidian $\afrak$ as part of a larger space, which is a group and on which the Whittaker functions are harmonic functions for an invariant process.

\begin{lemma}
\label{lemma:equivalence_toda_harmonicity}
A function $\psi_\mu: \afrak \rightarrow \C$ on $\afrak$ solves:
$$ H \psi_\mu = \half \langle \mu, \mu \rangle \psi_\mu$$
if and only if the function $\Phi_\mu: B\left( \R \right) \rightarrow \C$ defined by:
$$ \Phi_\mu(n e^x) := \exp\left( -\chi^-(n) \right) \psi_\mu(x) e^{-\langle \mu, x\rangle}$$
is harmonic for the hypoelliptic operator $\Dc^{(\mu)}$ defined in \ref{def:hypoelliptic_operator}.
\end{lemma}
\begin{proof}
Notice that for $t \in \R$ and for $\alpha \in \Delta$:
$$
\Phi_\mu(n e^x e^{t f_\alpha}) = \Phi_\mu\left(n \exp(t e^{-\alpha(x)} f_\alpha) e^x \right)
                               = \exp\left( - t e^{-\alpha(x)} \right) \Phi_\mu(n e^x)
$$
Hence:
$$
\Lc_{f_\alpha} \Phi_\mu(n e^x) = \frac{d}{dt}\left( \Phi_\mu(n e^x e^{t f_\alpha}) \right)_{|t=0}
                               = - e^{-\alpha(x)} \Phi_\mu(n e^x )
$$
Therefore, writing $\varphi_\mu = \psi_\mu e^{-\langle \mu, .\rangle}$, we have the following succession of equivalent statements:
\begin{align*}
                    \quad &
                    \Dc^{(\mu)} \Phi_\mu = 0\\
\Longleftrightarrow \quad &
\frac{1}{2} (\Delta \varphi_\mu) e^{-\chi^-} + \langle \mu, \nabla \varphi_\mu \rangle e^{-\chi^-} + \half \sum_{\alpha \in \Delta} \langle \alpha, \alpha \rangle \Lc_{f_\alpha} \Phi_\mu = 0\\
\Longleftrightarrow \quad &
\frac{1}{2} (\Delta \varphi_\mu) + \langle\mu, \nabla \varphi_\mu\rangle - \half \sum_{\alpha \in \Delta} \langle \alpha, \alpha \rangle e^{-\alpha(x)} \varphi_\mu = 0\\
\Longleftrightarrow \quad &
H \psi_\mu = \half \langle \mu, \mu \rangle \psi_\mu
\end{align*}
\end{proof}

\paragraph{Probabilistic integral representations on the boundary:}
By definition, the Poisson boundary of $\left( B_t(W^{(\mu)}); t \geq 0 \right)$ is the set of bounded function that are harmonic for this process. And the function $\Phi_\mu$ is in such a boundary as it is a bounded harmonic function. Following a standard idea in potential theory, we should be able to represent $\Phi_\mu$ as the integral of a function over the ``boundary'' of $B$. However, the Borel subgroup $B$ is not compact. But the notion of boundary needed is not the topological one. Furstenberg developed such a notion and a very good account of the theory of boundaries on Lie groups is \cite{bib:Bab02} in the case of random walks. The continuous case is, in a way, simpler.

For boundary, one has to consider a space with a $B\left( \R \right)$-action and a natural invariant measure. Restricting ourselves to the case where $\mu \in C$, it has been proven in \cite{bib:chh14_exit} that the $N$-part of the process $\left( B_t(W^{(\mu)}); t \geq 0 \right)$ (equation \eqref{eq:process_N_explicit}) and converges in $N$ when $t \rightarrow \infty$. Therefore, a natural choice for a boundary of $B\left( \R \right)$ is simply $N$ and the invariant measure is the law of $N_\infty(W^{(\mu)})$ when $\mu \in C$ - the inverse gamma distribution. The Borel subgroup $B$ acts on $N$ as:
$$ \forall b = na \in B, \forall n' \in N, (na) \cdot n' = n a n' a^{-1}$$

\begin{proposition}
\label{proposition:poisson_boundary_char}
For any bounded function $v$ on $N_{\geq 0}$, we obtain a harmonic function $\varphi$ for $\Dc^{(\mu)}$ by considering:
$$ \varphi: n a \mapsto \E\left( v(n a N_\infty(W^{(\mu)}) a^{-1} ) \right) $$
with $W$ a Brownian motion on $\afrak$. Reciprocally, any bounded harmonic function has this form.
\end{proposition}

Recalling that there is bijection between bounded harmonic functions and shift-invariant random variables. The proposition becomes intuitive upon realizing that shift-invariant random variables on the coordinate space of $\left( B_t(W^{(\mu)}); t \geq 0 \right)$ starting at $na$ must be function of only the convergent part, which is $n a N_\infty(W^{(\mu)}) a^{-1}$. The formal argument is the dévissage developed in \cite{bib:AT13}. More precisely, our proposition \ref{proposition:poisson_boundary_char} is a consequence of their theorem 1. We give a short proof after developing a few lemmas we will need in the sequel.

A useful fact is that the action of $A$ on $N$ via conjugation contracts space in the direction of the Weyl chamber:
\begin{lemma}
\label{lemma:limit_contraction}
For all $n \in N$:
$$ e^{x} n e^{-x} \stackrel{x \rightarrow \infty}{\longrightarrow} \Id $$
with $x$ staying inside a sector of the open Weyl chamber.
\end{lemma}
\begin{proof}
The claim is stable via multiplication and inversion in $N$. Therefore, the set of $n \in N$ satisfying the lemma is a closed subgroup of $N$. Because the unipotent group $N$ is generated by the one-parameter subgroups $N_\alpha = \left( e^{t f_\alpha} \right)_{t \in \C}$ for $\alpha \in \Delta$, we only need to prove it for $n \in N_\alpha$ for a fixed $\alpha$. We have for $t \in \C$:
$$ e^{x} e^{t f_\alpha} e^{-x} = \exp\left( t e^{-\alpha(x)} f_\alpha \right) \stackrel{x \rightarrow \infty}{\longrightarrow} \Id$$ 
\end{proof}

Also, let us record that:
\begin{lemma}
Fix $t,s >0$. For any path $X$, we have:
\begin{align}
\label{eq:B_decomposition}
B_{t+s}( X ) = & B_t( X ) B_s( X_{t+.} - X_{t})
\end{align}
\begin{align}
\label{eq:N_decomposition}
N_{t+s}( X ) = & N_t( X ) e^{X_t} N_s( X_{t+.} - X_{t}) e^{-X_t} 
\end{align}
In particular, for a Brownian motion $X^{(\mu)}$ with drift $\mu \in C$:
\begin{align}
\label{eq:N_infty_decomposition}
N_\infty( X^{(\mu)} ) = & N_t( X^{(\mu)} ) e^{X^{(\mu)}_t} N_\infty( X^{(\mu)}_{t+.} - X^{(\mu)}_{t}) e^{-X^{(\mu)}_t} 
\end{align}
\end{lemma}
\begin{proof}
Equation \eqref{eq:B_decomposition} follows from the left-invariance in the definition of the flow $\left( B_t( \cdot ); t \geq 0 \right)$. Equation \eqref{eq:N_decomposition} is obtained by taking only the $N$-part. Taking $s \rightarrow \infty$ yields the third equation, in the case of a Brownian motion with drift $\mu \in C$.
\end{proof}

\begin{proof}[Proof of proposition \ref{proposition:poisson_boundary_char}]
Consider $v$, a bounded function on $N_{\geq 0}$ and set for $na \in B$:
$$ \varphi(na) = \E\left( v(n a N_\infty( \widetilde{X}^{(\mu)}) a^{-1} ) \right)$$
where $\widetilde{X}$ is a Brownian motion independent of $X$. Without loss of generality we can assume that $\widetilde{X}^{(\mu)} = X^{(\mu)}_{t+.} - X^{(\mu)}$ and use \eqref{eq:N_infty_decomposition}. Hence:
$$ \varphi\left( B_t(X^{(\mu)}) \right)
 = \E\left( v\left( B_t(X^{(\mu)}) N_\infty( \widetilde{X}^{(\mu)}) e^{-X^{(\mu)}_t } \right) | \Fc_t^X \right)
 = \E\left( v(N_\infty( X^{(\mu)}) ) | \Fc_t^X \right)$$
which is a backward martingale. We have proved that the bounded $\varphi$ is harmonic.

Reciprocally, if $\varphi$ is a bounded harmonic function, then for all $na \in B$, $M_t := \varphi\left( n a B_t(X^{(\mu)}) \right)$ is a martingale that converges almost surely to $M_\infty$. The random variable $M_\infty$ is necessarily a deterministic function of the entire trajectories $\left( X^{(\mu)}_s ; s \geq 0\right)$ and $\left( N(X^{(\mu)})_s ; s \geq 0\right)$, which we write as:
$$ M_\infty = M_\infty\left( (x_s)_{s \geq 0}, (n_s)_{s \geq 0} \right)$$
Moreover, $M_\infty$ has to be shift-invariant i.e for all $t \geq 0$:
$$ M_\infty\left( (x_s)_{s \geq 0}, (n_s)_{s \geq 0} \right)
 = M_\infty\left( (x_{s+t})_{s \geq 0}, (n_{s+t})_{s \geq 0} \right)$$
Following the dévissage argument, define for all $n \in N_{\geq 0}$:
$$ M^n_\infty := M_\infty\left( (x_s)_{s \geq 0}, ( n N\left( x \right)_\infty^{-1} n_s)_{s \geq 0} \right)$$
This is a shift invariant function which does not depend on the starting point of the path $(n_s)_{s \geq 0}$. Therefore $ \E\left( M^n_\infty | \Fc_{t}^X \right)$
is only a function of $X_t^{(\mu)}$ and harmonic. As the Poisson boundary of Euclidian Brownian motion is trivial, $M^n_\infty$ is a constant random variable which we denote by $v(n)$. We conclude by:
$$ M_\infty = M^{N_\infty(x)}_\infty = v\left( N_\infty(x) \right)$$
\end{proof}

The previous subsection tells us to take $v \in \Pc$ as a character of the unipotent subgroup $N$, in order to obtain Whittaker functions. In the larger picture however, the central object is the law of $N_\infty(W^{(\mu)})$, the exit law.

\paragraph{Exit law:} For later use, we record the explicit expression of the gamma distribution $\Gamma_\mu$. From \cite{bib:chh14_exit}, fix a reduced word ${\bf i} \in R(w_0)$ and $\left( \beta_j^{\bf i} \right)_{1 \leq j \leq m}$ the associated enumeration of positive roots:
\begin{align}
\label{eq:positive_root_enumeration} 
\forall \leq 1 \leq j \leq m, \ \beta_j^{\bf i} = & s_{i_1} s_{i_2} \dots s_{i_{j-1}} \alpha_{i_j}
\end{align}
Then, the law of $\Gamma_\mu$ is given by:
\begin{align}
\label{eq:explicit_Gamma_mu} 
\Gamma_\mu \eqlaw x_{ \bf i }\left( \gamma_{ \langle \beta_1^\vee, \mu \rangle}, \dots, \gamma_{\langle \beta_m^\vee, \mu \rangle} \right)
\end{align}
where the $\gamma_{.}$ denote independent gamma random variables on $\R_+$ with specified parameters.

\subsection{Hitting a specific point on the boundary}
\label{subsection:condition_N_infty}
Let $W$ be a Brownian motion on $\afrak$ and $\mu \in C$. Because of theorem \ref{thm:conditional_representation}, we know that:
 $$X^{x_0,(\mu)} := x_0 + T_{g } \left( W^{(w_0 \mu)} \right)$$
is a Brownian motion starting at $x_0$ with drift $\mu \in C$ at $x_0$ conditionned to $N_\infty(X^{0,(\mu)}) = \Theta(g)$. Denote its natural filtration by:
$$ \Fc_t := \sigma\left( X^{x_0, (\mu)}_s | 0 \leq s \leq t \right)$$

The purpose of this section is to take $g$ as random and independent of $W$ therefore conditionning $N_\infty(X^{(\mu)})$ to follow a certain specific law. Having in mind the description of the Poisson boundary in subsection \ref{subsection:poisson_boundary}, we are conditioning $B_t\left(X^{(\mu)}\right)$ to hit the Poisson boundary at a certain specific point. To that endeavor, fix a $v \in \Pc$. Let $\P$ be a probability under which:
\begin{align}
\label{eq:P_distributions}
g \eqlaw \Gamma_\mu 
& \Longleftrightarrow N_\infty(X^{(\mu)}) \eqlaw D_\mu
\end{align}
Under $\P$, $X^{x_0, (\mu)}$ is a Brownian motion in $\afrak$ with drift $\mu$. The function $v$ is used to define a deformed probability measure $\P^v$. More precisely, define $\P^v$ thanks to its Radon-Nikodym derivative with respect to $\P$:
$$ \frac{d\P^v}{d\P}
 = \frac{ v\left( N_\infty(X^{x_0, (\mu)}) \right) }
        { \E\left[ v\left( N_\infty(X^{x_0, (\mu)}) \right) \right] }
$$
Moreover, $\P^v$ can be interpreted as a probability measure under which $N_\infty(X^{x_0, (\mu)})$ is conditioned to follow a certain law. This law has a density $v$ with respect to the original law of $N_\infty(X^{(x_0, \mu)}) \eqlaw e^{x_0} D_\mu e^{-x_0}$. We are interested in describing the process $X^{x_0,(\mu)}$ under $\P^v$. A first step is to explicit the likelihood process:
$$ L_t := \frac{d \P^v}{d \P}_{| \Fc_t }$$

Equation \eqref{eq:N_infty_decomposition} gives a decomposition of $N_\infty(X^{(\mu)})$ in terms of $\mathcal{F}_t^{X}$-measurable variables and an independent variable $N_{\infty}( X^{(\mu)}_{t+.} - X^{(\mu)}_{t} )$ with same law. The same decomposition holds easily with $X^{x_0, (\mu)}$ instead of $X^{(\mu)}$:
\begin{align}
\label{eq:N_infty_decomposition_x_0}
N_{\infty}(X^{x_0, (\mu)}) & = N_t(X^{x_0, (\mu)}) e^{ X_t^{x_0, (\mu)} } N_{\infty}( X_{t+.}^{x_0, (\mu)} - X_{t}^{x_0, (\mu)} ) e^{ -X_t^{x_0, (\mu)} }
\end{align}

\begin{lemma}[Likelihood]
\label{lemma:likelihood}
$$ L_t := \frac{d \P^v}{d \P}_{| \Fc_t } = \Phi_v \left( B_t(X^{x_0, (\mu)}) \right)$$
where:
$$\Phi_v\left( n e^x \right) = \frac{ \E\left( v\left( n e^{ x } D_\mu  e^{-x  } \right) \right) }
                                    { \E\left( v\left( e^{ x_0 } D_\mu  e^{-x_0} \right) \right) }$$ 
\end{lemma}
\begin{proof}
We drop the superscript $x_0, (\mu)$ for convenience in $X^{x_0, (\mu)}$. From the likelihood's definition:
\begin{align*}
L_t := \quad & \frac{d \P^v}{d \P}_{| \Fc_t }\\
= \quad & \P\left( \frac{d \P^v}{d \P} | \Fc_t \right)\\
= \quad & \frac{ \P\left( v\left( N_\infty(X) \right) | \Fc_t \right) } 
               { \P\left[ v\left( N_\infty(X) \right) \right] }\\
\stackrel{eq. \eqref{eq:N_infty_decomposition_x_0}}{=}
        & \frac{ \P\left[ v\left( N_{t}(X) e^{ X_t } N_{\infty}( X_{t+.} - X_{t} ) e^{ -X_t } \right) | \Fc_t \right] }
               { \P\left[ v\left( N_{\infty}( X ) \right) \right] }
\end{align*}
Because $N_{\infty}( X_{t+.} - X_{t} ) \stackrel{ \Lc }{=} D_\mu$ under $\P$ (eq. \eqref{eq:P_distributions}) and is independent of $\Fc_t$, we have that on the set $\left\{ N_t(X) = n, X_t = x \right\}$:
$$ L_t = \frac{ \E\left( v\left( n e^{ x } D_\mu  e^{-x  } \right) \right) }
              { \E\left( v\left( e^{ x_0 } D_\mu  e^{-x_0} \right) \right) }$$
\end{proof}

\begin{corollary}
\label{corollary:girsanov}
$$ X^{x_0,(\mu)}_t - x_0 - \mu t - \int_0^t \nabla_x \log \Phi_v \left( B_s(X^{x_0,(\mu)}) \right) ds$$
is a $\P^v$ Brownian motion.
\end{corollary}
\begin{proof}
 By Girsanov's theorem (\cite{bib:RevuzYor} Chapter VIII):
$$X^{x_0,(\mu)}_t - x_0 - \mu t - \int_0^t d \langle \log L, X \rangle_s$$
is a $\P^v$ Brownian motion. Thanks to the previous lemma and to the fact that $N_t(X^{x_0})$ has zero quadratic variation, the bracket $d \langle \log L, X \rangle_s$ is indeed:
$$\nabla_x \log \Phi_v \left( B_s(X^{x_0,(\mu)}) \right) ds$$
\end{proof}

\subsection{Remarkable points and corresponding exit laws }
\label{subsection:remarkable_deformation}

\begin{proof}[Proof of theorem \ref{thm:markovian_points}]
The goal is to identify the remarkable laws that will force $X^{x_0, (\mu)}$ into becoming a Markov process. In the light of corollary \ref{corollary:girsanov}, we are aiming at identifying a bounded function $v \in \Pc$ such that the term $\nabla_x \log \Phi_v (n e^x)$ to only depends on $x$ for a certain $v$. Equivalently $\Phi_v( n e^x)$ has to break down into the product of two functions $g_v(n) f_v(x)$. For $x$ going to infinity inside the Weyl chamber, $\Phi_v(x, n)$ actually converges to:
$$g_v(n)f_v(\infty) = \lim_{x \rightarrow \infty, x\in C} \mathbb{E}\left( v\left( n e^{x} D_\mu e^{-x} \right) \right) = v\left( n \right)$$
thanks to lemma \ref{lemma:limit_contraction}. Then, $\Phi_v$ has in fact to breaks down as 
\begin{align}
\label{eq:Phi_v_splitting}
\forall n \in N^{w_0}_{>0}, \forall x \in \afrak, \ \Phi_v ( n e^x) & = v\left( n \right) f_v(x) 
\end{align}

By hypothesis, $f_v$ is bounded and $f_v(x) \stackrel{x \rightarrow \infty}{\longrightarrow} 1$. As $\left( L_t = \Phi_v \left( B_t(X^{x_0,(\mu)}) \right) ; t \geq 0 \right)$ is a real martingale and $\Dc^{(\mu)}$ is the infinitesimal generator of our hypoelliptic Brownian motion (proposition \ref{proposition:inf_generator}), we have that $\Dc^{(\mu)} \Phi_v = 0$.
First, we have that for all $t \in \R$:
$$ \Phi_v(n e^x e^{t f_\alpha}) = \Phi_v(n e^{t e^{-\alpha(x)} f_\alpha} e^x ) \stackrel{eq. \eqref{eq:Phi_v_splitting} }{=} v\left( n e^{ t e^{-\alpha(x)} f_\alpha }\right) f_v(x)$$
Hence, for $n \in N$ and $x \in \afrak$:
\begin{align*}
0 = & \left( \Dc^{(\mu)} \Phi_v \right) (n e^x) \\
  = & v(n) \left( \half \Delta_\afrak + \mu \right)(f_v)(x) + \sum_{\alpha \in \Delta} \frac{d}{dt}\left( \Phi_v(n e^x e^{t f_\alpha}) \right)_{|t=0}\\
  = & v(n) \left( \half \Delta_\afrak + \mu \right)(f_v)(x) + f_v(x) \sum_{\alpha \in \Delta} e^{-\alpha(x)} \Lc_{f_\alpha} \cdot v(n)
\end{align*}
Dividing the previous equation by $\Phi_v ( n e^x) = v\left( n \right) f_v(x) > 0$, we obtain:
$$ 0 = \frac{1}{f_v(x)}\left( \half \Delta_\afrak + \mu \right)(f_v)(x) + \sum_{\alpha \in \Delta} e^{-\alpha(x)} \frac{ \Lc_{f_\alpha} v(n)}{v(n)},$$
which indicates that for any fixed $n$ and $n'$ in $N$:
$$ \forall x \in \afrak, \ \sum_{\alpha \in \Delta} e^{-\alpha(x)} \frac{ \Lc_{f_\alpha} \cdot v(n )}{v(n )}
                         = \sum_{\alpha \in \Delta} e^{-\alpha(x)} \frac{ \Lc_{f_\alpha} \cdot v(n')}{v(n')}$$
Since the family of functions $\left( x \mapsto e^{-\alpha(x)} \right)_{\alpha \in \Delta}$ is free, there are constants $C_\alpha$ such that $\frac{ \Lc_{f_\alpha} v(n )}{v(n )} = -C_ \alpha$. Hence $v(n) = \exp\left( -\sum_{\alpha \in \Delta} C_\alpha \chi_\alpha^-(n) \right)$. In order for $v$ for to be bounded on $N_{\geq 0}$, these constants are non-negative.
\end{proof}

\begin{proof}[Proof of theorem \ref{thm:whittaker_process_g} for $\mu \in C$]
\label{proof:whittaker_process_g_mu_in_C}
If $g$ is taken to follow the law $\Gamma_\mu(x_0)$, we are in the situation where theorem \ref{thm:markovian_points} applies with $v = e^{-\chi^-}$. Thanks to the equation \eqref{eq:whittaker_function_D_mu}, the function $\Phi_v$ in lemma \ref{lemma:likelihood} takes the form:
$$ 
\Phi_v(n e^x) = e^{-\chi^-(n)} \frac{ \psi_\mu(x  ) e^{-\langle \mu, x   \rangle} }
                                    { \psi_\mu(x_0) e^{-\langle \mu, x_0 \rangle} }
$$
The reference measure is already of the form $\P^v$, with $\P$ the probability under which $g \stackrel{\Lc}{=} \Gamma_\mu$. As such, thanks to corollary \ref{corollary:girsanov}, there is a Brownian motion $\widetilde{X}$ such that:
\begin{align*}
    X^{x_0,(\mu)}_t
= & x_0 + \mu t + \int_0^t \nabla_x \log \left( \psi_\mu \left( X^{x_0,(\mu)}_s \right) e^{- \langle \mu, X^{x_0,(\mu)}_s \rangle } \right) ds + \widetilde{X}_t\\
= & x_0 + \int_0^t \nabla_x \log \psi_\mu \left( X^{x_0,(\mu)}_s \right) ds + \widetilde{X}_t
\end{align*}
\end{proof}

\section{Properties of Landau-Ginzburg potentials}
\label{section:properties_LG}

In order to detail the properties of Landau-Ginzburg potentials, we will need to be more explicit than the subsection \ref{subsection:LG_potentials} in the preliminaries. We will describe the reference measure $\omega$ and the superpotential $f_B$ in explicit charts, which we will now introduce. These were constructed and studied in \cite{bib:chh14a}.

At the level of the group picture, there are maps $\varrho^L$, $\varrho^K$ and $\varrho^T$ that associate to every element in $\Bc$ a group element in respectively $U^{w_0}_{>0}$, $C^{w_0}_{>0}$ and $U^{w_0}_{>0}$. If $b \in \Bc$, the totally positive elements $\varrho^L(b)$, $\varrho^K(b)$ and $\varrho^T(b)$ are respectively called the Lusztig, Kashiwara and twisted Lusztig parameter. The knowledge of the highest weight $\lambda = \hw\left( b \right)$ allows to recover $b$ from either of these parameters thanks to maps $b^L_\lambda$, $b^K_\lambda$, $b^T_\lambda$. Figure \eqref{fig:geom_parametrizations} summarizes the situation, with all arrows being diffeomorphisms with explicit expressions given in section 4 of \cite{bib:chh14a}.

\begin{figure}[htp!]
\centering
\begin{tikzpicture}[baseline=(current bounding box.center)]
\matrix(m)[matrix of math nodes, row sep=5em, column sep=6em, text height=3ex, text depth=1ex, scale=2]
{
                    & x \in \Bc(\lambda) &                    \\
 z \in U^{w_0}_{>0} & v \in C^{w_0}_{>0} & u \in U^{w_0}_{>0} \\
};
\path[->, font=\scriptsize] (m-1-2) edge node[above]{$\varrho^L$} (m-2-1);
\path[->, font=\scriptsize] (m-1-2) edge node[right]{$\varrho^K$} (m-2-2);
\path[->, font=\scriptsize] (m-1-2) edge node[above]{$\varrho^T$} (m-2-3);

\draw [->] (m-2-1) to [bend left=30]  node[above, swap]{$b^L_\lambda$} (m-1-2);
\draw [->] (m-2-2) to [bend left=45]  node[auto, swap]{$b^K_\lambda$} (m-1-2);
\draw [->] (m-2-3) to [bend right=30] node[above, swap]{$b^T_\lambda$} (m-1-2);

\draw [->] (m-2-1) to [bend left=0 ]  node[above, swap]{$\eta^{e, w_0}$} (m-2-2);
\draw [->] (m-2-2) to [bend left=30]  node[auto , swap]{$\eta^{w_0, e}$} (m-2-1);
\end{tikzpicture}
\caption{Charts for the highest weight geometric crystal $\Bc(\lambda)$}
\label{fig:geom_parametrizations}
\end{figure}
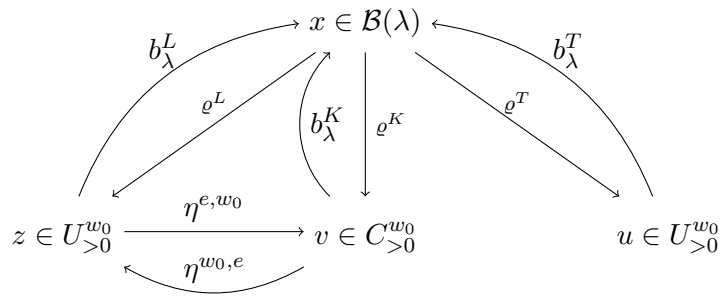

As stated in \cite{bib:chh14a}, such a diagram can be lifted at the level of the geometric path model. Once a choice of reduced word ${\bf i} \in R(w_0)$, it allows to read directly the geometric string coordinates $x_{-\bf i}^{-1} \circ \varrho^K$ (resp. the Lusztig coordinates $x_{\bf i}^{-1} \circ \varrho^L$) from a path, thanks to a map $\varrho_{\bf i}^K$ (resp. $\varrho_{\bf i}^L$). The projection $p$ present in figure \ref{fig:parametrizations_diagram} bridges between the path model and its group picture.

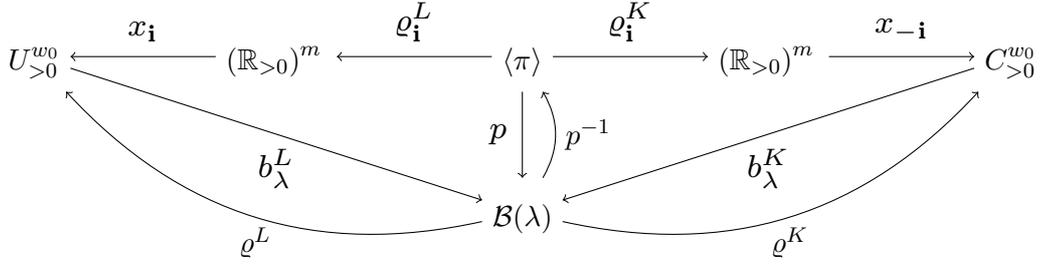
\begin{figure}[htp!]
\centering
\begin{tikzpicture}[baseline=(current bounding box.center)]
\matrix(m)[matrix of math nodes, row sep=3em, column sep=5em, text height=3ex, text depth=1ex, scale=1.2]
{ 
U_{>0}^{w_0} & \left( \R_{>0} \right)^m & \langle \pi \rangle    & \left( \R_{>0} \right)^m & C_{>0}^{w_0}\\
             &                          & \Bc(\lambda)\\
};
\path[->, font=\scriptsize] (m-1-3) edge node[above, scale=1.5]{$\varrho_{\bf i}^K$} (m-1-4);
\path[->, font=\scriptsize] (m-1-4) edge node[above, scale=1.5]{$x_{-\bf i}$} (m-1-5);
\path[->, font=\scriptsize] (m-1-3) edge node[above, scale=1.5]{$\varrho_{\bf i}^L$} (m-1-2);
\path[->, font=\scriptsize] (m-1-2) edge node[above, scale=1.5]{$x_{ \bf i}$} (m-1-1);

\path[->, font=\scriptsize] (m-1-3) edge node[left,  scale=1.5]{$p$} (m-2-3);
\draw[->] (m-2-3) to [bend right=30] node[right]{$p^{-1}$} (m-1-3);

\path[->, font=\scriptsize] (m-1-1) edge node[below, scale=1.5]{$b^L_\lambda$} (m-2-3);
\draw[->] (m-2-3) to [bend left=30] node[below]{$\varrho^L$} (m-1-1);
\path[->, font=\scriptsize] (m-1-5) edge node[below, scale=1.5]{$b^K_\lambda$} (m-2-3);
\draw[->] (m-2-3) to [bend right=30] node[below]{$\varrho^K$} (m-1-5);
\end{tikzpicture} 
\caption{Parametrizations for a connected crystal $\langle \pi \rangle$, with $\pi \in \Cc_0([0, T], \afrak)$ and $\lambda = \Tc_{w_0} \pi(T)$}
\label{fig:parametrizations_diagram}
\end{figure}

In this section, we start by giving explicit expressions of the Landau-Ginzburg potential in subsections \ref{subsection:toric_reference_measure} and \ref{subsection:superpotential} by presenting some surprising ingredients involved: the seemingly innocent toric reference measure $\omega$ and the superpotential map $f_B$. The latter has a rich algebraic structure we detail. Then we prove an estimate in coordinates which allows to prove that the Landau-Ginzburg potential is integrable on $\Bc\left( \lambda \right)$.

\subsection{Toric reference measure}
\label{subsection:toric_reference_measure}

Let us start with a simple object:
\begin{definition}[The measure $\omega_{toric}$]
 Define the measure $\omega_{toric}$ on $\R_{>0}^m$ by:
 $$ \omega_{toric} = \prod_{j=1}^m \frac{dt_j}{t_j}$$
\end{definition}
\begin{rmk}
 Notice that $\omega_{toric}$ is nothing but the flat measure in logarithmic coordinates, or the Haar measure on the multiplicative torus $\R_{>0}^m$.
\end{rmk}

$\omega_{toric}$ has the remarkable property that it is invariant under changes of parametrization for both the Lusztig and Kashiwara varieties.
\begin{thm}[Extension of \cite{bib:GLO1} lemma 3.1 and \cite{bib:Rietsch07} theorem 7.2]
 \label{thm:omega_invariance}
 For all reduced words ${\bf i}$ and ${\bf i'}$ in $R(w_0)$, the image measure of $\omega_{toric}$ through the map $x_{ \bf i'}^{-1} \circ x_{ \bf i}$ (resp. $x_{ -\bf i'}^{-1} \circ x_{ -\bf i}$) is itself. Meaning that if the following change of variables hold:
$$ \left( t_1', t_2', \dots, t_m' \right) = x_{ \bf i'}^{-1} \circ x_{ \bf i}\left(t_1, \dots, t_m \right)$$
$$ \left( c_1', c_2', \dots, c_m' \right) = x_{-\bf i'}^{-1} \circ x_{-\bf i}\left(c_1, \dots, c_m \right)$$
Then:
$$ \prod_{j=1}^m \frac{dt_j'}{t_j'} = \prod_{j=1}^m \frac{dt_j}{t_j}$$
$$ \prod_{j=1}^m \frac{dc_j'}{c_j'} = \prod_{j=1}^m \frac{dc_j}{c_j}$$
\end{thm}
\begin{proof}
Invariance with respect to changes of parametrizations in both coordinate systems are equivalent. Indeed, if we write:
$$\left( x_{-i_1}(c_1) \dots x_{-i_m}(c_m) \right)^T = c_1^{-\alpha_{i_1}^\vee} \dots c_j^{-\alpha_{i_m}^\vee} x_{i_m}(t_j) \dots x_{i_1}(t_1)$$
Then thanks to lemma 4.7 in \cite{bib:chh14a}, $\left(\log t_j\right)_{1 \leq j \leq m}$ and $\left(\log c_j\right)_{1 \leq j \leq m}$ are related to each other by a linear transformation with matrix $M$. The matrix $M$ is upper triangular with unit diagonal, therefore the transformation has a Jacobian equal to $1$.

We are only left with proving the invariance under $x_{ \bf i'}^{-1} \circ x_{ \bf i}$. By the Matsumoto's lemma, two reduced words ${\bf i}$ and ${\bf i'}$ can be obtained from each other by a sequence of braid moves. As such, it is sufficient to prove the statement for ${\bf i}$ and ${\bf i'}$ reduced words from a root system of rank $2$. This is exactly the computation made in \cite{bib:GLO1} lemma 3.1 for types $A_2$ and $B_2$ and in \cite{bib:Rietsch07} theorem 7.2 for all types. A proof without any computation comes as a side product of a previous result. Recall that the exit law of the hypoelliptic Brownian motion motion is given in terms of a $\Gamma_\mu$ random variable. We have thanks to equation \eqref{eq:explicit_Gamma_mu}:
\begin{align}
\label{eq:gamma_mu_law_in_coordinates}
  \P\left( x_{\bf i}^{-1}\left( \Gamma_\mu \right) \in {\bf dt}\right)
= & 
   \frac{1}{\prod_{\beta \in \Phi^+} \Gamma\left( \langle \beta^\vee, \mu \rangle \right) }
   e^{-\sum_{j=1}^m t_j }
   \prod_{j=1}^m t_j^{\langle \mu, \beta_j^{\bf i} \rangle}   
   \prod_{j=1}^m \frac{dt_j}{t_j}
\end{align}

Because we are dealing with an intrinsic object, the description in coordinates yields the same measures. Therefore, if $x_{\bf i}({\bf t}) = x_{\bf i'}({\bf t'}) = u$, by equating the previous probability measure for both parametrizations:
$$ e^{-\sum_{j=1}^m t_j }
   \prod_{j=1}^m t_j^{\langle \mu, \beta_j^{\bf i} \rangle}   
   \prod_{j=1}^m \frac{dt_j}{t_j}
 = 
   e^{-\sum_{j=1}^m t_j' }
   \prod_{j=1}^m (t_j')^{\langle \mu, \beta_j^{\bf i'} \rangle}   
   \prod_{j=1}^m \frac{dt_j'}{t_j'}
$$
We are done upon noticing that
$$ \chi(u) = \sum_{j=1}^m t_j = \sum_{j=1}^m t_j'$$
and 
$ [u \bar{w}_0]_- = t_j^{\beta_j^{\bf i}} = (t_j')^{\beta_j^{\bf i'}}$
\end{proof}

This allows us to define a measure on $\Bc(\lambda)$ that has virtually the same expression regardless of the chosen parametrization. It is nothing more than the image measure of $\omega_{toric}$ under any of the usual parametrizations of $\Bc(\lambda)$.
\begin{thm}
 \label{thm:omega_invariance_on_crystal}
 There is a unique measure on $\Bc(\lambda)$ denoted by $\omega_{\Bc(\lambda)}$ such that, for every reduced word ${\bf i} \in R(w_0)$, and for $x \in \Bc(\lambda)$ being given in either of the coordinates:
$$\varrho^K(x) = x_{\bf -i }\left( c_1 , \dots, c_m \right)$$
$$\varrho^L(x) = x_{\bf  i }\left( t_1,  \dots, t_m \right)$$
$$\varrho^T(x) = x_{\bf  i }\left( t_1', \dots, t_m'\right)$$
Then:
$$\omega(dx) = \prod_{j=1}^m \frac{dc_j}{c_j} = \prod_{j=1}^m \frac{dt_j}{t_j} = \prod_{j=1}^m \frac{dt_j'}{t_j'} $$

Moreover, $\omega$ is invariant with respect to crystal actions $e^._\alpha, \alpha \in \Delta$, meaning that:
$$ \forall c \in \R, \forall \alpha \in \Delta, \int_{\Bc(\lambda)} \varphi(x) \omega_{\Bc(\lambda)}( dx )
= \int_{\Bc(\lambda)} \varphi(e^c_\alpha \cdot x) \omega_{\Bc(\lambda)}( dx ) $$
\end{thm}
\begin{notation}
When the choice of highest weight crystal is clear from context, we will simply write $\omega$ instead of $\omega_{\Bc(\lambda)}$
\end{notation}
\begin{proof}
Fix a reduced word ${\bf i} \in R(w_0)$ and an element $x \in \Bc(\lambda)$. The toric reference measure on Lusztig coordinates is transported by the Sch\"utzenberger involution to the toric measure on twisted Lusztig coordinates as:
$$ \varrho^T\left( S(x) \right) = S\left( \varrho^L(x) \right)$$
giving the equality:
$$
\int_{\R_{>0}^m} \varphi \circ b_\lambda^{L} \circ x_{\bf i^*}\left( t_m, \dots, t_1 \right) \prod_{j=1}^m \frac{dt_j}{t_j}
=
\int_{\R_{>0}^m} \varphi \circ b_\lambda^{T} \circ x_{\bf i  }\left( t_1, \dots, t_m \right) \prod_{j=1}^m \frac{dt_j}{t_j}
$$
Notice that the change in order and reduced words. However, because of theorem \ref{thm:omega_invariance}, it does not matter hence the equality $\prod_{j=1}^m \frac{dt_j}{t_j} = \prod_{j=1}^m \frac{dt_j'}{t_j'}$.

Now, the parameters 
$$v = \varrho^K(x) = \circ x_{-\bf i^{op}}\left( c_m, \dots, c_1 \right)$$
$$u = \varrho^T(x) = \circ x_{ \bf i     }\left( t_1, \dots, t_m \right)$$
are linked by the transform (Section 4 of \cite{bib:chh14a})
$$ u = e^{-\lambda} v^T [v^T]_0^{-1} e^\lambda $$
which yields a monomial change of variable between $\left( c_1, \dots, c_m \right)$ and $\left( t_1, \dots, t_m \right)$ that preserves the toric measure. Hence:
$$
\int_{\R_{>0}^m} \varphi \circ b_\lambda^{K} \circ x_{-\bf i^{op}}\left( c_m, \dots, c_1 \right) \prod_{j=1}^m \frac{dc_j}{c_j}
=
\int_{\R_{>0}^m} \varphi \circ b_\lambda^{T} \circ x_{ \bf i}\left( t_1, \dots, t_m \right) \prod_{j=1}^m \frac{dt_j}{t_j}
$$

Invariance with respect to crystal actions comes from the fact that if $x \in \Bc(\lambda)$ has ${\bf i}$-Lusztig coordinates
$$ \left( t_1, t_2, \dots, t_m \right) $$
then $e^c_\alpha \cdot x$, with $\alpha = \alpha_{i_1}$, has Lusztig coordinates (proposition 8.18 in \cite{bib:chh14a}):
$$ \left( e^c t_1, t_2, \dots, t_m \right) $$
\end{proof}

\subsection{Algebraic structure of \texorpdfstring{$f_B$}{the superpotential}}
\label{subsection:superpotential}

\begin{notation}
Let $x$ be an element in $\Bc(\lambda)$ with $u = \varrho^T(x)$. Define the twist map as in \cite{bib:BZ97} by:
$$ \eta_{w_0}( u ) = [\bar{w}_0^{-1} u^T]_+$$
The main obstruction for $f_B$ to have a simple expression is this twist.
\end{notation}

Using the formulas in proposition 4.17 of \cite{bib:chh14a} we have the following semi-explicit expressions for the superpotential:
\begin{proposition}[Semi-explicit expressions in coordinates]
\label{proposition:semi_explicit_expressions}
Let $x \in \Bc(\lambda)$ and:
$$ v = \varrho^K(x) = x_{-\bf i}\left( c_1, \dots, c_m \right)$$
$$ u = \varrho^T(x) = x_{ \bf i}\left( t_1, \dots, t_m \right)$$
Then:
\begin{align*}
f_B(x) = & \chi \circ S \circ \iota \left( e^{-\lambda} [\bar{w}_0^{-1} u^T ]_+ e^{\lambda}\right) + \chi(u)
       =   \chi\left( e^{-\lambda} \eta_{w_0}(u) e^{\lambda}\right) + \sum_{j=1}^m t_j\\
       = & \chi\left( \eta^{w_0, e}(v) \right) + \chi\left(e^{-\lambda} v^T [v^T]_0^{-1} e^\lambda \right)
       =   \chi\left( \eta^{w_0, e}(v) \right) + \sum_{k=1}^m e^{-\alpha_{i_k}(\lambda)} c_k^{-1} \prod_{j=k+1}^m c_j^{-\alpha_{i_k}(\alpha_{i_j}^\vee)}
\end{align*}
\end{proposition}

The map $f_B$ has a nice and deep algebraic structure. A first flavour that is sufficient for our needs is:
\begin{thm}
\label{thm:superpotential_structure}
If $x \in \Bc(\lambda)$ with twisted Lusztig parameter
$$u = \rho^T(x) = x_{ \bf i}\left( t_1, \dots, t_m \right)$$
then $f_B(x)$ is a Laurent polynomial in the variables $\left( t_1, \dots, t_m \right)$ with positive coefficients.
\end{thm}

Before giving a proof, let us observe that $f_B$ has an expression in term of generalized minors:
\begin{lemma}[ Variant of corollary 1.25 in \cite{bib:BK06} ]
 For $x \in \Bc(\lambda)$, with twisted Lusztig parameter $u = \varrho^T(x)$, we have:
$$ f_B\left( x \right) = \sum_{\alpha \in \Delta} e^{-\alpha(\lambda)} \frac{ \Delta_{s_\alpha \omega_\alpha, w_0 \omega_\alpha}( u )}{\Delta_{\omega_\alpha, w_0 \omega_\alpha}( u )} + \Delta_{\omega_\alpha, s_\alpha \omega_\alpha}(u)$$ 
\end{lemma}
\begin{proof}
 The character $\chi_\alpha$ can also be expressed as a minor (\cite{bib:BK06} relation 1.8):
$$ \forall u \in U^{w_0}_{>0}, \chi_\alpha(u) = \Delta_{\omega_\alpha, s_\alpha \omega_\alpha}(u)$$
Therefore, if (proposition 4.17 in \cite{bib:chh14a}):
$$x = b_\lambda^T(u) = S \circ \iota\left( e^{-\lambda} [ \bar{w}_0^{-1} u^T ]_+ e^{\lambda}\right) \bar{w}_0 e^\lambda u$$
Using, $\chi \circ S \circ \iota = \chi$, we have:
\begin{align*}
f_B(x) & =  \chi \circ S \circ \iota\left( e^{-\lambda} [ \bar{w}_0^{-1} u^T ]_+ e^{\lambda}\right) + \chi(u)\\
& = \sum_{\alpha \in \Delta} e^{-\alpha(\lambda)} \chi_\alpha( [ \bar{w}_0^{-1} u^T ]_+ ) + \chi_\alpha(u)\\
& = \sum_{\alpha \in \Delta} e^{-\alpha(\lambda)} \Delta_{\omega_\alpha, s_\alpha \omega_\alpha}( [ \bar{w}_0^{-1} u^T ]_+ ) + \Delta_{\omega_\alpha, s_\alpha \omega_\alpha}(u)
\end{align*}
Moreover:
\begin{align*}
\Delta_{\omega_\alpha, s_\alpha \omega_\alpha}( [ \bar{w}_0^{-1} u^T ]_+ ) & = \Delta^{\omega_\alpha}( [ \bar{w}_0^{-1} u^T ]_+ \bar{s}_\alpha )\\
& = \Delta^{\omega_\alpha}( [ \bar{w}_0^{-1} u^T ]_0^{-1} [ \bar{w}_0^{-1} u^T ]_{0+} \bar{s}_\alpha )\\
& = \frac{ \Delta^{\omega_\alpha}( \bar{w}_0^{-1} u^T \bar{s}_\alpha )}
         { \Delta^{\omega_\alpha}( \bar{w}_0^{-1} u^T )}\\
& = \frac{ \Delta^{\omega_\alpha}( \bar{s}_\alpha^{-1} u \bar{w}_0 )}
         { \Delta^{\omega_\alpha}( u \bar{w}_0 )}
\end{align*}
Hence the result.
\end{proof}

\begin{proof}[Proof of theorem \ref{thm:superpotential_structure}]
Thanks to the previous lemma, all we need to know is that
$$ \Delta_{\omega_\alpha, s_\alpha \omega_\alpha}(u) $$
$$\frac{ \Delta_{s_\alpha \omega_\alpha, w_0 \omega_\alpha}( u )}{\Delta_{\omega_\alpha, w_0 \omega_\alpha}( u )}$$
is a Laurent polynomial with positive coefficients in the variables $t_j$. The first one is easy to deal with as:
$$ \Delta_{\omega_\alpha, s_\alpha \omega_\alpha}(u) = \chi(u) = \sum_{\alpha_{i_j} = \alpha} t_j$$
For the second one, using \cite{bib:BZ01} theorem 5.8, each of the minors $\Delta_{s_\alpha \omega_\alpha, w_0 \omega_\alpha}( u )$ and  $\Delta_{\omega_\alpha, w_0 \omega_\alpha}( u )$ are linear combinations of monomials with positive coefficients. The latter has only one monomial term by applying corollary 9.5 in \cite{bib:BZ01} (with $u=e$, $w=w_0$ and $\gamma = w_\alpha$).
\end{proof}

\subsubsection{Examples in rank 2}
Let $x$ be an element in $\Bc(\lambda)$ with $u = \varrho^T(x)$. For each classical Cartan-Killing type, we specify a reduced expression ${\bf i}$ for $w_0$ that gives rise to a parametrization of $u \in U^{w_0}_{>0}$:
$$ u = x_{i_1}\left( t_1 \right) \dots x_{i_m}\left( t_m \right)$$
We give explicit expressions in term of the $t_j$ variables for $f_B(x)$, while computing as an intermediary step the twist:
$$ \eta_{w_0}( u ) = [\bar{w}_0^{-1} u^T]_+$$

\begin{itemize}
 \item $A_2$: $w_0 = s_1 s_2 s_1$ $u = x_1(t_1) x_2(t_2) x_1(t_3)$
       \begin{align}
       \label{eq:example_twist_A2}
       \eta_{w_0}\left( u \right) & = x_1\left( \frac{1}{t_1\left(1 + \frac{t_1}{t_3}\right)} \right)
                                      x_2\left( \frac{1 + \frac{t_1}{t_3} }{t_2} \right)
                                      x_1\left( \frac{1}{t_1 + t_3} \right) 
       \end{align}
       $$ f_B(x) = t_1 + t_2 + t_3 + e^{-\alpha_1(\lambda)} \frac{1}{t_1} + e^{-\alpha_2(\lambda)}\left( \frac{1}{t_2} + \frac{t_1}{t_2 t_3} \right)$$
 \item $B_2$: $w_0 = s_1 s_2 s_1 s_2$ $u = x_1(t_1) x_2(t_2) x_1(t_3) x_2(t_4)$
       \begin{align}
       \label{eq:example_twist_B2}
       \eta_{w_0}\left( u \right) & = x_1\left( \frac{1}{t_1\left( 1 + \frac{t_1}{t_3}\left(1 + \frac{t_2}{t_4} \right)^2\right)} \right)
                                      x_2\left( \frac{\left(1 + \frac{t_1}{t_3}\left(1 + \frac{t_2}{t_4} \right)^2\right)}{t_2\left(1 + \frac{t_2}{t_4}\right)} \right) \\
                                  & \quad \cdot x_1\left( \frac{ \left(1 + \frac{t_2}{t_4} \right)^2}{t_1 \left(1 + \frac{t_2}{t_4} \right)^2 + t_3} \right)
                                      x_2\left( \frac{1}{t_2 + t_4} \right)
       \end{align}
       $$ f_B(x) = t_1 + t_2 + t_3 + t_4 + e^{-\alpha_1(\lambda)} \frac{1}{t_1} + e^{-\alpha_2(\lambda)}\left( \frac{1}{t_2} + \frac{t_1 t_2}{t_3 t_4^2} + \frac{t_1}{t_3 t_4}\right)$$
 \item $C_2$: $w_0 = s_1 s_2 s_1 s_2$ $u = x_1(t_1) x_2(t_2) x_1(t_3) x_2(t_4)$
       \begin{align}
       \label{eq:example_twist_C2}
       \eta_{w_0}\left( u \right) & = x_1\left( \frac{1}{t_1\left( 1 + \frac{t_1}{t_3}\left(1 + \frac{t_2}{t_4} \right)\right)} \right)
                                      x_2\left( \frac{\left(1 + \frac{t_1}{t_3}\left(1 + \frac{t_2}{t_4} \right)\right)^2}{t_2\left(1 + \frac{t_2}{t_4}\right)} \right) \\
                                  & \quad \cdot x_1\left( \frac{ \left(1 + \frac{t_2}{t_4} \right)}{t_1 \left(1 + \frac{t_2}{t_4} \right) + t_3} \right)
                                      x_2\left( \frac{1}{t_2 + t_4} \right)
       \end{align}
       $$ f_B(x) = t_1 + t_2 + t_3 + t_4 + e^{-\alpha_1(\lambda)} \frac{1}{t_1} + e^{-\alpha_2(\lambda)}\left( \frac{1}{t_2} + \frac{2 t_1 }{t_2 t_3} + \frac{t_1^2}{t_2 t_3^2} + \frac{t_1^2}{t_3^2 t_4}\right)$$
\end{itemize}

\subsubsection{Link to cluster algebras}
In the examples of the previous subsection, we witness the so-called Laurent phenomenon: When computing the characters $\chi_\alpha \circ \eta_{w_0} (u)$ that are a priori just rational expression in the variables $(t_1, \dots, t_m)$, many simplifications occur and we end up with a Laurent polynomial with positive coefficients.

The Laurent phenomenon is a characteristic of cluster algebras. A cluster algebra is a commutative algebra with a specific set of chosen generators called clusters. Here we are concerned with the coordinate algebra of the double Bruhat cell $G^{w_0, e} := B^+ w_0 B^+ \cap B$, $\C[G^{w_0,e}]$, which has the structure of a cluster algebra (Theorem 2.10 \cite{bib:BFZ05}). 

It is well known that the $t_j$ are linked to minors in $x$ from a specific 'cluster' $\Delta({\bf i_0})$ via an invertible monomial transformation (\cite{bib:BZ97}, \cite{bib:BZ01}). The theory of cluster algebras indicates that every minor is a Laurent polynomial in the variables of a previously fixed cluster. The fact that those Laurent polynomials have positive coefficients is still a quite open conjecture.  The link with cluster algebras is quite clear at this point: the minors appearing in the superpotential are variables in a cluster than can be obtained via seed mutation of the cluster $\Delta\left({ \bf i_0 }\right)$. Laurent phenomenon and coefficient's positivity is expected.

Therefore, theorem \ref{thm:superpotential_structure} is not a suprise given that $f_B$ can be expressed in term of generalized minors. We were able to prove it without any reference to the general theory of cluster algebras because the minors involved in our situation were not very complicated. A complete understanding of the underlying cluster algebra would provide more explicit versions of theorem \ref{thm:superpotential_structure}.

\subsubsection{Computations using the geometric path model}
The superpotential $f_B$ has the following expression in term of the geometric path model:
\begin{lemma}
\label{lemma:superpotential_path_model}
Let $x \in \Bc(\lambda)$ and $\pi \in \Cc_0\left( [0, T], \afrak \right)$. If $x$ has twisted Lusztig parameter $u = \varrho^T(x) = x_{\bf i}\left( t_1, \dots, t_m\right)$, while $\pi$ has usual Lusztig parameters $\varrho^L_{\bf i}(\pi) = (t_1, \dots, t_m)$, we have:
$$ f_B(x) = \sum_{j=1}^m t_j + \sum_{\alpha \in \Delta} e^{-\alpha(\lambda)} \int_0^T e^{-\alpha\left(\pi(s)\right)}ds$$
\end{lemma}
\begin{proof}
Via figure \ref{fig:parametrizations_diagram}, $N_T(\pi) = [ p(\pi) ]_- = [u \bar{w}_0]_-$. Hence:
\begin{align*}
\int_0^T e^{-\alpha(\pi)} & = \chi_\alpha^-\left( N_T(\pi) \right)\\
& = \chi_\alpha^-\left( [u \bar{w}_0]_- \right)\\
& = \chi_\alpha\left( [\bar{w}_0^{-1} u^T]_+ \right)
\end{align*}
Recalling that:
$$ f_B(x) = \chi(u) + \sum_{\alpha} e^{-\alpha(\lambda)} \chi_\alpha\left( [\bar{w}_0^{-1} u^T]_+ \right)$$
finishes the proof.
\end{proof}

The geometric path model allows the computation of minors using integration by parts, while keeping the positivity property obvious. We illustrate this claim by an explicit computation in the $A_2$ type. Choose ${\bf i} = (1,2,1)$ and consider as in the previous lemma a path $\pi \in \Cc_0\left( [0,T], \afrak \right)$ such that:
$$ \varrho^L_{\bf i}(\pi) = (t_1, t_2, t_3)$$
Then:
$$ \int_0^T e^{-\alpha_1(\pi)} = \frac{1}{t_1}$$
$$ \int_0^T e^{-\alpha_2(\pi)} = \frac{1}{t_2} + \frac{t_1}{t_2 t_3}$$
\begin{proof}
The Lusztig coordinates $(t_1, t_2, t_3)$ can be recovered from the path $\pi$ as follows (subsection 8.2 of \cite{bib:chh14a}). Write for $j=1,2,3$:
$$ t_j = \frac{1}{\int_0^T e^{-\alpha_{i_j}(\eta_{j-1}) } }$$
$$ \eta_{j-1} = T_{x_{\alpha_{i_j}(t_j)}} \eta_j\ ,$$
with the convention $\eta_0 = \pi$.

While the first identity is immediate, the second one needs a little more work.
\begin{align*}
  & \int_0^T e^{-\alpha_2(\pi)}\\
= & \int_0^T ds e^{-\alpha_2\left( \eta_1(s) \right)}\left( 1 + t_1 \int_0^s e^{-\alpha_1(\eta_1) } \right)^{-\alpha_2(\alpha_1^\vee)}\\
= & \int_0^T e^{-\alpha_2(\eta_1)} + t_1 \int_0^T ds e^{-\alpha_2\left( \eta_1(s) \right)} \int_0^s e^{-\alpha_1(\eta_1) }\\
= & \frac{1}{t_2} + t_1 \int_0^T ds e^{-\alpha_2\left( \eta_1(s) \right)} \int_0^s e^{-\alpha_1(\eta_1) }
\end{align*}
Moreover, using an integration by parts and the fact that $\int_0^T e^{-\alpha_2(\eta_2)} = \infty$:
\begin{align*}
  & \int_0^T ds e^{-\alpha_2\left( \eta_1(s) \right)} \int_0^s e^{-\alpha_1(\eta_1) }\\
= & \frac{1}{t_2} \int_0^T -\frac{d}{ds}\left( \frac{1}{1+t_2 \int_0^s e^{-\alpha_2(\eta_2)}} \right) ds \int_0^s e^{-\alpha_1(\eta_1) }\\
= & \frac{1}{t_2} \left[ \frac{-\int_0^s e^{-\alpha_1(\eta_1) }}{1+t_2 \int_0^s e^{-\alpha_2(\eta_2)}} \right]_0^T + \frac{1}{t_2} \int_0^T ds \frac{e^{-\alpha_1(\eta_1(s)) }}{1+t_2 \int_0^s e^{-\alpha_2(\eta_2)}}\\
= & \frac{1}{t_2} \int_0^T ds \frac{e^{-\alpha_1(\eta_1(s)) }}{1+t_2 \int_0^s e^{-\alpha_2(\eta_2)}}\\
= & \frac{1}{t_2} \int_0^T e^{-\alpha_1(\eta_2) }\\
= & \frac{1}{t_2 t_3}
\end{align*}
\end{proof}

\subsection{Existence and uniqueness of minimum on \texorpdfstring{$\Bc(\lambda)$}{a highest weight crystal}}
Because of the following theorem, $f_B$ deserves the name of potential as it behaves like a potential well on $\Bc(\lambda) \approx \R_{>0}^m$: level sets are compact.

\begin{thm}[\cite{bib:Rietsch11} proposition 11.3]
\label{thm:superpotential_well}
For $M>0$, consider the set:
$$ K_M(\lambda) = \left\{ x \in \Bc(\lambda) \ | \ f_B(x) \leq M \right\}$$
If $M$ is large enough, $K_M(\lambda)$ is a non-empty compact set.
\end{thm}
\begin{proof}
First parametrize $x \in K_M(\lambda)$ by $t = (t_1, \dots, t_m) \in \R_{>0}^m$ such that:
$$ x_{\bf i}(t_1, \dots, t_m) = u = \varrho^T(x)$$
Then:
\begin{align*}
f_B(x) & = \chi\left( e^{-\lambda} \eta_{w_0}(u) e^{\lambda} \right) + \chi(u)\\
       & = \chi\left( e^{-\lambda} \eta_{w_0}(u) e^{\lambda} \right) + \sum_{j=1}^m t_j
\end{align*}
Clearly, the condition $x \in K_M(\lambda)$ implies $t_j \leq M$ for $j=1, \dots, m$. All we need is to prove that the components of $t$ are bounded away from zero.

Here we can produce two arguments, the first one is the geometric argument produced by Rietsch. We give a quick sketch. Extend $\eta_{w_0}$ to the totally non-negative part of the flag manifold in the following way:
$$
\begin{array}{cccc}
\eta_{w_0}: & \left(G/B\right)_{\geq 0} & \rightarrow & \left( B \backslash G\right)_{\geq 0} \\
            &  u B                      & \mapsto     & B w_0 u^T
\end{array}
$$
It is straightforward that the extended $\eta_{w_0}$ maps $B^+ B \cap B w B$ to $B w_0 B \cap B w_0 w^{-1} B^+$. The idea is that if some of the $t_j$ go to zero, our group element $u B$ leaves the cell
$$U^{w_0}_{>0} B \subset B^+ B \cap B w_0 B$$
and exits to a cell of type $U^{w}_{>0} B$, $\ell\left( w \right) < \ell\left( w_0 \right)$. If:
$$ B u' = B x_{ \bf i}\left(t_1', \dots, t_m'\right) = \eta_{w_0}(u)$$
then, as some of the $t_j$ go to zero, $B u'$ heads to $B w_0 B \cap B w_0 w^{-1} B^+$. However, in order to reach it, some of the parameters $t_j'$ need to go infinity.

The second argument is based on our geometric path model and does not require precise knowledge of parametrizations of the totally non-negative part of the flag manifold. Using lemma \ref{lemma:superpotential_path_model}, if $\pi \in \Cc_0([0, T], \afrak)$ is a path with Lusztig parameters $\left(t_1, \dots, t_m\right)$, then:
$$ f_B(x) = \sum_{j=1}^m t_j + \sum_{\alpha \in \Delta} e^{-\alpha(\lambda)} \int_0^T e^{-\alpha\left(\pi(s)\right)}ds$$
As some of the $t_j$ go to zero, the path $\pi$ converges to $\eta \in \Cc_0([0, T[, \afrak)$, an extended path of type $\ell(w) > \ell(e)$ (see definition 7.18 in \cite{bib:chh14a}). Hence, because of the divergent behavior at the endpoint, there is an $\alpha$ such that:
$$ \int_0^T e^{-\alpha(\pi)} = \infty$$
This cannot happen on $K_M(\lambda)$, and the $t_j$ are indeed bounded away from zero.
\end{proof}

\begin{corollary}
$f_B$ reaches a minimum inside of $\Bc(\lambda)$
\end{corollary}

The following answers the uniqueness question raised by Rietsch in \cite{bib:Rietsch11}:
\begin{thm}
 \label{thm:superpotential_minimum}
 The superpotential $f_B$ reaches its minimum on $\Bc(\lambda)$ at a unique non degenerate point $m_\lambda$. Moreover, it is fixed by the Sch\"utzenberger involution and has zero weight:
$$ \wt( m_\lambda ) = 0$$
\end{thm}
\begin{proof}
We use logarithmic coordinates $\xi_j = \log\left( t_j \right)$. In such coordinates, denoting by ${ \bf \xi } \in \R^m$ the coordinate vector and $\langle ,\rangle$ the usual Euclidian scalar product, $f_B(x)$ has the form (Theorem \ref{thm:superpotential_structure}):
$$ f_B(x) = \sum_{i \in I} c_i e^{ \langle a_i, { \bf \xi} \rangle } $$
where $c_i$ are positive coefficients and $a_i \in \Z^m$ encode exponents. Among the $a_i, i \in I$, there is the Euclidian canonical basis $(e_j)_{1 \leq j \leq m} $ because of the term:
$$ \sum_{j=1}^m t_j = \sum_{j=1}^m t^{ \langle e_j, {\bf \xi} \rangle }$$
Now, it is easy to see that $f_B({\bf \xi})$ is strictly convex as for all $v \in \R^m$:
\begin{align*}
  & \sum_{k,l=1}^m v_k v_l \frac{\partial^2 f}{\partial \xi_k \partial \xi_l}\\
= & \sum_{k,l=1}^m \sum_{i \in I} c_i v_k a_{i,k} v_l a_{i,l} e^{ \langle a_i, \xi\rangle }\\
= & \sum_{i \in I} c_i \langle  v, a_i \rangle^2 e^{ \langle a_i, \xi\rangle }\geq 0
\end{align*}
The hessian matrix is also everywhere non-degenerate: the previous inequality is strict as soon as $v$ is non zero because among the $a_i$, there is the canonical Euclidian basis. Uniqueness for $m_\lambda$ follows.

The Sch\"utzenberger involution $S$ stabilizes $\Bc(\lambda)$ and $f_B \circ S = f_B$. Because of the minimum's uniqueness, one must have $S(m_\lambda) = m_\lambda$, hence $\wt( m_\lambda ) = w_0 \wt( m_\lambda )$ which implies that the weight is zero.

Another way of seeing that $\wt(m_\lambda)=0$ consists in computing the first order condition for a point $x \in \Bc(\lambda)$ being an extremal point:
\begin{align*}
  & f_B\left( e^c_\alpha \cdot x\right) -  f_B\left( x\right)\\
= & \frac{ e^{c} - 1}{ e^{\varepsilon_\alpha(x)} } + \frac{ e^{-c} - 1}{ e^{\varphi_\alpha(x)} } \\
= & c e^{\varepsilon_\alpha(x)} \left( 1 - e^{\varphi_\alpha(x) -\varepsilon_\alpha(x)} \right) + o(c)\\
= & c e^{\varepsilon_\alpha(x)} \left( 1 - e^{\alpha\left( \wt(x) \right)} \right) + o(c)
\end{align*}
If $x$ is critical, then for all $\alpha \in \Delta$, $\alpha\left( \wt(x) \right) = 0$. Hence $\wt(x) = 0$.
\end{proof}

The exact computation of this minimum would be interesting for example in computing the precise behavior of Whittaker function $\psi_\mu(\lambda)$, that we will introduce later, as $\lambda$ goes to '$-\infty$' and the semiclassical limit for the quantum Toda equation.

\subsection{An estimate}
The following estimate is crucial in order to prove the integrability of $e^{-f_B(x)} \omega(dx)$ on $\Bc(\lambda)$.

\begin{thm}
\label{thm:superpotential_estimate}
There are rational exponents $n_j>0$ depending only on the group such that for all $x \in \Bc(\lambda)$:
$$ f_B\left( x \right)  \geq \sum_{j=1}^m t_j + \frac{ e^{-\max_\alpha \alpha(\lambda)} }{ \prod_{j=1}^m t_j^{n_j} }  $$
where $x$ is parametrized as $\varrho^T(x) = x_{ \bf i }\left( t_1, \dots, t_m \right)$. 
\end{thm}
\begin{proof}
If:
$$ u = \varrho^T(x) = x_{ \bf i }\left( t_1, \dots, t_m \right)$$
Then using proposition \ref{proposition:semi_explicit_expressions}:
\begin{align*}
  & f_B(x)\\
= & \chi\left( u\right) + \sum_{\alpha} e^{-\alpha(\lambda)} \chi_\alpha\left( [ \bar{w}_0^{-1} u^T ]_+ \right)\\
\geq & \sum_{j=1}^m t_j + e^{-\max_\alpha \alpha(\lambda)} \chi\left( [ \bar{w}_0^{-1} u^T ]_+ \right)
\end{align*}
Moreover, in terms of the variable $t = \left( t_1, \dots, t_m\right)$,
$$L(t) = \chi\left( [ \bar{w}_0^{-1} u^T ]_+ \right)$$
is a Laurent polynomial with positive integer coefficients (Theorem \ref{thm:superpotential_structure}). We write it as:
$$ L(t) = \sum_{i \in I} c_i \frac{1}{t^{a_i}}$$
Here $I$ is an index set, $a_i = \left( a_i^1, \dots, a_i^m \right) \in \Z^m$ for $i \in I$ are exponent vectors and $c_i \in \N^*, i \in I$ are the Laurent polynomial's coefficients. We use the notation $t^a := \prod_{j=1}^m t_j^{a^j}$ for $a \in \Z^m$.

In order to prove the theorem, we will focus on the lattice cone
$$\Cc(\N) := \left\{ \sum_{i \in I} \lambda_i a_i, \lambda_i \in \N \right\}$$
and prove that:
\begin{align}
\label{eq:lattice_cone_nonempty}
\Cc(\N) \cap \left(\N^*\right)^m & \neq \emptyset  
\end{align}
Once this result obtained, pick an $m$-tuple $v \in \Cc(\N) \cap \left(\N^*\right)^m$ and call $M = \sum v_j \in \N^*$ the sum of its components. Then, for $t \in \left(\R_+^*\right)^m$:
\begin{align*}
     & L(t)^M\\
=    & \sum_{ i_1, i_2, \dots, i_M } c_{i_1} \dots c_{i_M}  \frac{1}{t^{a_{i_1} \dots a_{i_M}}}\\
\geq & \sum_{ i_1, i_2, \dots, i_M } \frac{1}{t^{a_{i_1} + \dots + a_{i_M}}}\\
=    & \sum_{ a \in \Cc(\N), \sum_{j=1}^m a^j = M} \frac{1}{t^a}\\
\geq & \frac{1}{t^v}
\end{align*}
Letting $n_j = \frac{v_j}{M} \in \Q_+^*$ will finish the proof.

Now, let us go back to proving identity \eqref{eq:lattice_cone_nonempty}. For the purpose of using a density argument, define the convex cones:
$$\Cc(\Q_+) := \left\{ \sum_{i \in I} \lambda_i a_i, \ \lambda_i \in \Q_+ \right\}$$
$$\Cc(\R_+) := \left\{ \sum_{i \in I} \lambda_i a_i, \ \lambda_i \in \R_+ \right\}$$
The convex cone $\Cc(\R_+)-\R_+^m = \left\{ a-b, a \in \Cc(\R_+), \ b \in \R_+^m \right\}$ cannot entirely lie in a linear half-space. If it was the case, denote by $H$ such a half-space defined by a normal direction $x \in \R^m$:
$$ H := \left\{ y \in \R^m \ | \ \langle y , x \rangle \geq 0 \right\}$$
We have:
$$ \forall i \in I, \ \langle a_i, x \rangle \geq 0$$
And since $-\R_+^m \subset H$, we necessarily have $x_j \leq 0$ for $j=1, \dots, m$. Now, for $\lambda \in \R$:
\begin{align*}
    L\left( e^{\lambda x_1}, \dots, e^{\lambda x_m} \right)
= & \sum_{i \in I } c_i e^{ -\lambda \langle a_i, x \rangle}
\end{align*}
Because $x$ is non zero with $x_j \leq 0$, taking $\lambda \rightarrow \infty$ forces at least one of the components of $t = \left( e^{\lambda x_1}, \dots, e^{\lambda x_m} \right)$ to zero, while $L(t)$ stays bounded. This contradicts theorem \ref{thm:superpotential_well}. We have then proved indeed that the convex cone $\Cc(\R_+)-\R_+^m$ cannot entirely lie in a linear half-space. Moreover, it is well known that the only convex cone in $\R^m$ that is not included in a half-space is $\R^m$, forcing $\Cc(\R_+)-\R_+^m = \R^m$. Therefore, $\Cc(\Q_+)-\Q_+^m$ is dense in $\R^m$ and $\Cc(\Q_+) \cap \left(\Q_+^*\right)^m$ is not empty. This implies the identity \eqref{eq:lattice_cone_nonempty}.
\end{proof}

\subsection{Link between the remarkable exit laws and Landau-Ginzburg potentials}
The following proposition shows that the remarkable $\Gamma_\mu\left( \lambda \right)$ found in section \ref{section:markovian_points} has the same law as the twisted Lusztig parameter corresponding to a canonical random variable on $\Bc(\lambda)$ (see definition \ref{def:canonical_probability_measure}).

\begin{proposition}
\label{proposition:link_with_canonical}
For $\mu \in C$ and $\lambda \in \afrak$:
$$ \Gamma_\mu(\lambda) \stackrel{\Lc}{=} \varrho^{T}\left( C_{w_0 \mu}(\lambda) \right)$$
This extends the definition of $\Gamma_\mu(\lambda)$ to all $\mu \in \afrak$. Moreover, the Whittaker function has indeed the integral representation:
$$ \psi_\mu(\lambda) = \int_{\Bc(\lambda)} e^{ \langle \mu, \wt(x) \rangle - f_B(x) } \omega(dx) $$
\end{proposition}
\begin{proof}
Let ${\bf i} = (i_1, \dots, i_m)$ be a reduced word for $w_0$ and let $(\beta_1^\vee, \dots, \beta_m^\vee)$ be the associated positive coroots enumeration (as in eq. \eqref{eq:positive_root_enumeration}). Let $\varphi$ be a positive measurable function on $U$ that will serve the purpose of test function, $\mu$ in the open Weyl chamber and $\lambda \in \afrak$. By definition \ref{def:generalized_gamma}:
\begin{align*}
          \E\left( \varphi\left(\Gamma_\mu\left( \lambda \right) \right) \right)
= \quad & \frac{b(\mu) e^{\langle \mu, \lambda \rangle} }{\psi_\mu(\lambda)} \E\left( \varphi(\Gamma_\mu) \exp\left( -\chi^-(e^{\lambda} \Theta\left( \Gamma_\mu  \right) e^{-\lambda} ) \right) \right)\\
\stackrel{eq. \eqref{eq:gamma_mu_law_in_coordinates} }{=}
        & \frac{ e^{\langle \mu, \lambda \rangle} }{\psi_\mu(\lambda)} \int_{(\R_{>0})^m} \prod_{j=1}^m \left( \frac{dt_j}{t_j} t_j^{\langle \beta_j^\vee, \mu \rangle} e^{-t_j} \right) \varphi(u) e^{-\chi^-(e^{\lambda} [u  \bar{w}_0]_{-} e^{-\lambda}) }
\end{align*}
where $ u = x_{i_1}(t_1) \dots x_{i_m}(t_m)$.

Now let us reorganize the terms in the previous integral. On the one hand, if $x = b_\lambda^T(u)$ or equivalently $u = \varrho^T(x)$, then thanks to proposition \ref{proposition:semi_explicit_expressions}, we notice that the term in the exponential is:
\begin{align*}
   \sum_{j=1}^m t_j + \chi^-\left( e^{\lambda} [u  \bar{w_0}]_{-} e^{-\lambda} \right)
= & \chi(u) + \chi\left( e^{-\lambda} [\bar{w_0}^{-1} u^T]_{+} e^{\lambda} \right)\\
= & \chi \circ S \circ \iota \left( e^{-\lambda} [\bar{w_0}^{-1} u^T]_{+} e^{\lambda}) \right) + \chi(u)\\
= & f_B\left( x \right)
\end{align*}
On the other hand, using the expression for the weight map on $x$ given in theorem 4.22 in \cite{bib:chh14a}:
\begin{align*}
    e^{\langle \mu, \lambda \rangle} \prod_{j=1}^m \left( t_j^{\langle \beta_j^\vee, \mu \rangle} \right)
= & \exp\left(\langle \mu, \lambda \rangle + \sum_{j=1}^m \langle \beta_j^\vee, \mu \rangle \log t_j \right)\\
= & \exp\left(\langle \mu, \lambda + \sum_{j=1}^m \log(t_j) \beta_j^\vee \rangle \right)\\
= & \exp\left(\langle w_0 \mu, w_0\left( \lambda - \sum_{j=1}^m \log(\frac{1}{t_j}) \beta_j^\vee \right) \rangle \right)\\
= & \exp\left(\langle w_0 \mu,  \wt( x ) \rangle \right)
\end{align*}
In the end, using the parametrizations introduced at the beginning of section \ref{section:properties_LG}:
\begin{align*}
    \E\left( \varphi\left(\Gamma_\mu\left( \lambda \right) \right) \right)
= & \frac{1}{\psi_{\mu}(\lambda)} \int_{\Bc(\lambda)} e^{ \langle w_0 \mu, \wt(x) \rangle - f_B(x) } \varphi \circ \varrho^{T}(x) \omega(dx)\\
= & \E\left( \varphi \circ \varrho^{T}\left( C_{w_0 \mu}(\lambda) \right) \right)
\end{align*}
\end{proof}

\section{Intertwined Markov operators}
\label{section:intertwined_markov_kernels}
With theorem \ref{thm:whittaker_process_g}, we proved that for $\mu \in C$, if $W$ is a standard Brownian motion in $\afrak$ and $\Theta\left( g \right)$ independent following the law of $D_\mu\left( x_0 \right)$, then
$$ X^{x_0}_t = x_0 + T_{g }\left( W^{ (w_0 \mu) } \right) ; t \geq 0$$
is Markovian, what we called the Whittaker process. The results of Rogers and Pitman in \cite{bib:RogersPitman} on Markov functions teach us that there should be an intertwining relation between the semi-groups of Brownian motion on the one hand, and the semi-group of the Whittaker process, using this remarkable law $D_\mu\left( x_0 \right)$. In fact, this is how the extensions of Pitman's theorem in \cite{bib:RogersPitman} and \cite{bib:OConnell} were proven. The only trick is that intertwining is easy to establish, once we know the answer. What we did so far is identifying the right objects. Using intertwining Markov operators, we will strenghten the previous result to all possible drifts $\mu$ and not only for $\mu$ in the Weyl chamber. 

\subsection{Markov functions }
Let us first quickly review the result of Pitman and Rogers on Markov function from \cite{bib:RogersPitman}. Let $S$ and $S_0$ be topological spaces. Let $\phi: S \rightarrow S_0$ be a continuous function. Consider a Markov process $(X_t)_{t \geq 0}$ with state space $S$ and define the process $Y_t=\phi(X_t)$. We are interested in sufficient conditions that insure the Markov property for $Y$.

Of course one can suppose that $\phi$ is surjective by setting $S_0 = \phi(S)$. And clearly, in most cases of interest where $\phi$ is not injective, the inclusion between filtrations $\Fc^{Y} \subset \Fc^{X}$ is strict. Meaning that the observation of $Y$ contains only partial information on $X$. And in order to quantify this information, we need to ``filter'' $X$ through $\Fc^{Y}$.

In the sequel, we denote by $P_t$ the semi-group for $X$, and $Q_t$ the semi-group for $Y$, when it exists. 
$\Phi: \Cc(S_0) \rightarrow \Cc(S)$ is the Markov operator from $S$ to $S_0$ given by $\Phi(f)(x) = f \circ \phi(x)$. It just transports measures on $S$ to their image measure on $S_0$.

A first answer would be Dynkin's criterion, for cases where $Y$ is Markovian for all initial laws of $X$:
\begin{thm}[Dynkin's criterion]
 If there exist a Markov operator $Q$ such that:
$$ \forall t \geq 0, P_t \circ \Phi = \Phi \circ Q_t $$
meaning, in terms of transporting measures, that the following diagram is commutative:
\begin{center}
\begin{tikzpicture}
\matrix(a)[matrix of math nodes, row sep=3em, column sep=3.5em,
text height=1.5ex, text depth=0.25ex]
{ S   & S  \\
  S_0 & S_0 \\ };
\path[->]
(a-1-1) edge node[auto] { $P_t$} (a-1-2)
(a-2-1) edge node[auto] { $Q_t$} (a-2-2)
(a-1-1) edge node[auto] { $\Phi$} (a-2-1)
(a-1-2) edge node[auto] { $\Phi$} (a-2-2);
\end{tikzpicture}
\end{center}
Then $Y$ is a Markov process and its semi-group is $Q_t$.
\end{thm}

\begin{rmk}
In probabilistic terms, the condition that $P_t \circ \Phi(f)(x)$ only depends on $\phi(x)$ translates as saying that the law of $(\phi( X_t ) | X_0 = x)$ only depends of $\phi(x)$. The theorem seems then quite trivial. We wrote it that way to stress the intertwining.
\end{rmk}

Another solution has been formalized in \cite{bib:RogersPitman}. In some cases, if $Y$ starts at $y \in S_0$, it is Markovian only for specific entrance laws $\Kc(y, .)$ on $X$. In such a case, this initial law for $X$ is going to be the ``missing'' information from $\Fc^Y \subsetneq \Fc^X$.\\

Furthermore, at each time, we must ask the missing information to be stationnary in law, otherwise filtering $ X_t | \Fc_t^Y $ will give a fluctuating distribution and will not be able to extract the law of $X_t$ conditionnally to $\Fc^Y$, in such a way that it depends only on $Y_t$. One could speak of a ``Markovian stationary coupling'' or a ``Markovian filtering'' phenomenon, which brings the following equivalent definitions due to Rogers and Pitman:

\begin{thm}
\label{thm:intertwining}
Let $\Kc: \Cc(S) \rightarrow \Cc(S_0)$ be a Markov operator, $X$ a Markov process with semigroup $P_t$ and $Y = \phi(X)$. $Y$ is assumed to start at $y$. The following propositions are equivalent:
\begin{description}
 \item[$(i)$] ( Markovian filtering )
 $$\forall t \geq 0,  \forall f \in \Cc(S), \forall y \in S_0, \E_{ X_0 \sim \Kc(y,.) }\left( f( X_{t} ) | \Fc_{t}^Y \right) = \Kc(f)(Y_t) \quad a.s $$
where the subscript $X_0 \sim \Kc(y,.)$ indicates the initial law for $X$.
 \item[$(ii)$] ( Intertwining operators ) For all $t\geq 0$, $Q_t := \Kc \circ P_t \circ \Phi$ satisfies:
 $$ \Kc \circ \Phi = id_{S_0} $$
 $$ \Kc \circ P_t = Q_t \circ \Kc$$
Meaning, in terms of transporting measures, that the following diagram is commutative:
\begin{center}
\begin{tikzpicture}
\matrix(a)[matrix of math nodes, row sep=3em, column sep=3.5em,
text height=1.5ex, text depth=0.25ex]
{ S   & S   &  \\
  S_0 & S_0 & S_0 \\ };
\path[->]
(a-1-1) edge node[auto] { $P_t$} (a-1-2)
(a-2-1) edge node[auto] { $\Kc$} (a-1-1)
(a-2-2) edge node[auto] { $\Kc$} (a-1-2)
(a-2-1) edge node[auto] { $Q_t$} (a-2-2)
(a-2-2) edge node[auto] { $id$} (a-2-3)
(a-1-2) edge node[auto] { $\Phi$} (a-2-3);
\end{tikzpicture}
\end{center}

\end{description}
In both cases, $Q_t$ is a semi-group and is interpreted as:
$$ Q_t(f)(y) = \E_{ X_0 \sim \Kc(y,.)}\left( f( Y_t ) \right) $$
\end{thm}

\begin{proof}
The semi-group property of $Q_t$ is a consequence of $(ii)$:
\begin{align*}
    Q_{t+s}
= & Q_{t+s} \circ \Kc \circ \Phi\\
= & \Kc \circ P_{t+s} \circ \Phi\\
= & \Kc \circ P_t \circ P_s \circ \Phi\\
= & Q_t \circ \Kc \circ P_s \circ \Phi\\
= & Q_t \circ Q_s \circ \Kc \circ \Phi\\
= & Q_t \circ Q_s
\end{align*}
$(i) \Rightarrow (ii):$ Let $g \in \Cc(S_0)$. By taking $f=g\circ \phi$ in $(i)$ and $t=0$, one gets:
$$ \Kc \circ \Phi (g)(y) = \E_{ X_0 \sim \Kc(y,.) }\left( \Phi (g)( X_{0} ) \right) = \E_{ X_0 \sim \Kc(y,.) }\left( g( Y_{0} ) \right)= g(y)$$
Concerning the second one:
 \begin{align*}
    \Kc \circ P_t (f)(y)
= & \E_{ X_0 \sim \Kc(y, .) }\left( f( X_t ) \right)\\
= & \E_{ X_0 \sim \Kc(y, .) }\left( \E_{ X_0 \sim \Kc(y, .) }\left( f( X_t ) | \Fc_t^Y \right) \right)\\
= & \E_{ X_0 \sim \Kc(y, .) }\left( \Kc(f)(Y_t) \right)\\
= & \Kc \circ P_t \circ \Phi \circ \Kc (f)(y)\\
= & Q_t \circ \Kc (f)( y )
 \end{align*}

$(ii) \Rightarrow (i):$ Consider increasing times $0 \leq s_1 \leq s_2 \dots \leq s_n \leq t$ and test functions $g_1, \dots, g_n$ on $S_0$:
\begin{align*}
  & \E_{ X_0 \sim \Kc(y, .) }\left( g_1(Y_{s_1}) \dots g_n(Y_{s_n}) f(X_t) \right)\\
= & \Kc P_{s_1} \Phi(g_1) P_{s_2-s_1} \Phi(g_2) \dots P_{s_n-s_{n-1}} \Phi(g_n) P_{t-s_n} f (y)\\
= & Q_{s_1} \Kc \Phi(g_1) P_{s_2-s_1} \Phi(g_2) \dots P_{s_n-s_{n-1}} \Phi(g_n) P_{t-s_n} f (y)\\
= & Q_{s_1} g_1 Q_{s_2-s_1} g_2 \dots Q_{s_n-s_{n-1}} g_n \Kc P_{t-s_n} f (y)\\
= & Q_{s_1} g_1 Q_{s_2-s_1} g_2 \dots Q_{s_n-s_{n-1}} g_n Q_{t-s_n} \Kc f (y)\\
= & \E_{ X_0 \sim \Kc(y, .) }\left( g_1(Y_{s_1}) g_2( Y_{s_2} ) \dots g_n( Y_{s_n} ) \Kc(f)(Y_t) \right)
\end{align*}
This proves the Markovian filtering property.
\end{proof}

\begin{thm}[Pitman and Rogers criterion \cite{bib:RogersPitman}]
 \label{thm:rogerspitman_criterion}
 If the equivalent hypotheses previously cited are satisfied, take $X$ with initial law $\Kc(y, .)$ and $Y = \phi(X)$. Then $Y$ is a Markov process starting at $y$ and its semi-group is $Q_t$. 
\end{thm}
\begin{proof}
Let us prove the Markov property:
\begin{align*}
    \E_{X_0 \sim \Kc(y, .)}\left( f( Y_{t+s} ) | \Fc_{s}^Y \right)
= & \E_{X_0 \sim \Kc(y, .)}\left( \E\left( f \circ \phi(X_{t+s}) | \Fc_{s}^X \right)  | \Fc_{s}^Y \right)\\
= & \E_{X_0 \sim \Kc(y, .)}\left( P_t\left( f \circ \phi \right)(X_s)  | \Fc_{s}^Y \right)\\
= & \E_{X_0 \sim \Kc(y, .)}\left( P_t \circ \Phi \left( f \right)(X_s)  | \Fc_{s}^Y \right)
\end{align*}
Using the hypothesis (i) from theorem \ref{thm:intertwining}, we have:
\begin{align*}
    \E_{X_0 \sim \Kc(y, .)}\left( f( Y_{t+s} ) | \Fc_{s}^Y \right)
= & \Kc \circ P_t \circ \Phi \left( f \right)(Y_s)\\
= & Q_t \circ \Kc \circ \Phi \left( f \right)(Y_s)\\
= & Q_t (f)\left( Y_s \right) 
\end{align*}
\end{proof}

\subsection{The canonical measure intertwines the hypoelliptic BM and the highest weight process}

Now, let us specialize the Rogers-Pitman framework of Markov functions to our case. The semi-group for the hypoelliptic Brownian motion $(B_t(W^{(\mu)}), t\geq 0)$ is denoted by $P_t$:
$$ P_t := \exp\left( t \Dc^{(\mu)} \right)$$
Recall that the highest weight process is:
$$\Lambda_t := \hw( B_t(W^{(\mu)}) ) = \Tc_{w_0} W^{(\mu)}_t $$
Finally, define the Markov kernel $\Kc_\mu: \Cc\left( \Bc \right) \rightarrow \Cc(\afrak)$ from $\afrak$ to $\Bc$ by:
$$\forall \varphi \in \Cc( \Bc ),  \Kc_\mu(\varphi)(\lambda) := \E\left( \varphi\left( C_\mu(\lambda) \right) \right)$$
Since the random variable $C_\mu(\lambda)$ is $\Bc(\lambda)$ valued, it is clear that:
$$ \Kc_\mu \circ \hw = id_{\afrak}$$

The following ``Markovian filtering'' holds:
\begin{thm}
\label{thm:markovian_filtering}
Let $x_0 \in \afrak$, $\mu \in \afrak$, $W^{(\mu)}$ a BM in the Cartan subalgebra $\afrak$ and $C_\mu(x_0)$ an independent random variable whose distribution follows the canonical probability measure on $\Bc(x_0)$, with spectral parameter $\mu$.
If:
$$ X^{x_0}_t := \hw\left( C_\mu(x_0) B_t(W^{(\mu)}) \right)$$
and $f: \Bc \longrightarrow \R$ is a bounded function, then:
$$ \E\left( f\left( C_\mu(x_0) B_t(W^{(\mu)}) \right) | \Fc_t^{X^{x_0}}, X^{x_0}_t = x \right) = \E\left( f( C_\mu(x) ) \right)$$
\end{thm}
\begin{proof}
For notational reason, we write $\Fc_t$ instead of $\Fc_t^{X^{x_0}}$. As a first step, let us prove that the theorem for general $\mu \in \afrak$ is a consequence of the case $\mu \in -C$ using a change of probability measure. Assume for now that the result is true for $\mu \in -C$. It is straightforward to check from the definition \ref{def:canonical_probability_measure} that for $\nu \in \afrak$:
$$ \E\left( f\left( C_\mu(x_0) \right) \right) = \frac{\psi_\nu(x_0)}{\psi_\mu(x_0)} \E\left( e^{\langle \wt\left( C_\nu(x_0) \right), \mu - \nu \rangle } f\left( C_\nu(x_0) \right) \right).$$
Using the Girsanov-Cameron-Martin theorem for Brownian motion (\cite{bib:KaratzasShreve91} theorem 5.1), for any functional $F$:
$$ \E\left( F\left( W^{(\mu)}_s; s \leq t \right) \right) = 
   \E\left( e^{ \langle W^{(\nu)}, \mu - \nu \rangle - \half (||\mu||^2 - ||\nu||^2) t }
            F\left( W^{(\nu)}_s; s \leq t \right) \right)$$
Hence, because $C_\mu(x_0)$ and $W^{(\mu)}$ are independent, in the following change of probability, the density is the product of the two previous densities:
$$ \frac{\psi_\nu(x_0)}{\psi_\mu(x_0)} \exp\left( \langle \wt\left( C_\nu(x_0) \right) + W^{(\nu)}, \mu - \nu \rangle - \half (||\mu||^2 - ||\nu||^2) t \right)$$
Using the Bayes formula, we have that, on the set $\{ X_t^{x_0} = x \}$:
\begin{align*}
  & \E\left( f\left( C_\mu(x_0) B_t(W^{(\mu)}) \right) | \Fc_t \right)\\
= & \frac{\E\left( e^{ \langle \wt\left( C_\nu(x_0) \right) + W^{(\nu)}, \mu - \nu \rangle - \half (||\mu||^2 - ||\nu||^2) t } f\left( C_\nu(x_0) B_t(W^{(\nu)}) \right) | \Fc_t \right)}
         {\E\left( e^{ \langle \wt\left( C_\nu(x_0) \right) + W^{(\nu)}, \mu - \nu \rangle - \half (||\mu||^2 - ||\nu||^2) t } | \Fc_t \right)}\\
= & \frac{\E\left( e^{ \langle \wt\left( C_\nu(x_0) B_t(W^{(\nu)}) \right), \mu - \nu \rangle }
                   f\left( C_\nu(x_0) B_t(W^{(\nu)}) \right) | \Fc_t \right)}
         {\E\left( e^{ \langle \wt\left( C_\nu(x_0) B_t(W^{(\nu)}) \right), \mu - \nu \rangle }
                   | \Fc_t \right)}
\end{align*}
Applying the result for $\nu \in -C$, one has:
$$
    \E\left( f\left( C_\mu(x_0) B_t(W^{(\mu)}) \right) | \Fc_t^{X^{x_0}}, X_t^{x_0} = x \right)
=   \frac{\E\left( e^{ \langle \wt\left( C_\nu(x) \right), \mu - \nu \rangle } f\left( C_\nu(x) \right) \right)}
         {\E\left( e^{ \langle \wt\left( C_\nu(x) \right), \mu - \nu \rangle } \right)}
=   \E\left( f\left( C_\mu(x) \right) \right),
$$
which concludes the first part of the proof.

Now, for the second part, let us prove the theorem when $\mu \in -C = w_0 C$. Using proposition \ref{proposition:link_with_canonical}, we write $C_\mu(x_0) = z \bar{w}_0 e^{x_0} g$
where $g \stackrel{\Lc}{=} \Gamma_{w_0 \mu}(x_0)$. Notice that:
\begin{align*}
X^{x_0}_t & = \hw\left( C_\mu(x_0) B_t(W^{(\mu)}) \right)\\
& = \log [ \bar{w}_0^{-1} z \bar{w}_0 e^{x_0} g  B_t(W^{(\mu)}) ]_0\\
& = x_0 + \log [ g  B_t(W^{(\mu)}) ]_0\\
& = x_0 + \log [ g  B_t(W^{(\mu)}) ]_0\\
& = x_0 + T_{g } \left( W^{(\mu)} \right)
\end{align*}
For shorter notations introduce $n= N_t\left( X^{x_0} \right)$ and $x = X_t^{x_0}$. Thus we have:
$$ n e^x = \left[ e^{x_0} g B_t(W^{(\mu)}) \right]_{-0}$$
Then, using the properties of the Gauss decomposition:
\begin{align*}
  & B_t(W^{(\mu)}) \\
= & \left(e^{x_0} g\right)^{-1} e^{x_0} g B_t(W^{(\mu)})\\
= & \left(e^{x_0} g\right)^{-1} n e^x \left[e^{x_0} g B_t(W^{(\mu)})\right]_+
\end{align*}
But since $B_t(W^{(\mu)}) \in B$, we have:
$$ B_t(W^{(\mu)}) = \left[ \left(e^{x_0} g\right)^{-1} n e^x \right]_{-0}$$
And the following decomposition holds:
\begin{align}
e^{x_0} g B_t(W^{(\mu)}) & = e^{x_0} g \left[ \left(e^{x_0} g\right)^{-1} n e^x \right]_{-0}\\
& = n e^x \left[ \left(e^{x_0} g\right)^{-1} n e^x \right]_{+}^{-1}\\
& = n e^x \left[ \left(e^{x_0} g e^{-x_0}\right)^{-1} n e^x \right]_{+}^{-1}\\
& = n e^x \left[ \left( e^{-x} n^{-1} e^{x_0} g e^{-x_0} e^x \right)^{-1} \right]_{+}^{-1}
\end{align}
Therefore:
\begin{align*}
  & C_\mu(x_0) B_t( W^{(\mu)} )\\
= & b_x^T\left( \left[ \bar{w}_0^{-1} C_\mu(x_0) B_t( W^{(\mu)} ) \right]_+ \right)\\
= & b_x^T\left( \left[ e^{x_0} g B_t( W^{(\mu)} ) \right]_+ \right)\\
= & b_x^T\left( \left[ \left( e^{-x} n^{-1} e^{x_0} g e^{-x_0} e^x \right)^{-1} \right]_{+}^{-1} \right)
\end{align*}

Thanks to the conditional representation theorem \ref{thm:conditional_representation}, we know that:
$$ N_\infty(X^{x_0}) = \Theta( e^{x_0} g e^{-x_0} ) $$
Moreover, equation \eqref{eq:N_infty_decomposition_x_0} tells us:
$$ N_\infty(X^{x_0}) = n e^{x} N_\infty(X^{x_0}_{t+.} - X^{x_0}_t) e^{-x}$$
Hence, since $\Theta(ng) = n \Theta(g)$ for $n \in N$:
$$ N_\infty(X^{x_0}_{t+.} - X^{x_0}_t) = \Theta( e^{-x} n^{-1} e^{x_0} g e^{-x_0} e^{x} )$$
And:
$$ C_\mu(x_0) B_t( W^{(\mu)} ) = b_x^T\left( \left[ \Theta^{-1}\left( N_\infty(X^{x_0}_{t+.} - X^{x_0}_t) \right)^{-1} \right]_{+}^{-1} \right) $$
As $ \Theta^{-1}\left( N_\infty(X^{x_0}_{t+.} - X^{x_0}_t) \right) \in U$, we have in the end:
\begin{align}
\label{eq:final_decomposition}
C_\mu(x_0) B_t( W^{(\mu)} ) & = b_x^T \circ \Theta^{-1}\left( N_\infty(X^{x_0}_{t+.} - X^{x_0}_t) \right)
\end{align}

Under our working probability measure $g \stackrel{\Lc}{=} \Gamma_\mu(x_0)$ and $X^{x_0}$ follows the Whittaker process (Theorem \ref{thm:whittaker_process_g}). In the context of subsection \ref{subsection:condition_N_infty}, our working probability measure can be considered of the form $\P^v$. Under the equivalent probability measure $\P$, $g$ has the same law as $\Gamma_\mu$ and $X^{x_0}$ is distributed as a BM with drift $w_0 \mu \in C$.
$$  \frac{d \P^v}{d \P}            = \exp\left( -\chi_-\left( N_\infty(X^{x_0}) \right) \right)
                                     \frac{ b(w_0 \mu) }
                                          { \psi_{w_0 \mu}(x_0) e^{- \langle w_0 \mu, x_0 \rangle }} $$

$$ \frac{ d \P^v}{d \P}_{| \Fc_t } = \exp\left( -\chi_-\left( n \right) \right)
                                     \frac{ \psi_{w_0 \mu}(x  ) e^{- \langle w_0 \mu, x   \rangle }} 
                                          { \psi_{w_0 \mu}(x_0) e^{- \langle w_0 \mu, x_0 \rangle }} $$
Thus, we get the simplification:
\begin{align*}
  & \frac{d \P^v}{d \P} / \frac{ d \P^v}{d \P}_{| \Fc_t }\\
= & \frac{ b(w_0\mu) }{ \psi_{w_0\mu}(x) e^{- \langle w_0 \mu, x \rangle }}
    \frac{ \exp\left( -\chi_-\left( n e^x N_\infty(X^{x_0}_{t+.}-X^{x_0}) e^{-x} \right) \right) }
         { \exp\left( -\chi_-\left( n \right) \right) }\\
= & \frac{ b(w_0\mu) }{ \psi_{w_0\mu}(x) e^{- \langle w_0 \mu, x \rangle }}
    \exp\left( -\chi_-\left( e^x N_\infty(X^{x_0}_{t+.}-X^{x_0}) e^{-x} \right) \right)
\end{align*}
Therefore, on the set $\{ X_t^{x_0} = x \}$, by equation \eqref{eq:final_decomposition} and using the fact that $\E = \P^v$:
\begin{align*}
  & \E\left( f\left( C_\mu(x_0) B_t(W^{(\mu)}) \right) | \Fc_t\right)\\
= & \E\left( f \circ b_x^T \circ \Theta^{-1}\left( N_\infty(X^{x_0}_{t+.} - X^{x_0}_t) \right)
           | \Fc_t \right)\\ 
= & \P^v\left( f \circ b_x^T \circ \Theta^{-1}\left( N_\infty(X^{x_0}_{t+.} - X^{x_0}_t) \right)
             | \Fc_t \right)
\end{align*}
By the Bayes formula:
\begin{align*}
  & \E\left( f\left( C_\mu(x_0) B_t(W^{(\mu)}) \right) | \Fc_t \right)\\
= & \frac{ \P\left( \frac{d \P^v}{d \P}
                    f \circ b_x^T \circ \Theta^{-1}\left( N_\infty(X^{x_0}_{t+.} - X^{x_0}_t) \right)
                    | \Fc_t \right)}
         { \frac{ d \P^v}{d \P}_{| \Fc_t } }\\
= & \frac{ b(w_0\mu) }{ \psi_{w_0\mu}(x) e^{- \langle w_0 \mu, x \rangle }}
    \P\left( e^{-\chi_-\left( e^x N_\infty(X^{x_0}_{t+.}-X^{x_0}) e^{-x} \right)}
             f \circ b_x^T \circ \Theta^{-1}\left( N_\infty(X^{x_0}_{t+.} - X^{x_0}_t) \right)
             | \Fc_t \right)
\end{align*}

Since under $\P$, $X^{x_0}$ is a Brownian motion with drift $w_0 \mu$, we know that $N_\infty\left( X^{x_0}_{t+.} - X^{x_0}_{t} \right)$ is independent from $\Fc_t$ and has the same law as $D_{w_0 \mu}$. In the end:
\begin{align*}
  & \E\left( f\left( C_\mu(x_0) B_t(W^{(\mu)}) \right) | \Fc_t^{X^{x_0}}, X^{x_0}_t = x \right)\\
= & \frac{ b(w_0\mu) }{ \psi_{w_0\mu}(x) e^{- \langle w_0 \mu, x \rangle }}
    \E\left( \exp\left( -\chi_-\left( e^x D_{w_0 \mu} e^{-x} \right) \right)
             f \circ b_x^T \circ \Theta^{-1}\left( D_{w_0 \mu} \right) \right)\\
= & \E\left( f \circ b_x^T \circ \Theta^{-1}\left( D_{w_0 \mu}(x) \right) \right)\\
= & \E\left( f \circ b_x^T \left( \Gamma_{w_0 \mu}(x) \right) \right)
\end{align*}
Proposition \ref{proposition:link_with_canonical} yields the result by giving:
$$ C_\mu(x) \stackrel{\Lc}{=} b_x^T \left( \Gamma_{w_0 \mu}(x) \right)$$
\end{proof}

As a consequence, the condition (i) of theorem \ref{thm:intertwining} is valid with an initial law for the hypoelliptic Brownian motion being $C_\mu(x_0)$. Moreover
$$ Q_t := \Kc_\mu \circ P_t \circ \hw$$
is a semi-group making the following diagram commutative.
\begin{center}
\begin{tikzpicture}
\matrix(a)[matrix of math nodes, row sep=3em, column sep=3.5em,
text height=1.5ex, text depth=0.25ex]
{ \Bc    & \Bc    &  \\
  \afrak & \afrak & \afrak \\ };
\path[->]
(a-1-1) edge node[auto] { $e^{t \Dc^{(\mu)}}$} (a-1-2)
(a-2-1) edge node[auto] { $\Kc_\mu$} (a-1-1)
(a-2-2) edge node[auto] { $\Kc_\mu$} (a-1-2)
(a-2-1) edge node[auto] { $Q_t$} (a-2-2)
(a-2-2) edge node[auto] { $id$} (a-2-3)
(a-1-2) edge node[auto] { $\hw$} (a-2-3);
\end{tikzpicture}
\end{center}

The Pitman-Rogers criterion \ref{thm:rogerspitman_criterion} is applicable and tells us that $X^{x_0}$ is Markov with semigroup $Q_t$. It can be easily identified:
\begin{proposition}
The semigroup $Q$ is generated by the Doob transform of the quantum Toda Hamiltonian:
$$ Q_t = \exp\left( t \Lc \right) $$
with:
$$ \Lc = \psi_\mu^{-1} (H - \half \langle \mu, \mu \rangle) \psi_\mu
       = \half \Delta_\afrak + \langle \nabla \log \psi_\mu, \nabla \rangle$$
\end{proposition}
\begin{proof}
When $\mu \in -C = w_0 C$, we are in the same situation as theorem \ref{thm:whittaker_process_g}, where we identified the infinitesimal generator as:
$$  \half \Delta_\afrak + \langle \nabla \log \psi_{w_0\mu}, \nabla \rangle$$
Hence the result as $\psi_{w_0\mu} = \psi_{\mu}$ (Property (ii) in theorem \ref{thm:whittaker_functions_properties}). For general $\mu \in \afrak$, we use the fact that:
$$ \forall t \geq 0, Q_t = \Kc_\mu \circ P_t \circ \hw$$
Therefore $Q_t$ has infinitesimal generator:
$$ \Lc_\mu = \Kc_\mu \circ \left( \half \Delta_\afrak + \langle \mu, \nabla_\afrak \rangle + \half \sum_{\alpha \Delta} \langle \alpha, \alpha \rangle f_\alpha \right) \circ \hw$$
Against a smooth function $f: \afrak \rightarrow \R$, at a point $x$, $\Lc_\mu(f)(x)$ is analytic in the parameter $\mu$ and equal to
$$ \half \Delta_\afrak + \langle \nabla \log \psi_{\mu}, \nabla \rangle$$
for $\mu \in -C$. The result holds by analytic extension.
\end{proof}

\subsection{Intertwining property at the torus level}
The geometric Duistermaat-Heckman measure intertwines Brownian motion and the quantum Toda Hamiltonian. Formally, introduce the Markov kernel $\hat{\Kc}_\mu: \Cc\left( \afrak \right) \rightarrow \Cc(\afrak)$ defined for every positive $\varphi \in \Cc( \afrak )$ by:
$$
 \hat{\Kc}_\mu(\varphi)(\lambda) 
 := \E\left( \varphi \circ \wt \left( C_\mu(\lambda) \right) \right)\\
  = \frac{1}{\psi_\mu(\lambda)} \int_\afrak e^{\langle \mu, k \rangle} \ \varphi(k) \ DH^\lambda(dk)
$$

\begin{thm}
\label{thm:intertwining_torus}
The following diagram is commutative.
\begin{center}
\begin{tikzpicture}
\matrix(a)[matrix of math nodes, row sep=3em, column sep=3.5em,
text height=1.5ex, text depth=0.25ex]
{ \afrak & \afrak \\
  \afrak & \afrak \\ };
\path[->]
(a-1-1) edge node[above] { $e^{ t\left( \half \Delta_\afrak + \langle \mu, \nabla_\afrak \rangle \right) }$} (a-1-2)
(a-2-1) edge node[left]  { $\hat{\Kc}_\mu$} (a-1-1)
(a-2-2) edge node[right] { $\hat{\Kc}_\mu$} (a-1-2)
(a-2-1) edge node[below] { $Q_t$} (a-2-2);
\end{tikzpicture}
\end{center}
\end{thm}
\begin{proof}
Apply the earlier intertwining 
$\Kc_\mu \circ e^{t \Dc^{\mu}} = Q_t \circ \Kc_\mu$
to functions depending on the weight only.
\end{proof} 

\section{Asymptotics}
\label{section:asymptotics}
In this section, we investigate how the law of $C_\mu(\lambda)$ behaves as $\lambda$ goes to infinity in the opposite Weyl chamber. In subsection \ref{subsection:entrance_point}, this will be very important in order to complete the proofs of theorems \ref{thm:highest_weight_is_markov} and \ref{thm:canonical_measure}, by having the Whittaker process in theorem \ref{thm:whittaker_process_g} start at ``$x_0 = -\infty$''.

\subsection{Behavior of the canonical measure}

The following proposition is based on a weak version of the Laplace method. 
\begin{proposition}
\label{proposition:convergence_in_proba}
For given $\zeta \in \afrak$, $\mu \in \afrak$ and $M \rightarrow \infty$, we have convergence in probability for the $\Bc(\zeta)$-valued random variable:
$$ e^{-M \rho^\vee} C_\mu(\zeta - 2M \rho^\vee) e^{M \rho^\vee} \stackrel{\P}{\longrightarrow} m_\zeta$$
where $m_\zeta$ is the unique minimizer of the superpotential $f_B$ on $\Bc(\zeta)$.
\end{proposition}

\begin{corollary}
\label{corollary:convergence_in_proba}
With the same notations, as $M \rightarrow \infty$, the following limit holds in $B$:
$$ C_\mu(\zeta - 2M \rho^\vee) \stackrel{\P}{\longrightarrow} id$$
\end{corollary}
\begin{proof}
Thanks to theorem \ref{thm:superpotential_minimum}, $\wt(m_\zeta) = 0$ and therefore $m_\zeta \in N_{>0}^{w_0}$. The result follows from the previous proposition \ref{proposition:convergence_in_proba} and the lemma \ref{lemma:limit_contraction}.
\end{proof}

\begin{proof}[Proof of proposition \ref{proposition:convergence_in_proba}]
For easier notation write:
$$ X_M := e^{-M \rho^\vee} C_\mu(\zeta - 2M \rho^\vee) e^{M \rho^\vee}$$
Because the map $y \mapsto e^{-M \rho^\vee} y e^{M \rho^\vee}$ maps $\Bc(\zeta - 2M \rho^\vee)$ to $\Bc(\zeta)$, this random variable is indeed in $\Bc(\zeta)$. The latter claim is a consequence of (iii) Property 4.14 in \cite{bib:chh14a}:
\begin{align*}
    \hw( e^{-M \rho^\vee} y e^{M \rho^\vee} )
= & -w_0 M \rho^\vee + \hw\left( y \right) + M \rho^\vee\\
= & -w_0 M \rho^\vee + \left( \zeta - 2M \rho^\vee \right) + M \rho^\vee\\
= & \zeta
\end{align*}

Let $\varphi$ be a positive test function on $\Bc(\zeta)$.  By the canonical measure's definition:
\begin{align*}
    \E\left( \varphi( X_M )\right)
& = \E\left( \varphi( e^{-M \rho^\vee} C_\mu(\zeta-2M \rho^\vee) e^{M \rho^\vee} ) \right)\\
& = \frac{1}{\psi_\mu(\zeta - 2M\rho^\vee)} \int_{\Bc(\zeta - 2M \rho^\vee)} e^{ \langle \mu, \wt(y) \rangle - f_B(y)} \varphi\left( e^{-M \rho^\vee} y e^{M \rho^\vee} \right) \omega(dy) 
\end{align*}
Now, let us make the change of variables $x =  e^{-M \rho^\vee} y e^{M \rho^\vee}$. By definition, $\omega(dy)$ is the toric measure on the twisted Lusztig parameters of $y$ and:
$$ \varrho^T(x) = e^{-M \rho^\vee} \varrho^T(y) e^{M \rho^\vee}.$$
Hence the image measure for $\omega$ on $\Bc(\zeta-2M\rho^\vee)$ under this change of variable is again $\omega$ on $\Bc(\zeta)$:
\begin{align}
\label{eq:convergence_in_proba_1}
\omega(dy) = \omega(dx) 
\end{align}
Moreover:
\begin{align}
\label{eq:convergence_in_proba_2}
\wt(y) &= \wt(x)
\end{align}
\begin{align}
\label{eq:convergence_in_proba_3}
f_B(y) &=  e^M f_B\left( x \right)
\end{align}
For the latter claim, write:
$$ y = u' \bar{w}_0 e^{\zeta - 2M \rho^\vee} u$$
where $u' \in U^{w_0}_{>0}$, $u \in U^{w_0}_{>0}$. Hence:
$$ x = e^{-M \rho^\vee} u' e^{M \rho^\vee} \bar{w}_0 e^\zeta e^{-M \rho^\vee} u e^{M \rho^\vee}$$
Therefore from definition \ref{def:f_B}:
\begin{align*}
f_B(x) & = \chi\left( e^{-M \rho^\vee} u' e^{M \rho^\vee} \right) + \chi\left( e^{-M \rho^\vee} u e^{M \rho^\vee} \right)\\
& = e^{-M} \chi\left( u \right) + e^{-M} \chi\left( u' \right)\\
& = e^{-M} f_B(y)
\end{align*}
In the end, putting equations \eqref{eq:convergence_in_proba_1}, \eqref{eq:convergence_in_proba_2} and \eqref{eq:convergence_in_proba_3} together yields the appropriate formula for the Laplace method:
\begin{align}
\label{eq:convergence_in_proba_formula}
  \E\left( \varphi( X_M )\right)
& = \frac{ \int_{\Bc(\zeta)} e^{ \langle \mu, \wt(x) \rangle - e^M f_B(x)} \varphi\left( x \right) \omega(dx) } 
         { \int_{\Bc(\zeta)} e^{ \langle \mu, \wt(x) \rangle - e^M f_B(x)} \omega(dx) }
\end{align}

The following is quite standard. By theorem \ref{thm:superpotential_minimum}, $f_B$ has a unique minimizer on $\Bc(\zeta)$ denoted by $m_\zeta$. Consider $V$ a neighborhood of $m_\zeta$. Because $m_\zeta$ is a non-degenerate critical point, such a neighborhood contains a compact set of the form:
$$ K_\delta := \left\{  x \in \Bc(\zeta) \ | \ f_B(x) \leq f_B(m_\zeta) + \delta \right\}$$
for $\delta$ small enough. We denote by $V^c$ the complement of $V$. The theorem is proved once the following holds:
$$ \P\left( X_M \in V^c \right) \stackrel{M \rightarrow \infty}{\longrightarrow} 0 $$
We have:
\begin{align*}
       \P\left( X_M \in V^c \right)
\leq & \P\left( X_M \in K_\delta^c \right)\\
=    & \frac{ \int_{K_\delta^c} e^{ \langle \mu, \wt(x) \rangle - e^M f_B(x)} \omega(dx) }
            { \int_{\Bc(\zeta)} e^{ \langle \mu, \wt(x) \rangle - e^M f_B(x)} \omega(dx) }\\
=    & \frac{1}{1+
       \frac{\int_{K_\delta  } e^{ \langle \mu, \wt(x) \rangle - e^M (f_B(x)-f_B(m_\zeta)-\delta)} \omega(dx) }
            {\int_{K_\delta^c} e^{ \langle \mu, \wt(x) \rangle - e^M (f_B(x)-f_B(m_\zeta)-\delta)} \omega(dx) } }
\end{align*}
In the ratio of two integrals, the numerator goes to infinity as $M \rightarrow \infty$ because for instance of the contribution of $K_{\frac{\delta}{2}} \subset K_{\delta}$:
\begin{align*}
& \int_{K_\delta  } e^{ \langle \mu, \wt(x) \rangle - e^M (f_B(x)-f_B(m_\zeta)-\delta)} \omega(dx)\\
& \geq \int_{K_{\frac{\delta}{2}}} e^{ \langle \mu, \wt(x) \rangle - e^M (f_B(x)-f_B(m_\zeta)-\delta)} \omega(dx)\\
& \geq \int_{K_{\frac{\delta}{2}}} e^{ \langle \mu, \wt(x) \rangle - e^M \frac{\delta}{2} } \omega(dx)\\
& \rightarrow \infty
\end{align*}
The denominator decreases to zero as $M \rightarrow \infty$ using the dominated convergence theorem.
\end{proof}

As a consequence:
\begin{lemma}
\label{lemma:asymptotics_N_orbits}
Let $x_0 = \zeta - 2M \rho^\vee$ for any $\zeta \in \afrak$. Then as $M \rightarrow \infty$, in term of left $N$-orbits:
 $$ N e^{x_0} \Gamma_{\mu}(x_0) \stackrel{ \P }{\longrightarrow} N \bar{w}_0^{-1}$$
\end{lemma}
\begin{proof}
 Using proposition \ref{proposition:link_with_canonical}, we write:
$$ \Gamma_\mu(x_0) = \varrho^T\left( C_{w_0 \mu}(x_0) \right) = [\bar{w}_0^{-1} C_{w_0 \mu}(x_0)]_+$$
Let $X_M := e^{-M \rho^\vee } C_{w_0 \mu}(x_0) e^{M \rho^\vee }$. Thanks to proposition \ref{proposition:convergence_in_proba}, $X_M$ converges in probability to $m_\zeta$. Hence:
\begin{align*}
    N e^{x_0} \Gamma_{\mu}(x_0) 
= & N e^{x_0} [\bar{w}_0^{-1} e^{M \rho^\vee} X_M e^{-M \rho^\vee} ]_+ \\
= & N e^{x_0 + M \rho^\vee} [\bar{w}_0^{-1} X_M]_+ e^{-M \rho^\vee}\\
= & N e^{\zeta - M \rho^\vee} [\bar{w}_0^{-1} X_M]_{-0}^{-1} \bar{w}_0^{-1} X_M e^{-M \rho^\vee}\\
= & N e^{\zeta - M \rho^\vee} [\bar{w}_0^{-1} X_M]_{0}^{-1} \bar{w}_0^{-1} X_M e^{-M \rho^\vee }
\end{align*}
Moreover, $X_M \in \Bc(\zeta)$ and as such, $[\bar{w}_0^{-1} X_M]_{0}^{-1} = e^{-\zeta}$. Therefore:
$$  N e^{x_0} \Gamma_\mu(x_0)
=  N e^{- M \rho^\vee  } \bar{w}_0^{-1} X_M e^{-M \rho^\vee }
=  N \bar{w}_0^{-1} e^{M \rho^\vee } X_M e^{-M \rho^\vee }
$$
The fact that $e^{M \rho^\vee } X_M e^{-M \rho^\vee } \rightarrow \Id$ concludes the proof.
\end{proof}

\subsection{Entrance point at \texorpdfstring{``$-\infty$''}{minus infinity}}
\label{subsection:entrance_point}
In order to prove theorem \ref{thm:highest_weight_is_markov}, we need to take $x_0$ to ``$-\infty$'' in theorem \ref{thm:whittaker_process_g}. Indeed, thanks to the asymptotic analysis in the previous subsection, the following theorem holds.

\begin{thm}
\label{thm:measure_concentration}
Consider the family of random path transforms $x_0 + T_{\Gamma_{\mu}(x_0) }$ for $x_0 \in \afrak$. For $x_0 = -M \rho^\vee$, $M \rightarrow \infty$, we have the following convergence in probability for every continuous path $\pi \in \Cc( \R^+, \afrak)$:
$$ \forall t >0, x_0 + T_{\Gamma_{\mu}(x_0) }\pi(t) \stackrel{\P}{\longrightarrow} + \Tc_{w_0}\pi(t)$$ 
\end{thm}
\begin{proof}
 Recall that the path transform $T_g \pi(t)$ on a path is defined for $t>0$ as:
$$ e^{T_g \pi(t)} = [ g B_t(\pi)]_0$$
Using lemma \ref{lemma:asymptotics_N_orbits}, there is a sequence $n_M \in N$ such that as $M \rightarrow \infty$:
$$n_M e^{x_0} \Gamma_\mu(x_0) \stackrel{\P}{\rightarrow} \bar{w}_0^{-1}$$
Hence:
\begin{align*}
    \exp\left( x_0 + T_{\Gamma_{\mu}(x_0) } \pi(t) \right)
= & [ e^{x_0} \Gamma_{\mu}(x_0) B_t(\pi)]_0 \\
= & [ n_M e^{x_0} \Gamma_{\mu}(x_0) B_t(\pi)]_0 \\
\rightarrow & \exp\left( \Tc_{w_0} \pi(t) \right)
\end{align*}
\end{proof}

We are now ready to complete the proofs of theorems \ref{thm:highest_weight_is_markov} and \ref{thm:canonical_measure}.

\begin{proof}[Proof of theorems \ref{thm:highest_weight_is_markov} and \ref{thm:canonical_measure}]
Take in theorem \ref{thm:markovian_filtering} $x_0 = \zeta - 2M \rho^\vee$ and as in corollary \ref{corollary:convergence_in_proba} take $M \rightarrow \infty$, giving:
$$ C_\mu(\zeta - 2M \rho^\vee) \stackrel{\P}{\longrightarrow} id$$
The Markov process $X^{x_0}_t = \hw\left( C_\mu(x_0) B_t(W^{(\mu)}) \right)$ will converge in probability to the highest weight process $\Lambda_t = \hw\left( B_t(W^{(\mu)}) \right)$. The filtering equation in theorem \ref{thm:markovian_filtering} degenerates to the relation in theorem \ref{thm:canonical_measure}.
\end{proof}

\section{Degeneration to the tropical case}
\label{section:degenerations}

Let $h>0$. In \cite{bib:chh14a} section 6, we have described a natural $h$-deformation of the geometric Littelmann path model that is given by rescaling paths and corresponding actions. This deformation was interpreted as a change of semi-fields and we will now describe the deformed structures. The $h \rightarrow 0$ limit makes sense, and is exactly the free version of the continuous Littelmann model given in \cite{bib:BBO2}. A cutting procedure is needed in order to ``prune'' such a free Kashiwara crystal, and obtain a polytope (Remark 6.6 in \cite{bib:chh14a}). While Berenstein and Kazhdan have used the superpotential function $f_B$ (\cite{bib:BK00, bib:BK04}) to encode this cutting procedure as in equation \ref{eq:BK_pruning}, there is not a clear reason why it should be that way. In our point of view, the superpotential $f_B$ appeared naturally in the canonical measure on geometric crystals.

While describing deformations, we will see that the $h$-deformations of theorem \ref{thm:highest_weight_is_markov} uses the operator $h \Tc_{w_0} h^{-1} \stackrel{h \rightarrow 0}{\longrightarrow} \Pc_{w_0}$. This recovers the crystalline generalisation of Pitman's theorem proved in \cite{bib:BBO} where $\Pc_{w_0} W$ is Brownian motion conditionned to never leave the Weyl chamber. In the $A_n$ type, this is a realization of Dyson's Brownian motion which gives a connection to Random Matrix theory.

Also, in this crystallization procedure, the canonical measure degenerates to the uniform measure on a polytope, which is nothing but the string polytope in the appropriate coordinates. This recovers previous results.

\subsection{Semifields and Maslov quantification}
\label{subsection:semifields_and_maslov}
This subsection mainly follows the presentation of Itenberg in \cite{bib:Itenberg}. A semifield $\left( S, \oplus, \odot \right)$ is defined as the next best thing to a field, as we weaken the assumption of invertibility for the law $\oplus$. 

\begin{definition}
 A semifield is an algebraic structure $\left( S, \oplus, \odot \right)$ such that:
\begin{itemize}
 \item $\left( S, \oplus\right)$ is a commutative semigroup.
 \item $\left( S, \odot \right)$ is a commutative group with neutral element $e$.
 \item Distributivity of $\odot$ over $\oplus$:
	$$\forall a, b, c \in S, \left( a \oplus b \right) \odot c =  (a \odot c) \oplus (b \odot c)$$
\end{itemize}
\end{definition}

The universal semifield we have been working with so far is $\left(\R_{>0}, + , . \right)$. On this semi-field, the natural counterpart of rational functions with $n$ indeterminates is the set of \emph{rational and substraction free} expressions $\R_{>0}\left(x_1, \dots, x_n\right) $. Plainly, elements in $\R_{>0}\left(x_1, \dots, x_n\right)$ are rational functions with indeterminates $\left( x_1, \dots, x_n \right)$, real positive coefficients and using only operations $+$, $\times$ and $/$. For example $f(x_1, x_2) = \frac{x_1^3 + x_2^3}{x_1 + x_2} = x_1^2 - x_1 x_2 + x_2^2 \in \R_{>0}(x_1, x_2)$. It is easy to check that if endowed with the same operations, rational subtraction free expressions also form a semi-field.\\

Another classical example is the \emph{tropical} semifield $\left(\R, \min, + \right)$ as one easily checks that $+$ is distributive over $\min$. Its importance in representation theory is related to Kashiwara's crystal basis, as changes of coordinates are rational functions on $S_0$. The study of algebraic curves on this field has given rise to tropical geometry, now a field of its own, where $\max$ usually replaces $\min$. The name ``tropical'' was coined by French computer scientists to honor their colleague Imre Simon for his work on the max-plus algebra. It has no intrinsic meaning aside from refering to the weather in Brazil. In fact, we will see later that this semifield can be viewed like the zero temperature limit of family of semifields $S_h$. This suggests the name of ``crystallized'' semifield, that fits better in name to the crystal basis. However, it is too late to reverse the trend, already solidly established.

Tropicalization (or crystallization) is a procedure that takes as input objects on the semi-field $\left(\R_{>0}, +, . \right)$ and gives objects on $\left(\R_{>0}, \min, + \right) $. As such, if $f \in \R_{>0}(x_1, \dots, x_n)$, a substraction free rational function, one obtains $[f]_{trop}$ a function in the variables $\left(x_1, \dots, x_n\right)$ applying the morphism of semi-fields $[\ \ ]_{trop}$. If $a$ and $b$ are elements in $\R_{>0}(x_1, \dots, x_n)$ then:

$$ \begin{array}{cccc}
 [\ \ ]_{trop}: & \left(\R_{>0}, + , . \right) & \longrightarrow & \left(\R, \min, + \right) \\
                & a + b                        & \mapsto	 & \min( [\ a \ ]_{trop}, [\ b \ ]_{trop})\\
                & a . b                        & \mapsto	 & [\ a \ ]_{trop} + [\ b \ ]_{trop}\\
                & a / b                        & \mapsto 	 & [\ a \ ]_{trop} - [\ b \ ]_{trop}\\
                & a \in \R_{>0}                & \mapsto 	 & 0
    \end{array}
$$

A less algebraic definition could be used, using a limit that always exists:
\begin{proposition}[Analytic tropicalization procedure - Proposition 3.1 in \cite{bib:chh14_exit}]
\label{proposition:analytic_tropicalization}
For $f$ a rational and substraction free expression in $k$ variables, we have for all $(x_1, \dots, x_k) \in \R^k$ and $h>0$:
\begin{align}
-h \log f\left( e^{-\frac{x_1}{h}}, \dots, e^{-\frac{x_k}{h}} \right) & = [f]_{trop}\left( x_1, \dots, x_k\right) + O(h)
\end{align}
where $O(h)$ is a quantity such that $\frac{O(h)}{h}$ is bounded as $h \rightarrow 0$, uniformly in the variables $(x_1, \dots, x_k)$.
\end{proposition}

Such a limit suggests a continuous deformation from $\left(\R_{>0}, + , . \right)$ to $\left(\R, \min, + \right)$ called the Maslov quantification of real numbers. Define the continuous family of semifields 
$\left( S_h = \R, \oplus_h, \odot_h \right)$ for $h\geq 0$ with:
$$ a \oplus_h b = -h \log\left( e^{-\frac{a}{h}} + e^{-\frac{b}{h}}\right) $$
$$ a \odot_h b = a + b $$
At the limit, when $h$ goes to zero, we recover the previous example $\left(\R, \min, + \right)$. All the semifields $\left( S_h \right)_{h>0}$ are isomorphic to $\left(\R_{>0}, + , . \right)$ except for $h=0$. The isomorphism of semifields that transports structure is $\psi_h = -h \log: \left(\R_{>0}, + , . \right) \rightarrow \left( S_h = \R, \oplus_h, \odot_h \right)$. As such, $\psi_{h,h'} = \psi_{h'} \circ \psi_h^{-1}:\left( S_h = \R, \oplus_h, \odot_h \right) \rightarrow \left( S_{h'} = \R, \oplus_{h'}, \odot_{h'} \right)$ is a rescaling when identifying both semifields to $\R$: $\psi_{h,h'}\left( x \right) = \frac{h'}{h} x$.

\begin{notation}
A tilde will refer to quantified variables when there is the possibility of confusing them with variables in $\R_{>0}$. In $c = e^{-\frac{\tilde{c}}{h}}$, $c$ is seen as a variable in the usual semi-field $\R_{>0}$ while $\tilde{c}$ is in $S_h$.
\end{notation}

\subsection{A remark on integrals of semifield valued functions}
Let $f: [0, T] \rightarrow S_h$ be a (smooth) function with values in the semifield $S_h$. For readability purposes in this subsection, the subscript $h$ will be dropped when designating operations on $S_h$. The monoid of integers in $S_h$ is the monoid generated by the neutral element $0$. It is in fact given by all numbers $n_h = \mathop{\bigoplus}_{i=1}^n 0 = -h \log\left(n\right)$. 
As such, Riemann sums in $S_h$ take the form:
$$ 
  \left( 1 \oslash n_h \right) \odot \mathop{\bigoplus}_{i=1}^n  f(t_i)
= -h \log\left( \frac{1}{n} \sum_{i=1}^n e^{-\frac{f(t_i)}{h}} \right)
\mathop{\longrightarrow}_{n \rightarrow +\infty} -h \log\left( \int_0^T e^{-\frac{f}{h}} \right) 
$$
where $f$ is a continuous real function. Therefore, the natural candidate for integrals on the semifield $S_h$ are exponential functionals and the $h=0$ limit gives $ \inf_{0 \leq s \leq T} f(s)$ using the Laplace method. Hence, the appearance of integrals of exponentials whether in the geometric Littelmann path model \cite{bib:chh14a} or thanks to the Toda potential is no surprise. In the formalism of semi-fields, all crystal actions become in fact rational, in the sense of the semi-field $S_h$. Exponential integrals are simply semifield integrals that degenerate to infimums.

\subsection{Deformed structures}
Following the same idea as \cite{bib:BFZ96} section (2.2), one can define the Lusztig and Kashiwara varieties $U^{w_0}_{>0}$ and $C^{w_0}_{>0}$ by their parametrizations, by identifying $m$-tuples that give the same element. Since changes of parametrization $x_{\bf i'}^{-1} \circ x_{\bf i}$ are rational and substraction free, one can view them as rational for the semi-field $S_h$ and define:

\begin{definition}[ Lusztig and Kashiwara varieties on $S_h$ ]
  $$U_{>0}^{w_0}(S_h) := \left\{ ({\bf t}^{\bf i})_{{\bf i} \in R(w_0)} \in \left(S_h^m\right)^{R(w_0)} \ | \ \forall {\bf i}, {\bf i'} \in R(w_0), x_{\bf i'}^{-1} \circ x_{\bf i}( {\bf t^{\bf i}} ) = {\bf t^{\bf i'}} \right\}$$
  $$C_{>0}^{w_0}(S_h) := \left\{ ({\bf c}^{\bf i})_{{\bf i} \in R(w_0)} \in \left(S_h^m\right)^{R(w_0)} \ | \ \forall {\bf i}, {\bf i'} \in R(w_0), x_{\bf-i'}^{-1} \circ x_{\bf-i}( {\bf c^{\bf i}} ) = {\bf c^{\bf i'}} \right\}$$
\end{definition}

The $h \rightarrow 0$ limit gives the tropicalized version of the changes of parametrization. As such, by theorem 5.2 \cite{bib:BZ01}, the Lusztig variety $U_{>0}^{w_0}(S_0)$ really encodes the Lusztig parametrization of the $G^\vee$-canonical basis; while the tropical Kashiwara variety encodes the string parametrization.

On the side of the Littelmann path model, $h$-Littelmann models for different $h$ are equivalent, provided that we properly rescale the reals in the actions, and values taken by the structural maps (\cite{bib:chh14a} remark 6.5). In fact, the set of real numbers had to be considered as the semifield $S_h$, and this rescaling becomes natural as we also have to change the structure semifield. Now we are ready to list some $h$-deformation of the structure results given in \cite{bib:chh14a}. In order to distinguish between structures at $h = 1$ and for $h$, let $L = \left( \wt, \left( \varepsilon_\alpha, \varphi_\alpha, e^._\alpha \right)_{\alpha \in \Delta} \right)$ the path crystal structure for $h=1$ and $L_{h} = \left( \wt', \left( \varepsilon^{'}_{\alpha}, \varphi^{'}_{\alpha}, e^{'.}_{\alpha} \right)_{\alpha \in \Delta} \right)$ for generic $h$. We will use the subscript $h$ in $\langle \pi \rangle_h$ to indicate the crystal generated by $\pi$ using the $h$-deformed structure.
 
\paragraph{Generated crystal:} Let $\langle \pi \rangle_h$ be the $h$-Littelmann crystal generated by $\pi$. After transporting the structure to $h=1$ by rescaling, we have to consider the geometric crystal generated by $h^{-1}\pi $. In the end, in term of the geometric structure (h=1), we have:
$$ \langle \pi \rangle_h = h \langle \frac{\pi}{h} \rangle$$

\paragraph{Highest weight:} The natural invariant under crystal action, which plays the role of highest weight, is then 
\begin{align}
\label{eq:h_deformed_pitman}
\Tc_{w_0}^h := & h \Tc_{w_0} h^{-1}
\end{align}
It is natural because varying $h$ interpolates between different path models, and gives for each $h$ the highest weight path. The limit \eqref{eq:limit_pitman} is consistent with the continuous Littelmann path model considered in \cite{bib:BBO} and \cite{bib:BBO2} and recovers the Pitman operator.

\paragraph{Parametrizations:} Fix ${\bf i} \in R(w_0)$.\\
Transporting the semi-field structure from $\R_{>0}$ to $S_h$ and using the results about parametrizations of geometric crystals in section 8 of \cite{bib:chh14a}, we will now define $h$-deformed string coordinates of a path:
$$ \begin{array}{cccc}
\varrho_{\bf i}^{h, K}: & \Cc_0\left( [0, T], \afrak \right) & \longrightarrow & \left(S_h\right)^m \\
                        &         \pi                      & \mapsto	     & \left( c_1, \dots, c_m \right)\\
    \end{array}
$$
For $\pi \in \Cc_0\left( [0, T], \afrak \right)$, the $m$-tuple ${\bf c} = \left( c_1, \dots, c_m \right) = \varrho_{\bf i}^{h, K}(\pi)$ is defined recursively as:
$$ \forall 1 \leq k \leq m, c_k = h \log\int_0^T \exp\left( -h^{-1}\alpha_{i_k}\left( \Tc_{s_{i_1} \dots s_{i_{k-1}}}^h \pi \right) \right)$$
Clearly, as $h \rightarrow 0$, one recovers the definition of string parameters in the classical Littelmann path model (see \cite{bib:BBO2}):
$$ \forall 1 \leq k \leq m, h \log\int_0^T \exp\left( -h^{-1}\alpha_{i_k}\left( \Tc_{s_{i_1} \dots s_{i_{k-1}}}^h \pi \right) \right) \rightarrow - \inf_{0 \leq t \leq T} \alpha_{i_k}\left( \Pc_{s_{i_1} \dots s_{i_{k-1}}}(\pi) \right)$$
Finally, thanks to diagram \eqref{fig:parametrizations_diagram} and the morphism of semi-fields $\psi_h = -h \log$  we set:
\begin{align}
\label{eq:q_kashiwara_for_paths}
\forall \pi \in \Cc_0\left( [0, T], \afrak \right), \varrho_{\bf i}^{h, K}(\pi) := \psi_h \circ x_{\bf -i}^{-1} \circ \varrho^K\left( B_T(h^{-1} \pi) \right)
\end{align}
where we applied the semi-field morphism $\psi_h = -h \log$ on $\R_{>0}^m$ point-wise.\\

Similarly, $h$-deformed Lusztig parameters are constructed. Indeed, all elements of $\langle \pi \rangle_h$ can be projected on the lowest path $\eta = h e^{-\infty}_\alpha h^{-1} \pi$ and every single path can be recovered via:
$$ \pi = h T_g h^{-1} \eta$$
where 
\begin{itemize}
 \item $ g = x_{i_1}\left( e^{-\frac{t_1}{h}} \right) \dots x_{i_m}\left( e^{-\frac{t_1}{h}} \right) \in U^{w_0}_{>0}$
 \item $\eta_j = h e^{-\infty}_{s_{i_1} \dots s_{i_j}} \cdot \frac{\pi}{h} = h e^{-\infty}_{\alpha_{i_j}} \cdot \frac{\eta_{j-1}}{h}$
 \item $t_j = h \log\int_0^T \exp\left( -h^{-1}\alpha_{i_k}\left( h e_{s_{i_1} \dots s_{i_{k-1}}}^{-\infty} h^{-1} \pi \right) \right)$
\end{itemize}

We are aiming at understanding the law $\varrho_{\bf i}^{h, K}\left( \pi \right)$ when $\pi$ is taken as a Brownian motion. This is the natural $h$-deformation of the previously studied canonical measure, viewed in string coordinates.

\subsection{Brownian scaling and consequences}
In order to obtain the announced deformation of our probabilistic results, the tool we will use is the Brownian scaling property. For $W$ a Brownian motion in $\afrak$, $\mu \in \afrak$ and $c>0$, it is the equality in law between processes:
\begin{align}
\label{eq:brownian_scaling}
\left( W_{t}^{(\mu)} \ ; \ t \geq 0 \right) \stackrel{\Lc}{=} \left( c W_{t/c^2}^{(c\mu)} \ ; \ t \geq 0 \right)
\end{align}

Let us first examine the effect of scaling on the flow $B_.(.)$:
\begin{lemma}[Effect of accelerating a path $X$ on $B_.\left(X\right)$]
\label{lemma:accelerating_paths}
Given a continuous path $X$ in $\afrak$:
 $$ B_t\left( X_{./c^2} \right) = c^{-2\rho^\vee} B_{t/c^2}\left( X \right) c^{2\rho^\vee}$$
where $\rho^\vee$ is the Weyl covector \eqref{eq:def_weyl_covector}.
\end{lemma}
\begin{proof}
Start from equation \eqref{eq:process_B_explicit} and use the change of variable $u_j = t_j/c^2$:
 \begin{align*}
   & B_t\left( X_{\frac{.}{c^2}} \right) \\
 = & \left( \sum_{k \geq 0} \sum_{ i_1, \dots, i_k } \int_{ t \geq t_k \geq \dots \geq t_1 \geq 0} 
     f_{i_1} \dots f_{i_k}
     \prod_{j=1}^k dt_j \frac{|| \alpha_{i_j} ||^2}{2} e^{ -\alpha_{i_j}(X_{t_j/c^2}) }
     \right) e^{X_{\frac{t}{c^2}}}\\
 = & \left( \sum_{k \geq 0} \sum_{ i_1, \dots, i_k } \int_{ t \geq t_k \geq \dots \geq t_1 \geq 0} 
     f_{i_1} \dots f_{i_k}
     c^{2k} \prod_{j=1}^k dt_j \frac{|| \alpha_{i_j} ||^2}{2} e^{ -\alpha_{i_j}(X_{u_j}) }
     \right) e^{X_{\frac{t}{c^2}}}\\
 = & \left( \sum_{k \geq 0} \sum_{ i_1, \dots, i_k } \int_{ t \geq t_k \geq \dots \geq t_1 \geq 0} 
     c^{-2\rho^\vee} f_{i_1} \dots f_{i_k} c^{2\rho^\vee}
     \prod_{j=1}^k dt_j \frac{|| \alpha_{i_j} ||^2}{2} e^{ -\alpha_{i_j}(X_{u_j}) }
     \right) e^{X_{\frac{t}{c^2}}}\\
 = & c^{-2\rho^\vee} B_{\frac{t}{c^2}}\left( X \right) c^{2\rho^\vee}
 \end{align*}
\end{proof}
By applying the previous lemma to Brownian motion, we have thanks to the Brownian scaling \eqref{eq:brownian_scaling}:
\begin{corollary}
\label{corollary:B_scaling}
$$ \left( B_{t}\left( h^{-1} W^{(\mu)} \right) ; t \geq 0 \right) \eqlaw \left( h^{-2\rho^\vee}  B_{t/h^2}\left( W^{(h \mu)} \right) h^{2\rho^\vee}; t \geq 0 \right)$$
\end{corollary}

Therefore, we can give a deformation of theorems \ref{thm:highest_weight_is_markov} and \ref{thm:canonical_measure}. Define the rescaled highest weight process for $t>0$ as:
$$ \Lambda_t^h := h \ \hw\left( B_t( h^{-1} W^{(\mu)} ) \right) =  \Tc_{w_0}^h(W^{(\mu)})_t$$
The properly rescaled Whittaker function on $\afrak$ is, with $m = \ell(w_0)$:
$$\forall \lambda \in \afrak, \psi_{h, \mu}(\lambda) = h^{m}\psi_{h\mu}\left( \frac{\lambda-4 h \log(h) \rho^\vee}{h}\right)$$
Using theorem \ref{thm:whittaker_functions_properties}, it is immediate that when $\mu \in C$, $\psi_{h, \mu}$ solves the eigenfunction equation:
\begin{align}
\label{eq:q_deformed_toda}
\frac{1}{2} \Delta \psi_{h, \mu} - \sum_{\alpha \in \Delta} \frac{1}{2} \langle \alpha, \alpha \rangle h^2 e^{-h^{-1} \alpha\left( . \right)}\psi_{h, \mu}  = & \frac{ \langle \mu, \mu \rangle }{2} \psi_{h, \mu} 
\end{align}
with $\psi_{h, \mu}(x) e^{-\langle \mu, x \rangle}$ being bounded and having growth condition:
$$ \lim_{x \rightarrow \infty, x \in C} \psi_{h, \mu}(x) e^{-\langle \mu, x \rangle} = h^{m} b(h \mu)$$

\begin{thm}[Markov property for rescaled highest weight]
\label{thm:q_highest_weight_is_markov}
The process $\Lambda^h$ is a diffusion with infinitesimal generator
$$ \psi_{h, \mu}^{-1}
   \left( \frac{1}{2} \Delta - \sum_{\alpha \in \Delta} \frac{1}{2} \langle \alpha, \alpha \rangle h^2 e^{-h^{-1} \alpha\left( x \right)} - \frac{ \langle \mu, \mu \rangle }{2} \right)
   \psi_{h, \mu}
 = \frac{1}{2} \Delta + \nabla \log\left(\psi_{h, \mu}\right) \cdot \nabla $$
\end{thm}
\begin{proof}
Thanks to corollary \ref{corollary:B_scaling} and properties 4.14 in \cite{bib:chh14a} of $\hw$:
$$ \Lambda_t^h ; t \geq 0 \stackrel{\Lc}{=} 4 h \log(h) \rho^\vee + h \ \hw\left(B_{t/h^2}\left( W^{(h \mu)} \right) \right) ; t \geq 0$$
The result is a consequence of theorem \ref{thm:highest_weight_is_markov} and the following general fact applied to the highest weight process. Consider an Euclidian space $V$ and $a \in V$. If $(X_t; t \geq 0)$ is a diffusion on $V$ with generator $\Lc$ then $(h X_{t/h^2} + a ; t \geq 0)$ is a diffusion with generator $\Gc$. For a smooth function $f: V \rightarrow \R$, we have $\Gc(f): x \mapsto h^{-2} \Lc\left( f(h. + a)\right)(\frac{x-a}{h})$. Here, one needs to take $V = \afrak$, $X$ is the highest weight process and $a = 4 h \log(h) \rho^\vee$.\\
\end{proof}

The deformation of theorem \ref{thm:canonical_measure} is:
\begin{thm}[Rescaled canonical measure]
\label{thm:q_canonical_measure}
 $$\forall t>0, \left( B_t\left( h^{-1} W^{(\mu)} \right) | \Lambda_t^h = \lambda \right)
\stackrel{\Lc}{=} h^{-2\rho^\vee} C_{h\mu}\left( \frac{\lambda - 4 q\log(h) \rho^\vee}{h} \right) h^{2\rho^\vee}
$$
\end{thm}
\begin{proof}
Using corollary \ref{corollary:B_scaling}, we have for fixed $t>0$:
\begin{align*}
                  & \left( B_t\left( h^{-1} W^{(\mu)} \right) | \Lambda_t^h = \lambda \right)\\
\stackrel{\Lc}{=} & \left( h^{-2\rho^\vee} B_{t/h^2}\left( W^{(h \mu)} \right) h^{2\rho^\vee} |
                \hw\left( h^{-2\rho^\vee} B_{t/h^2}\left( W^{(h \mu)} \right) h^{2\rho^\vee} \right)= h^{-1} \lambda \right)\\
               =  & \left( h^{-2\rho^\vee} B_{t/h^2}\left( W^{(h \mu)} \right) h^{2\rho^\vee} |
                \hw\left( B_{t/h^2}\left( W^{(h \mu)} \right) \right)= h^{-1} \lambda - 4 \log h \rho^\vee\right)\\
\end{align*}
Combining this with theorem \ref{thm:canonical_measure} yields the result.
\end{proof}

\subsection{Explicit computation in string coordinates}

We are now able to give an integral formula for the law of $h$-deformed string parameters extracted from a finite Brownian path. We present it in a form that allows to compute the $h\rightarrow 0$ limit. It uses the map $\eta^{w_0, e}$ defined as:
$$ \forall v \in B \cap U \bar{w}_0 U, \eta^{w_0, e}\left( v \right) = [\left( \bar{w}_0 v^T\right)^{-1}]_+$$
Recall that $\eta^{w_0, e}$ restricts to a bijection from $C_{>0}^{w_0}$ to $U_{>0}^{w_0}$ (See figure \ref{fig:geom_parametrizations} or \cite{bib:BZ01}, corollary 5.6).

\begin{proposition}
\label{proposition:q_string_formula}
Let ${\bf i} \in R(w_0)$ and $t>0$. Consider in $\afrak$ a Brownian motion with drift $\mu$ up to time $t$, $\left( W^{(\mu)}_u; 0 \leq u \leq t \right)$. Then for any $\varphi: \R^m \rightarrow \R$ bounded measurable function:
\begin{align*}
  & \E\left[ \varphi\left( \varrho_{\bf i}^{h, K}\left( W^{(\mu)}_u; 0 \leq u \leq t \right) \right) \ | \ \Tc_{w_0}^h W^{(\mu)}_t = \lambda \right]\\
= &  \frac{1}{\psi_{h, \mu}\left(\lambda\right)}
     \int_{\R^m} {\bf dc} \ \varphi({\bf c}) \exp\left( \langle \mu, \lambda - \sum_{k=1}^m c_k \alpha_{i_k}^\vee \rangle- f_{B,h,\lambda}^{K, {\bf i}}({\bf c})\right)
\end{align*}
where ${\bf dc}$ is the Lebesgue measure on $\R^m$ and the deformed superpotential in string coordinates is given by:
\begin{align}
f_{B,h,\lambda}^{K, {\bf i}}({\bf c}) 
& := \sum_{\alpha \in \Delta} \frac{2 h^2}{\langle \alpha, \alpha \rangle} \chi_\alpha \circ \eta^{w_0, e} \circ x_{\bf-i}\left( e^{-h^{-1}c_1}, \dots, e^{-h^{-1}c_m} \right)\\
& \ \ + \sum_{j=1}^m \frac{2 h^2}{\langle \alpha_{i_j}, \alpha_{i_j} \rangle} \exp\left( -\frac{\lambda-\left(c_j+\sum_{k=j+1}^m c_k \alpha_{i_j}(\alpha_{i_k}^\vee)\right)}{h} \right)
\end{align}
\end{proposition}
\begin{proof}
Using equation \eqref{eq:q_kashiwara_for_paths} and writing $f := \varphi \circ \psi_h \circ x_{- \bf i}^{-1} \circ \varrho^K$:
\begin{align*}
    \E\left[ \varphi\left( \varrho_{\bf i}^{h, K}\left( W^{(\mu)}_u; 0 \leq u \leq t \right) \right) \ | \ \Tc_{w_0}^h W^{(\mu)}_t = \lambda \right]
= & \E\left( f\left( B_t(W^{(\mu)}) \right) \ | \ \Tc_{w_0}^h W^{(\mu)}_t = \lambda \right)
\end{align*}
As a consequence of theorem \ref{thm:q_canonical_measure} and then the integral formula from equation \eqref{eq:canonical_measure_density}, we have:
\begin{align*}
  & \E\left( \varphi\left( \varrho_{\bf i}^{h, K}\left( W^{(\mu)}_u; 0 \leq u \leq t \right) \right) \ | \ \Tc_{w_0}^h W^{(\mu)}_t = \lambda \right)\\
= & \E\left( f\left( \rho^{-2\rho^\vee} C_{h\mu }(\frac{\lambda - 4 h \log h \rho^\vee }{h}) \rho^{2\rho^\vee} \right)  \right)\\
= & \frac{h^m}{\psi_{h, \mu}\left(\lambda\right)}
    \int_{\Bc\left( \frac{\lambda - 4 h \log h \rho^\vee }{h}\right)}
    f( \rho^{-2\rho^\vee} x \rho^{2\rho^\vee} ) e^{\langle h\mu, \wt(x)\rangle - f_B(x)} \omega(dx)
\end{align*}
Making the change of variable $y = \rho^{-2\rho^\vee} x \rho^{2\rho^\vee} $, which maps $\Bc\left( \frac{\lambda - 4 h \log h \rho^\vee }{h}\right)$ to $\Bc\left( \frac{\lambda}{h}\right)$:
\begin{align*}
  & \E\left( \varphi\left( \varrho_{\bf i}^{h, K}\left( W^{(\mu)}_u; 0 \leq u \leq t \right) \right) \ | \ \Tc_{w_0}^h W^{(\mu)}_t = \lambda \right)\\
= & \frac{h^m}{\psi_{h, \mu}\left(\lambda\right)}
    \int_{\Bc\left( h^{-1} \lambda \right)}
    f( y ) e^{\langle h \mu, \wt(y) - f_B(\rho^{2\rho^\vee} y \rho^{-2\rho^\vee} )} \omega(dy)
\end{align*}
And for:
$${\bf c} = \psi_h \circ x_{- \bf i}^{-1} \circ \varrho^K(y)$$
or equivalently
$$y = b_{h^{-1}\lambda}^K \circ x_{- \bf i}\left( e^{-\frac{c_1}{h}}, \dots, e^{-\frac{c_m}{h}}\right)$$
We explicit the previous integral in terms of the variable ${\bf c}$. Theorem \ref{thm:omega_invariance_on_crystal} leads to:
\begin{align}
\label{eq:explicit_measure}
\omega(dy) & = h^m \prod_{k=1}^m d c_k = h^m {\bf dc}
\end{align}
Theorem 4.22 in \cite{bib:chh14a} gives the explicit expression for the weight map:
\begin{align}
\label{eq:explicit_weight}
\wt(y) & = h^{-1} \left( \lambda - \sum_{k=1}^m c_k \alpha_{i_k}^\vee \right)
\end{align}
Finally, we claim:
\begin{align}
\label{eq:explicit_superpotential}
f_B(\rho^{2\rho^\vee} y \rho^{-2\rho^\vee} ) & = f_{B, h, \lambda}^{K, {\bf i}}({\bf c})
\end{align}
Putting together equations \eqref{eq:explicit_measure}, \eqref{eq:explicit_weight} and \eqref{eq:explicit_superpotential} yields the result. Now, we only need to prove the last equation. Recall that, by writing $v = x_{- \bf i}\left( e^{-\frac{c_1}{h}}, \dots, e^{-\frac{c_m}{h}}\right) \in C_{>0}^{w_0}$ and using proposition 4.17 in \cite{bib:chh14a}:
\begin{align*}
y = & b_{h^{-1}\lambda}^K \circ x_{- \bf i}\left( e^{-\frac{c_1}{h}}, \dots, e^{-\frac{c_m}{h}}\right)\\
  = & b_{h^{-1}\lambda}^K (v)\\
  = & \eta^{w_0, e}\left( v \right) \bar{w}_0 v^T [v^T]_0^{-1} e^{h^{-1} \lambda}
\end{align*}
Therefore:
\begin{align*}
  & f_B(\rho^{2\rho^\vee} y \rho^{-2\rho^\vee} )\\
= & f_B(\rho^{2\rho^\vee} \eta^{w_0, e}\left( v \right) \bar{w}_0 v^T [v^T]_0^{-1} e^{h^{-1} \lambda} \rho^{-2\rho^\vee} )\\
= & \chi\left( \rho^{2\rho^\vee} \eta^{w_0, e}\left( v \right) \rho^{-2\rho^\vee} \right)
  + \chi\left( \rho^{2\rho^\vee} e^{-h^{-1} \lambda} v^T [v^T]_0^{-1} e^{h^{-1} \lambda} \rho^{-2\rho^\vee} \right) \\
= & \sum_{\alpha \in \Delta} \frac{2 h^2}{\langle \alpha, \alpha \rangle} \chi_\alpha \circ \eta^{w_0, e}(v)
  + \sum_{j=1}^m \frac{2 h^2}{\langle \alpha_{i_j}, \alpha_{i_j} \rangle} \exp\left( -\frac{\lambda-\left(c_j+\sum_{k=j+1}^m c_k \alpha_{i_j}(\alpha_{i_k}^\vee)\right)}{h} \right)
\end{align*}
\end{proof}

\subsection{Crystallization}
As the following proposition indicates, the superpotential degenerates to an indicator function of a polytope. It is the string polytope for the Langlands dual $G^\vee$, thereby recovering where the string parameters sit for the usual highest weight Kashiwara crystals $\Bfrak(\lambda)$.  

\begin{proposition}
For every ${\bf i} \in R(w_0)$, there is a cone $\Cc_{\bf i}^\vee$ such that:
$$ \forall {\bf c} \in \R,
   \lim_{h \rightarrow 0} e^{-f_{B, h, \lambda}^{K, {\bf i}}({\bf c})}
 = \mathds{1}_{\left\{ {\bf c} \in \Cc_{\bf i}^\vee\right\}} \mathds{1}_{\left\{ c_j \leq \lambda-\sum_{k=j+1}^m c_k \alpha_{i_j}(\alpha_{i_k}^\vee)\right\}}$$
It is the string cone for the group $G^\vee$ and it is given for any choice of ${\bf i'} \in R(w_0)$ by:
$$ \Cc_{\bf i}^\vee = \left\{ {\bf c} \in \R^m \ | \ [x_{\bf i'}^{-1} \circ \eta^{w_0, e} \circ x_{\bf-i}]_{trop}({\bf c}) \in \R_+^m \right\} $$
\end{proposition}
Notice the appearance of exactly the same cutting condition of the string cone $\Cc_{\bf i}^\vee$ as the condition given in \cite{bib:Littelmann98} page 5, giving the string polytope associated to the highest weight $\lambda$.
\begin{proof}
Looking at proposition \ref{proposition:q_string_formula}, the deformed superpotential $e^{-f_{B, h, \lambda}^{K, {\bf i}}({\bf c})}$ is the product of two terms. Each one of them leads to an indicator function. The easier one to deal with is:
\begin{align*}
& \exp\left( - \sum_{j=1}^m \frac{2 h^2}{\langle \alpha_{i_j}, \alpha_{i_j} \rangle} \exp\left( -\frac{\lambda-\left(c_j+\sum_{k=j+1}^m c_k \alpha_{i_j}(\alpha_{i_k}^\vee)\right)}{h} \right) \right)\\
& \stackrel{h \rightarrow 0}{\rightarrow} \mathds{1}_{\left\{ c_j \leq \lambda-\sum_{k=j+1}^m c_k \alpha_{i_j}(\alpha_{i_k}^\vee)\right\}}
\end{align*}
For the other term:
$$ \exp\left( - \sum_{\alpha \in \Delta}^m \frac{2 h^2}{\langle \alpha, \alpha \rangle} \chi_\alpha \circ \eta^{w_0, e} \circ x_{\bf-i}\left( e^{-h^{-1}c_1}, \dots, e^{-h^{-1}c_m} \right) \right) $$
Start by choosing a reduced word ${\bf i'} \in R(w_0)$, independently of ${\bf i} \in R(w_0)$. This choice will not play any role. Let $f$ be the rational substraction free function given by:
$$ f = x_{\bf i'}^{-1} \circ \eta^{w_0, e} \circ x_{\bf-i}: \R_{>0}^m \rightarrow \R_{>0}^m $$
Component-wise, we write $f = \left( f_1, \dots, f_m \right)$. Then, after organizing that term and using the analytic tropicalization procedure given in proposition \ref{proposition:analytic_tropicalization}:
\begin{align*}
  & \exp\left( - \sum_{\alpha \in \Delta}^m \frac{2 h^2}{\langle \alpha, \alpha \rangle} \chi_\alpha \circ \eta^{w_0, e} \circ x_{\bf-i}\left( e^{-h^{-1}c_1}, \dots, e^{-h^{-1}c_m} \right) \right)\\
& = \exp\left( - \sum_{j=1}^m \frac{2 h^2}{\langle \alpha_{i_j'}, \alpha_{i_j'} \rangle} f_j\left( e^{-h^{-1}c_1}, \dots, e^{-h^{-1}c_m} \right) \right)\\
& = \exp\left( - \sum_{j=1}^m \frac{2 h^2}{\langle \alpha_{i_j'}, \alpha_{i_j'} \rangle} \exp\left( -\frac{-q\log f_j\left( e^{-h^{-1}c_1}, \dots, e^{-h^{-1}c_m} \right)}{h} \right) \right)\\
& = \exp\left( - \sum_{j=1}^m \frac{2 h^2}{\langle \alpha_{i_j'}, \alpha_{i_j'} \rangle} \exp\left( -\frac{[f_j]_{trop}({\bf c}) + \Oc(h)}{h} \right) \right)\\
& \stackrel{h \rightarrow 0}{\rightarrow} \mathds{1}_{\left\{ {\bf c} \in \R^m \ | \ \forall 1 \leq j\leq m, [f_j]_{trop}({\bf c}) \geq 0 \right\}}
\end{align*}
Finally:
$$ \left\{ {\bf c} \in \R^m \ | \ \forall 1 \leq j\leq m, f_j({\bf c}) \geq 0 \right\}
 = \left\{ {\bf c} \in \R^m \ | \ [x_{\bf i'}^{-1} \circ \eta^{w_0, e} \circ x_{\bf-i}]_{trop}({\bf c}) \in \R_+^m \right\}$$
is the set of ${\bf i}$-string parameters that are mapped to non-negative ${\bf i'}$-Lusztig parameters. It has to be exactly the string cone $\Cc_{{\bf i}}^\vee$ thanks to \cite{bib:BZ01}, theorem 5.7. Clearly, changing ${\bf i'}$ tantamounts to changing charts for the Lusztig parameters, and these charts are bijections of the positive orthant $\R_+^m$.
\end{proof}

As a consequence, the geometric Duistermaat-Heckman measure degenerates to the classical one, and its Laplace transform degenerates to the 'asymptotic' Schur functions (see \cite{bib:BBO2} theorem 5.5):
$$ \forall (\lambda, \mu) \in \afrak^2, 
   h_\mu(\lambda) := \frac{\sum_{w \in W} (-1)^{\ell(w)} e^{\langle \mu, w \lambda \rangle} }
                          {\prod_{\beta \in \Phi^+} \langle \beta^\vee, \mu \rangle }$$
\begin{proposition}
For $\lambda \in \afrak$:
$$\lim_{h \rightarrow 0} \psi_{h, \mu}\left( \lambda \right) = h_\mu(\lambda)$$
where:
$$h_\mu\left(\lambda\right) = \int_{\Cc_{\bf i}^\vee} {\bf dc} \varphi({\bf c}) \exp\left( \langle \mu, \lambda - \sum_{k=1}^m c_k \alpha_{i_k}^\vee \rangle \right) \mathds{1}_{\left\{ c_j \leq \lambda-\sum_{k=j+1}^m c_k \alpha_{i_j}(\alpha_{i_k}^\vee)\right\}}$$
Moreover, for $\mu \in C$, $h_\mu$ is a harmonic function on $C$ the Weyl chamber with Dirichlet boundary conditions and growth condition:
$$ \lim_{\lambda \rightarrow \infty, \lambda \in C} h_\mu(\lambda) e^{-\langle \mu, \lambda \rangle} =
   \frac{1}{\prod_{\beta \in \Phi^+} \langle \beta^\vee, \mu \rangle}$$
\end{proposition}
\begin{proof}
The function $\psi_{h, \mu}\left( \lambda \right)$ plays the role of normalization constant in proposition \ref{proposition:q_string_formula}, hence:
$$\psi_{h, \mu}\left(\lambda\right) = \int_{\R^m} {\bf dc} \varphi({\bf c}) \exp\left( \langle \mu, \lambda - \sum_{k=1}^m c_k \alpha_{i_k}^\vee \rangle- f_{B,h,\lambda}^{K, {\bf i}}({\bf c})\right)$$
The previous proposition yields the convergence of $\psi_{h, \mu}$ to $h_\mu$.\\
Now consider $\mu \in C$. In order to see it is a harmonic function on the Weyl chamber with Dirichlet boundary conditions, one can look at equation \eqref{eq:q_deformed_toda} and notice that the potential goes to zero inside the Weyl chamber and $+\infty$ outside.\\
For the growth condition, since $\psi_{h, \mu}(\lambda)e^{-\langle \mu, \lambda \rangle}$ is monotonically increasing for any sequence $\lambda_n \in C, \lambda_n \rightarrow \infty$ along a ray, the convergence to $h^m b(h\mu)$ is uniform in $h$ by Dini's theorem. Therefore, we can obtain the limiting behavior of $h_\mu(\lambda) e^{-\langle \mu, \lambda \rangle}$ by inspecting:
$$ \lim_{h \rightarrow 0} h^m b(h\mu) = \lim_{h \rightarrow 0} \prod_{\beta \in \Phi^+ } h \Gamma(h \langle \beta^\vee, \mu \rangle)$$
Recalling that for all $z>0$, $\lim_{h\rightarrow 0} h \Gamma(hz) = \frac{1}{z}$ finishes the proof.
\end{proof}

Also, now we can recover the following results, already known to \cite{bib:BBO} and \cite{bib:BBO2}, as degenerations of theorems \ref{thm:q_highest_weight_is_markov} and \ref{thm:q_canonical_measure}.
\begin{thm}
For $W^{(\mu)}$ a Brownian motion in $\afrak$ with drift $\mu$, $\Pc_{w_0}( W^{(\mu)})$ is Brownian motion conditioned to stay in the Weyl chamber by a Doob transform. It has the infinitesimal generator:
$$ \half \Delta + \langle \log \nabla h_\mu, \nabla \rangle$$
And for every $\varphi$ bounded measurable function on $\R^m$ and $t>0$:
\begin{align}
\label{eq:zero_canonical_measure}
  & \E\left( \varphi\left( \varrho_{\bf i}^{h=0, K}\left( W^{(\mu)}_u; 0 \leq u \leq t \right) \right) \ | \ \Pc_{w_0}( W^{(\mu)})_t = \lambda \right)\\
= &  \frac{1}{h_\mu\left(\lambda\right)}
     \int_{\Cc_{\bf i}^\vee} {\bf dc} \varphi({\bf c}) \exp\left( \langle \mu, \lambda - \sum_{k=1}^m c_k \alpha_{i_k}^\vee \rangle \right) \mathds{1}_{\{ c_j \leq \lambda-\sum_{k=j+1}^m c_k \alpha_{i_j}(\alpha_{i_k}^\vee)\}}
\end{align}
\end{thm}
\begin{proof}
 Immediate.
\end{proof}

\bibliographystyle{halpha}
\bibliography{Bib_Thesis}

\end{document}